\documentclass[11pt]{amsart}

\usepackage{subfiles}

\usepackage{tabu, fullpage,diagbox, booktabs}
\usepackage{placeins}
\usepackage[margin=1in]{geometry}
\usepackage[dvipsnames]{xcolor}   		             		
\usepackage{graphicx}			
\usepackage{amssymb}
\usepackage{dsfont}
\usepackage{mathrsfs}
\usepackage{amsthm}
\usepackage{amsmath}
\usepackage{stmaryrd}
\usepackage{tikz}
\usepackage{tikz-cd}
\usepackage{accents}
\usepackage{upgreek}
\usepackage{enumerate}
\usepackage{bm}
\usepackage{mathtools}
\usepackage[all]{xy}
\usepackage{caption}

\usepackage{url}
\usepackage{float}
\usepackage{todonotes}
\usepackage{bbm}
\usepackage{easyeqn}

\usepackage[full]{textcomp}
\usepackage[cal=cm]{mathalfa}
\usepackage{xparse}
\usepackage{comment}
\usepackage{cite}

\tikzset{
	commutative diagrams/.cd, 
	arrow style=tikz, 
	diagrams={>=stealth}
}

\newcommand{\elt}{(\mathcal{C},(\mathcal{L},s))}

\newcommand{\moduli}{\mathrm{Orb}_{\Lambda}(\mathcal{A}_r)}
\newcommand{\logmoduli}{\mathrm{Log}_{\Lambda}(\mathcal{A}_r|\mathcal{D}_r)}
\newcommand{\lmoduli}{\mathrm{Orb}_{\Lambda}(\mathcal{A}_{\lambda r})}

\newcommand{\geommoduli}{\mathrm{Orb}_{\Lambda}(X_{D,r})}
\theoremstyle{definition}

\makeatletter
\newcommand{\colim@}[2]{%
	\vtop{\m@th\ialign{##\cr
			\hfil$#1\operator@font colim$\hfil\cr
			\noalign{\nointerlineskip\kern1.5\ex@}#2\cr
			\noalign{\nointerlineskip\kern-\ex@}\cr}}%
}
\newcommand{\colim}{%
	\mathop{\mathpalette\colim@{\rightarrowfill@\textstyle}}\nmlimits@
}
\makeatother

\makeatletter
\def\@tocline#1#2#3#4#5#6#7{\relax
	\ifnum #1>\c@tocdepth 
	\else
	\par \addpenalty\@secpenalty\addvspace{#2}%
	\begingroup \hyphenpenalty\@M
	\@ifempty{#4}{%
		\@tempdima\csname r@tocindent\number#1\endcsname\relax
	}{%
		\@tempdima#4\relax
	}%
	\parindent\z@ \leftskip#3\relax \advance\leftskip\@tempdima\relax
	\rightskip\@pnumwidth plus4em \parfillskip-\@pnumwidth
	#5\leavevmode\hskip-\@tempdima
	\ifcase #1
	\or\or \hskip 1em \or \hskip 2em \else \hskip 3em \fi%
	#6\nobreak\relax
	\dotfill\hbox to\@pnumwidth{\@tocpagenum{#7}}\par
	\nobreak
	\endgroup
	\fi}
\makeatother

\usetikzlibrary{calc}
\usetikzlibrary{fadings}
\usetikzlibrary{decorations.pathmorphing}
\usetikzlibrary{decorations.pathreplacing}

\newcounter{marginnote}
\DeclareFontFamily{U}{wncy}{}
    \DeclareFontShape{U}{wncy}{m}{n}{<->wncyr10}{}
    \DeclareSymbolFont{mcy}{U}{wncy}{m}{n}
    \DeclareMathSymbol{\Sh}{\mathord}{mcy}{"58} 

\DeclareMathAlphabet{\mathpzc}{OT1}{pzc}{m}{it}

\usepackage[backref]{hyperref}
\hypersetup{
	colorlinks   = true,          
	urlcolor     = blue,          
	linkcolor    = purple,          
	citecolor   = blue             
}

\theoremstyle{definition}
\newtheorem{theorem}{Theorem}[section]

\newtheorem{conjecture}[theorem]{Conjecture}
\newtheorem{corollary}[theorem]{Corollary}
\newtheorem{lemma}[theorem]{Lemma}
\newtheorem{proposition}[theorem]{Proposition}
\newtheorem{remark}[theorem]{Remark}

\newtheorem*{runningexample*}{Running example}

\newtheorem*{aside*}{Aside}

\newtheorem{constr}[theorem]{Construction}
\newtheorem{convention}[theorem]{Convention}
\newtheorem{definition}[theorem]{Definition}
\newtheorem{example}[theorem]{Example}

\newtheorem{proposition-definition}[theorem]{Proposition-Definition}
\newtheorem{question}[theorem]{Question}

\newtheorem{slogan}[theorem]{Slogan}
\newtheorem{keypoint}[theorem]{Key point}
\newtheorem{maintheorem}{Theorem}

\newtheorem{mainproposition}[maintheorem]{Proposition}

\newenvironment{construction}    
{%
	\pushQED{\qed}\begin{constr}}
	{\popQED\end{constr}}

\newcommand{\BGm}{\mathbf{B} \mathbb{G}_m}
\newcommand{\BGmr}{\mathbf{B} \mathbb{G}_{m,r}}

\newcommand{\bcd}{\begin{center}\begin{tikzcd}}
	\newcommand{\ecd}{\end{tikzcd}\end{center}}

\newcommand{\C}{\mathbb{C}}

\newcommand{\N}{\mathbb{N}}
\newcommand{\Z}{\mathbb{Z}}

\newcommand{\defelt}{(\tilde{\mathcal{C}},(\tilde{\mathcal{L}},\tilde{s}))}
\newcommand{\defeltgen}{(\tilde{\mathcal{C}}_{\eta},(\tilde{\mathcal{L}}_{\eta},\tilde{s}_{\eta}))}

\usepackage{accents}

\NewDocumentCommand{\compatibilitydatum}{m m m m m m O{} O{} O{}}{
	\begin{equation*} \begin{tikzcd}[ampersand replacement=\&]
	\: \arrow{r} \& {#1} \arrow{r} \arrow{d}{#7} \& {#2} \arrow{r} \arrow{d}{#8} \& {#3} \arrow{r}{[1]} \arrow{d}{#9} \& \: \\
	\: \arrow{r} \& {#4} \arrow{r} \& {#5} \arrow{r} \& {#6} \arrow{r} \& \:
	\end{tikzcd} \end{equation*}}

\NewDocumentCommand{\commutingsquare}{m m m m o O{} O{} O{} O{}}{
	\begin{equation}\begin{tikzcd}[ampersand replacement=\&] \label{#5}
	#1 \arrow{r}{#6} \arrow{d}{#7} \& #2 \arrow{d}{#8} \\
	#3 \arrow{r}{#9} \& #4
	\end{tikzcd}\IfValueTF{#5}{\label{#5}}{} \end{equation}}

\NewDocumentCommand{\cartesiansquare}{m m m m O{} O{} O{} O{}}{
	\begin{equation*}\begin{tikzcd}[ampersand replacement=\&]
	#1 \arrow{r}{#5} \arrow{d}{#6} \arrow[dr, phantom, "\square"] \& #2 \arrow{d}{#7} \\
	#3 \arrow{r}{#8} \& #4
	\end{tikzcd} \end{equation*}}

\NewDocumentCommand{\cartesiansquarelabel}{m m m m m O{} O{} O{} O{}}{
	\begin{tikzcd}[ampersand replacement=\&]
	#1 \arrow{r}{#6} \arrow{d}{#7} \arrow[dr, phantom, "\square"] \& #2 \arrow{d}{#8} \\
	#3 \arrow{r}{#9} \& #4
	\end{tikzcd}\IfValueTF{#5}{\label{#5}}{}
}

\NewDocumentCommand{\triangleofspaces}{m m m O{} O{} O{}}{
	\begin{tikzcd} [ampersand replacement=\&]
	#1 \arrow{r}{#4} \arrow[bend right]{rr}{#5} \& #2 \arrow{r}{#6} \& #3
	\end{tikzcd}}

\begin{document}
	\title{Moduli spaces of twisted maps to smooth pairs}
	\author{Robert Crumplin}
	\maketitle

 \begin{abstract}
     We study moduli spaces of twisted maps to a smooth pair in arbitrary genus, and give geometric explanations for the known comparisons between orbifold and logarithmic Gromov--Witten invariants. We classify the irreducible components of the space of twisted maps to the universal target in terms of tropical data, and show that each is built from torsors under semi-abelian varieties coming from Jacobians of the domain curves. Natural morphisms between these spaces for different rooting parameters is in terms of multiplication in the Jacobian fibres: for rooting parameters sufficiently large and divisible, the growth is monomial, with exponent between $0$ and $\max(0,2g-1)$, and the main component is the only one with exponent $0$. In genus $1$ the virtual class is the fundamental class, which recovers the polynomiality results of Tseng and You. We conclude with geometric examples, including Maulik's original genus $1$ example.
 \end{abstract}

	\tableofcontents

\newpage

\newpage

\addtocontents{toc}{\protect\setcounter{tocdepth}{1}}

\setcounter{tocdepth}{0}
\tableofcontents
\newpage

\section{Introduction} 

\subsection{Introduction}\label{intro section}

The subject of this paper is motivated by the study of the enumerative geometry of maps $$(C| p_1 + \cdots + p_n) \rightarrow (X|D) $$ to a smooth pair $(X|D)$, with prescribed non-negative tangency orders at the markings $p_i$. As always in enumerative geometry, one needs a proper moduli space in order to define invariants via integration. There are two well-studied routes to study this problem: \begin{itemize}
    \item \textbf{Logarithmic geometry.} One encodes the tangency conditions via additional monoid structure on $C$ and $X$  \cite{GS11, Che14, AC14}.

    \item \textbf{Orbifold geometry.} One replaces the target with the root stack $X_{D,r}$ and tangencies are encoded via group homomorphisms between isotropy groups $\mu_{s_i} \rightarrow \mu_r$; the domains are now twisted curves \cite{AGV01,AV02,Cad07}.
\end{itemize}

 Two of the key players throughout will be the genus $g$ of the domain curves, and the \textit{rooting parameter} $r$ which we are free to fix. Work of \cite{ACW17} shows that in genus $g = 0$, $D$ a smooth divisor and rooting parameter $r$ sufficiently large, the logarithmic and orbifold Gromov--Witten invariants agree on the nose. This is generalised to the higher rank setting in \cite{BNR22} which shows the same remains true, after making an initial log blowup of the target. The negative contact order setting is studied in \cite{BNR24}. A common technique throughout all of these papers is to study a ``universal moduli space". This is a space controlling the virtual geometry of maps to any smooth pair, whose objects are described in terms of line bundle-section pairs on curves. One key benefit is that one can often convert virtual statements about a space of maps to $X$ into non-virtual statements on the universal space. Indeed, the universal spaces in \cite{ACW17, BNR24} can often be forced to be irreducible, and then one proceeds by studying their birational geometry. In this paper, the universal space of interest will be denoted $\moduli$ which controls twisted maps to a smooth pair.

 In higher genus, a major obstruction to proceeding in this manner is due to our lack of understanding of the geometry of $\moduli$. This motivates the following question:

 \begin{question} \label{question cpts}
     What are the irreducible components of $\moduli$?
 \end{question}

Furthermore, in higher genus and for a smooth divisor, \cite{TY2, TY3} showed that the orbifold Gromov--Witten invariants are polynomial in $r$ of degree at most $\max(0,2g - 1)$ and the constant term recovers the associated logarithmic invariant. Their proof follows from an involved degeneration and localisation calculation for a fixed smooth pair $(X|D)$, without reference to the universal moduli space. 

\begin{question}\label{Q1}
    Is there a geometric explanation for the polynomiality, without the need to degenerate, at the level of moduli spaces?
\end{question}

\begin{question}\label{Q2}
    Can we give a geometric interpretation of the coefficients of the polynomials of orbifold invariants?
\end{question}

The rest of this paper gives answers to questions \ref{question cpts}, \ref{Q1} \& \ref{Q2} by formalising the following idea: 

\begin{slogan}
    We have a complete understanding of the irreducible components of $\moduli$ which are \begin{enumerate}
        \item Indexed by combinatorial/tropical style data (Theorem \ref{main irred cpts}). 
        \item Each component is naturally built from gluing together semi-abelian torsors over boundary strata of $\logmoduli$, the space of universal log maps. These semi-abelian torsors are related to universal Jacobians of the domain curves (Proposition \ref{main action}).

        \item Each irreducible component (and more generally each stratum indexed by tropical data) grows \textit{monomially} in the rooting parameter $r$, where the monomial is in terms of the $r$-torsion of the associated semi-abelian variety (Theorem \ref{main degree}).

        \item The monomials are bounded above by the size of the $r$-torsion of the Jacobian of a smooth genus $g$ curve, i.e. by $r^{2g}$, along with a discrepancy related to automorphisms of the curves, bringing the bound down to $r^{2g-1}$. Also, there is a distinguished component $Z_{\mathrm{main}} \subset \moduli$ which is the unique component contributing monomial $r^0 = 1$. This component is birational to $\logmoduli$ (Theorem \ref{main theorem polynomial}).
    \end{enumerate}
\end{slogan}

\subsection{Main results }

In section \ref{setup} we assign, to each object $\elt \in \moduli$, combinatorial data called a mod $r$ tropical type (Definition \ref{orbitroptype}), denoted $[\tau]$. This is roughly the data of a usual tropical type, but the slopes and balancing condition only live in $\Z/r\Z$. In particular $[\tau]$ records a dual graph $\Gamma$ to $\mathcal{C}$. This endows $\moduli$ with a stratification into closed subspaces $$Z_{[\tau]} \subset \moduli.$$ There is a distinguished class of mod $r$ tropical types we call \textit{essential types} (Definition \ref{essentialtype}). The classification of irreducible components using this stratification is studied in Section \ref{classification}.

\begin{maintheorem}[Theorem \ref{irredcpts}]\label{main irred cpts}
     There is a bijection $$\mathrm{Irred. \ cpts \ of \ } \moduli \leftrightarrow \mathrm{Inducible \ essential \ tropical \ types \ } [\tau] $$ given by $$[\tau] \mapsto Z_{[\tau]}.$$
\end{maintheorem} \qed

 A key input to the above Theorem is the construction of a semi-abelian group action on each stratum $Z_{[\tau]}^{\circ}$, which acts freely, and transitively on the locus of objects with a fixed domain curve. This allows us to reduce many computations on the moduli space to two easier computations: one on semi-abelian varieties, and one on boundary strata of $\overline{\mathfrak{M}}_{g,n}$. The proof at its core is a deformation theory argument for line bundle-section pairs on twisted curves.

\begin{mainproposition}[= Proposition \ref{action}]\label{main action}
    For $[\tau]$ any mod $r$ tropical type, there is a family of semi-abelian varieties $\mathrm{Jac}_{[\tau]}$ and a free action $$\mathrm{Jac}_{[\tau]} \times [(\mathbb{G}_m)^{b_0(\Gamma_0)}/ \mathbb{G}_m] \curvearrowright Z_{[\tau]}^{\circ} $$ which acts transitively on the sub-locus of data with a fixed curve $\mathcal{C}$. Here $\Gamma_0$ is the subgraph of the dual graph spanned by the components on which the section does not vanish identically.
\end{mainproposition} \qed

\begin{figure}[h!]
    \centering
    \tikzset{every picture/.style={line width=0.75pt}} 

\begin{tikzpicture}[x=0.75pt,y=0.75pt,yscale=-1,xscale=1]

\draw  [fill={rgb, 255:red, 126; green, 211; blue, 33 }  ,fill opacity=1 ] (193,159.98) .. controls (193,132.38) and (245.61,110) .. (310.5,110) .. controls (375.39,110) and (428,132.38) .. (428,159.98) .. controls (428,187.59) and (375.39,209.97) .. (310.5,209.97) .. controls (245.61,209.97) and (193,187.59) .. (193,159.98) -- cycle ;
\draw  [fill={rgb, 255:red, 248; green, 231; blue, 28 }  ,fill opacity=1 ] (301.45,35.97) -- (397,35.97) -- (356.05,143) -- (260.5,143) -- cycle ;
\draw    (301,172.97) ;
\draw [shift={(301,172.97)}, rotate = 0] [color={rgb, 255:red, 0; green, 0; blue, 0 }  ][fill={rgb, 255:red, 0; green, 0; blue, 0 }  ][line width=0.75]      (0, 0) circle [x radius= 3.35, y radius= 3.35]   ;
\draw    (161,180.97) .. controls (200.6,151.27) and (250,205.89) .. (289.8,177.87) ;
\draw [shift={(291,177)}, rotate = 143.13] [color={rgb, 255:red, 0; green, 0; blue, 0 }  ][line width=0.75]    (10.93,-3.29) .. controls (6.95,-1.4) and (3.31,-0.3) .. (0,0) .. controls (3.31,0.3) and (6.95,1.4) .. (10.93,3.29)   ;
\draw    (320,141.97) ;
\draw [shift={(320,141.97)}, rotate = 0] [color={rgb, 255:red, 0; green, 0; blue, 0 }  ][fill={rgb, 255:red, 0; green, 0; blue, 0 }  ][line width=0.75]      (0, 0) circle [x radius= 3.35, y radius= 3.35]   ;
\draw    (503,138.97) .. controls (430.73,110.26) and (360.42,176.62) .. (321.18,143.02) ;
\draw [shift={(320,141.97)}, rotate = 42.71] [color={rgb, 255:red, 0; green, 0; blue, 0 }  ][line width=0.75]    (10.93,-3.29) .. controls (6.95,-1.4) and (3.31,-0.3) .. (0,0) .. controls (3.31,0.3) and (6.95,1.4) .. (10.93,3.29)   ;
\draw    (339,80.97) ;
\draw [shift={(339,80.97)}, rotate = 0] [color={rgb, 255:red, 0; green, 0; blue, 0 }  ][fill={rgb, 255:red, 0; green, 0; blue, 0 }  ][line width=0.75]      (0, 0) circle [x radius= 3.35, y radius= 3.35]   ;
\draw    (211,67.97) .. controls (250.6,38.27) and (298.04,109.52) .. (337.8,81.84) ;
\draw [shift={(339,80.97)}, rotate = 143.13] [color={rgb, 255:red, 0; green, 0; blue, 0 }  ][line width=0.75]    (10.93,-3.29) .. controls (6.95,-1.4) and (3.31,-0.3) .. (0,0) .. controls (3.31,0.3) and (6.95,1.4) .. (10.93,3.29)   ;

\draw (57,180.4) node [anchor=north west][inner sep=0.75pt]    {$\left( E,\left(\mathcal{O}\left(\sum _{i\ =\ 1}^{n}\tilde{c}_{i} p_{i}\right) ,s\right)\right) \ $};
\draw (448,141.4) node [anchor=north west][inner sep=0.75pt]    {$\left( E,\left(\mathcal{O}\left(\sum _{i\ =\ 1}^{n}\tilde{c}_{i} p_{i}\right) ,0\right)\right) \ $};
\draw (32,35.4) node [anchor=north west][inner sep=0.75pt]    {$\left( E,\left(\mathcal{O}\left(\sum _{i\ =\ 1}^{n}\tilde{c}_{i} p_{i}\right) \otimes \mathcal{J} ,0\right)\right) \ $};
\draw (217,141.4) node [anchor=north west][inner sep=0.75pt]    {$Z\mathrm{_{\mathrm{m} ain}^{\circ }}$};
\draw (342,42.4) node [anchor=north west][inner sep=0.75pt]    {$Z\mathrm{_{[ \tau ]}^{\circ }}$};

\end{tikzpicture}
    \caption{Two irreducible components of $\moduli$ in genus $1$.}
    \label{fig:genus1components}
\end{figure}

In Section \ref{polynomiality} we investigate the polynomial properties of the strata $Z_{[\tau]}$ as we vary the rooting parameter. For a fixed $r$, there are natural comparison maps between the universal moduli spaces (Section \ref{comparison maps}) $$\pi_{\lambda r, r}: \mathrm{Orb}_{\Lambda}(\mathcal{A}_{\lambda r}) \rightarrow \moduli $$ for any $\lambda \in \N$. We first compute the degree of the comparison morphisms, showing the result is a combinatorial factor multiplied by a monomial whose exponent depends explicitly on tropical data:

\begin{maintheorem}[= Theorem \ref{degrees} + Theorem \ref{essential degree theorem}]\label{main degree}
    Let $[\tau]$ be a mod $\lambda r$ tropical type. Then under the comparison map $\pi_{\lambda r,r}$, $Z_{[\tau]}$ maps onto a stratum $Z_{[\tau']}$ for $[\tau']$ a mod $r$ tropical type. Furthermore, the map $Z_{[\tau]} \rightarrow Z_{[\tau']}$ is representable by DM stacks, with $0$ dimensional fibres of finite degree $$ C(\lambda)  \cdot \lambda^{|E^{\mathrm{b}}| - b_0(\Gamma_+) + b_1(\Gamma_+) + 2\sum_{v \in V_+} g_v}, \quad C(\lambda) = \prod_{e \in E} \frac{t_e(r)}{t_e(\lambda r)}, $$ where $C(\lambda)$ is the inverse of the number of ghost automorphisms of the twisted curves killed by the partial coarsening ($t_e$ denoting the order of the isotropy at the node $e$), and the exponent of the monomial term is in terms of torsion in the semi-abelian variety of Proposition \ref{main action}.

    Furthermore, if we additionally suppose \begin{itemize}
     \item \textbf{$r$ is sufficiently large:} $r > \max_{v \in V_0, v \in e} (d_v - m_{\vec{e}} + \sum_{v \in i \in L(\Gamma)} c_i )$.
        \item \textbf{$r$ is sufficiently divisible:} $m_{\vec{e}} | r $ for all $e \in E$.
        \item \textbf{Nodal contacts are non-trivial:} Suppose that $m_{\vec{e}} \not = 0$ for all $e \in E$.
        \item $[\tau]$ is an essential type
    \end{itemize} then $C(\lambda) = \lambda^{-|E|}$ is a monomial in $\lambda$. Thus the total degree is the monomial $\lambda^{2\sum_{v \in V_+} g_v - |V_+|}$, and in fact this exponent lies between $0$ and $\max(0,2g-1)$.
\end{maintheorem} \qed

 The first two constraints have arisen in related work, namely the main results of \cite{ACW17} require these conditions. Furthermore, this fits in intuitively with the notion of ``sensitive subdivision" of \cite{BNR22}; the auxiliary choice of $r$ gives an initial log modification of the target after which the invariants are well-behaved. Section \ref{polynomiality} packages the above results into a concise statement about Chow-valued polynomials. Before stating the main Theorem of the section, we need to introduce an appropriate notion of compatible families of strata as we vary the rooting parameter. This is done via $\widehat{\Z}$-tropical types (Definition \ref{Zhattype}). A single $\widehat{\Z}$-tropical type $\hat{\tau}$ induces a family of strata $Z_{[\hat{\tau}_r]} \subset \moduli$ for each $r$ which map to each other under the comparison morphisms.

\begin{maintheorem}[= Theorem \ref{essential degree theorem} + Corollary \ref{essential polynomial}]\label{main theorem polynomial} 
    
    Let $\hat{I}$ be a finite set of essential $\widehat{\Z}$-tropical types. Define a family of Chow classes $$P_{\hat{I},r}(\lambda) := (\pi_{\lambda r, r})_* (\sum_{j \in \hat{I}} [Z_{[\hat{\tau}_{j,\lambda r}]}]) \in A_*(\moduli).$$ 
    
    Then under the numerical assumptions on $r$ above, $P_{\hat{I},r}(\lambda)$ is a polynomial in $\lambda$ of degree at most $\max(0,2g-1)$ and whose constant term is given by the contributions from the main component. The top degree terms are given by irreducible components whose associated type has an internal genus $g$ vertex.
\end{maintheorem} \qed

This is reminiscent of the main results of \cite{TY2,TY3} on the polynomiality of orbifold invariants. To upgrade Theorem \ref{main theorem polynomial} into a result about the Gromov--Witten invariants of smooth pairs, we need to express the virtual geometry of $\moduli$ in terms of the strata. Indeed, in Section \ref{dimension calculations} we study the dimensions of strata of the moduli space and deduce the following trichotomy: \begin{itemize}
    \item $g = 0:$ $\moduli = Z_{\mathrm{main}}$ is irreducible of expected dimension.
    \item $g = 1:$ $\moduli$ is reducible, but equidimensional of expected dimension.
    \item $g > 1:$ $\moduli$ is reducible and not equidimensional, with components larger than expected dimension.
\end{itemize} As a result we may have non-trivial virtual theory for $g > 1$. As a test case, we prove the virtual class in $g = 1$ is the usual fundamental class, thus recovering results of Tseng--You in this instance. 

\begin{maintheorem}[= Corollary \ref{g = 1 virtual class} + Theorem \ref{g = 1 virtual polynomial}]

In genus $g = 1$, the virtual class is $$[\moduli]^{\mathrm{vir}} = \sum_{[\tau] \ \mathrm{essential}} [Z_{[\tau]}]$$ which is the usual fundamental class on each finite type open substack. Hence, for genus $g = 1$ and $r$ sufficiently large and divisible, the family of Chow classes $$Q(\lambda) := (\pi_{\lambda r, r})_{*}[\lmoduli]^{\mathrm{vir}} \in A_*(\moduli)$$ is a polynomial in $\lambda$ of degree $1 = 2g - 1$ whose constant term is given by the contribution from $Z_{\mathrm{main}}$.

\end{maintheorem} \qed

Section \ref{geometric virtual classes} explains how to upgrade these results about the universal moduli space to ones about the moduli space of maps to a genuine smooth pair $(X|D)$, by relating the universal and geometric moduli spaces via basechange.

We conclude the paper with some down to earth geometric examples (Section \ref{examples}), situating familiar geometries in our new framework. This includes the original Maulik example of $X = E \times \mathbb{P}^1$ relative to two elliptic fibres, and a variant with marked points and insertions.

\subsection{History of comparison results and future directions.}\label{history}

As mentioned above, there is much previous work on comparison results. We detail other relevant results here: \begin{itemize}

    \item In genus $g = 1$ and rank $1$, the orbifold invariants are of degree at most $1$. The top coefficient is shown to be a \textit{genus $0$ mid-age invariant} \cite{TY3, You21}.

    \item Stationary invariants of $(X|D)$ for $X$ a curve are shown to be constant for all genera and equal to logarithmic stationary invariants \cite{UKW21, TY2}.

    \item \cite[Conjecture W]{BNR22} conjectures the constant coefficient in the higher rank, higher genus setting is the \textit{naive log invariant} (\cite[Section 2.2]{BNR22}). This is roughly an integral over a fibre product of the logarithmic spaces $\mathrm{Log}_{\Lambda}(X|D_i)$ for $i = 1, \cdots, k$. 

    \item A ``reconstruction theorem" is a statement about how absolute Gromov--Witten invariants of $X$ and the strata of $D$ should determine the orbifold Gromov--Witten invariants of $X_{D,\vec{r}}$. Statements of this type are proved in \cite{TY1} for smooth $D$ arbitrary genus and \cite{BNR22} for higher rank genus $0$. Related to this, \cite[Theorem 1.6]{TY3} tells us that the coefficients of the polynomial of orbifold invariants for a smooth pair $(X|D)$ are expressible in terms of lower genus invariants of $(X|D)$ along with the absolute invariants of $D$.

    \item Recent work of Chiodo and Holmes \cite{CH24} provides an interpretation of the \cite{JPPZ17} procedure in terms of isolating a specific component of a space embedded in a universal Jacobian over the stack of $r$-spin curves. This study is analogous to finding the main component $Z_{\mathrm{main}}$ as in section \ref{classification} and also identifying it as the constant term of a polynomial, akin to Theorem \ref{main theorem polynomial}. We expect the work here will provide a route to classify all non-main components in their setting, and likewise their arguments will inspire many of the proofs in upcoming work \cite{CJ}.
\end{itemize} 

We expect the techniques of this paper will provide a route to re-derive and generalise the above results and verify Conjecture W. Indeed, we will see an intuitive explanation for each monomial term of the polynomial in terms of a stratum of the universal moduli space, and this is pursued further in upcoming work with Sam Johnston. In this upcoming work we study the virtual class in more generality. The key idea is to construct an embedding of $$\moduli \hookrightarrow \mathcal{P}ic$$ where $\mathcal{P}ic$ denotes the universal Picard stack over the space of \textit{twisted log maps}. $\mathcal{P}ic$ is log smooth and then analysis of the Segre class of this embedding gives the virtual class as an expression in terms of a sum of intersections of the strata $Z_{[\tau]}$.

\subsection{Acknowledgements}

I am extremely grateful for many insightful discussions with Sam Johnston who taught me much of what I know about logarithmic and twisted maps. I am grateful to Mark Gross and Luca Battistella for comments and corrections. The same can be said for Navid Nabijou who has helped me through the whole project, especially in understanding stack subtleties. Dhruv Ranganathan suggested the original problem to me at Abramorama, at which we had many insightful discussions, and I am thankful to Jonathan Wise for renewing my interest in the problem. I would also like to thank Renzo Cavalieri for his many wonderful questions about this project during my visit to CSU, and also Longting Wu for his many helpful discussions during my visit to SUSTech. Fenglong You gave very helpful insights throughout, especially on his related polynomiality and mid-age work. This work was supported by the Engineering and Physical Sciences Research Council [EP/S021590/1]. The EPSRC Centre for Doctoral Training in
Geometry and Number Theory (The London School of Geometry and Number Theory), University College London.

\section{Irreducible components of the universal space of twisted maps }\label{irreducible components}

\subsection{Set up}\label{setup}

\subsubsection{Contact data and the moduli space}\label{contactdata}

Throughout the remainder of this paper we denote by $\Lambda$ the data of \begin{itemize}
    \item \textbf{Source rooting parameters} $s_1, \cdots, s_n \in \N$.
    \item \textbf{Ages} $a_1, \cdots, a_n \in \mathbb{Q}$ such that $a_i \in \frac{1}{s_i}\Z \cap [0,1)$.
    \item \textbf{Coarse degree} $d \in \Z$.
    \item \textbf{Genus} $g \in \N$.
\end{itemize} Such data $\Lambda$ we refer to as \textbf{contact data}. 
Additionally, throughout we will pick a \textbf{target rooting parameter} $r \in \N$. From the above data we can construct \textbf{coarse contact orders} \begin{equation}\label{coursecontacts}
    c_i := r a_i.
\end{equation} From now on, we assume the following numerical conditions hold: 

\begin{enumerate}
    \item \textbf{Divisibility.} For each marking $i$ we have $$s_i | r $$ and $$r| c_i s_i. $$
    \item \textbf{Coprimality.} For each marking $i$ we have $$r/\mathrm{gcd}(r,c_i) = s_i.$$
    \item \textbf{Size.} We have $$r > \mathrm{max}_{i} c_i.$$
\end{enumerate}

\begin{remark}
    Note the above conditions are not logically independent. For example, condition $(3)$ is implied by condition \eqref{coursecontacts} along with the fact that $0 \leq a_i < 1$ for each $i$. However these are all important properties, so we list them separately, and this also provides a point of comparison with \cite[Section 2.1]{BNR24}.
\end{remark}

Given contact data $\Lambda$ satisfying the above conditions we may form a moduli space $$\mathrm{Orb}_{\Lambda}(\mathcal{A}_r) $$ of \textbf{representable twisted pre-stable maps to} $\mathcal{A}_r$ \textbf{with contact data} $\Lambda$. Here $\mathcal{A}_r$ denotes the $r$th root stack of $\mathcal{A} := [\mathbb{A}^1/\mathbb{G}_m]$ along $\BGm \subset \mathcal{A}$. More precisely, $\mathrm{Orb}_{\Lambda}(\mathcal{A}_r)$ is the stack over schemes whose objects are the data of $(\mathcal{C},(\mathcal{L},s))$ where:

\begin{itemize}
    \item $\mathcal{C}$ is an $n$-marked pre-stable twisted curve of genus $g$, with markings $p_i = B \mu_{s_i}$.
    \item $(\mathcal{L},s)$ is a line bundle-section pair on $\mathcal{C}$ such that, if $\pi: \mathcal{C} \rightarrow C$ is the coarse space, the $r$th tensor power $(\mathcal{L},s)^{\otimes r}$ is pulled back from $C$ via $\pi$.
    \item At marking $p_i \in \mathcal{C}$, the age of the representation is $$\mathrm{age}_{p_i} \mathcal{L} = a_i. $$
    \item The degree of $\mathcal{L}$ is $$\tilde{d} = \mathrm{deg} \mathcal{L} = d/r \in \frac{1}{r} \Z. $$ We refer to $\tilde{d}$ as the \textbf{gerby degree}. We assume that $$\tilde{d} = \sum_{i = 1}^n a_i$$ always holds.

    \item Additionally, we enforce the following conditions on degrees: For any subcurve $\mathcal{C}' \subseteq \mathcal{C}$, $$-\frac{1}{2} < \mathrm{deg} (\mathcal{L}|_{\mathcal{C}'}) < \frac{1}{2}. $$ In particular, we need $$r > 2d.$$
\end{itemize}

We may define the \textbf{gerby contact order} at marking $p_i$ to be $$\tilde{c}_i = \frac{s_ic_i}{r} = s_i a_i \in \N. $$ The moduli space $\mathrm{Orb}_{\Lambda}(\mathcal{A}_r)$ is an Artin stack which we will see is connected and, for $g >0$, of infinite type. This moduli stack was first introduced in \cite{ACW17} in genus $0$, and studied further in \cite{BNR22}, \cite{BNR24}.

\subsubsection{Tropical types}\label{troptype}

We recall the definition of the type of a twisted tropical map as introduced in \cite[Definition 3.6]{BNR24}, and make a few different variants of it to suit our higher genus orbifold setting.

\begin{definition}\label{pretroptype}
    We define a \textbf{twisted pre-tropical type} with contact data $\Lambda$ and parameter $r$ to be the data of \begin{enumerate}
        \item \textbf{Source graph.} An $n$-marked graph $\Gamma$ with vertex set $V(\Gamma)$, bounded edges $E(\Gamma)$ and legs $L(\Gamma)$.
        \item \textbf{Genus labelling.} A genus labelling for each $v \in V(\Gamma)$,  $g_v \in \N$, such that $$\sum_{v \in V(\Gamma)} g_v + b_1(\Gamma) = g.$$
        \item  \textbf{Degrees.} A labelling for each $v \in V(\Gamma)$, $d_v \in \Z$, such that $$\sum_{v \in \ V(\Gamma)} d_v = d.$$ Additionally, we enforce that for every subset $S \subset V(\Gamma)$ we have $$-\frac{1}{2} < \frac{1}{r}\sum_{v \in S} d_v < \frac{1}{2}. $$
    
        \item \textbf{Image cones.} A cone of $\mathbb{R}_{\geq 0}$ associated to every vertex and edge of $\Gamma$, $$v \leadsto \sigma_v, e \leadsto \sigma_e$$ such that if $v \in e$ then $\sigma_v \subset \sigma_e.$
    
    \end{enumerate}

We often denote this data by $\underline{\tau}$.
\end{definition}

When we have a line bundle on a twisted curve, the age of the line bundle is not well-defined at nodal points due to the actions on either branch being inverse to each other. The following definition will make it easier for us to record which branch we are taking and will appear throughout the rest of the paper.

\begin{definition}\label{alphas}
    Let $\mathcal{C}$ be a pre-stable twisted curve with node $q \cong B \mu_{t_q}$. Let $\nu : \tilde{\mathcal{C}} \rightarrow \mathcal{C}$ be the normalisation with $\nu^{-1}(q) = \{q_1,q_2\}$. Let $\mathcal{L}$ be a line bundle on $\mathcal{C}$. Define $$\alpha_{q,i}(\mathcal{L}) := t_q \cdot \mathrm{age}_{q_i}(\nu^{*}\mathcal{L}).$$ When $q$ arises as a node from the intersection of two distinct irreducible components $\mathcal{C}_1, \mathcal{C}_2 \subset \mathcal{C}$, the choice of $q_i$ is equivalent to the choice of component $\mathcal{C}_i$ and therefore we may write $$\alpha_{q,\mathcal{C}_i}(\mathcal{L}) := t_q \cdot \mathrm{age}_q (\mathcal{L}|_{\mathcal{C}_i}) $$ which also agrees with $\alpha_{q,i}(\mathcal{L})$. 
\end{definition}

Observe that $$\alpha_{q,1}(\mathcal{L}) + \alpha_{q,2}(\mathcal{L}) = t_q$$ if $\mathrm{age}_{q}(\mathcal{L}) \not = 0$, else $$\alpha_{q,1}(\mathcal{L}) + \alpha_{q,2}(\mathcal{L}) = \alpha_{q,i}(\mathcal{L}) = 0 .$$ In either case we have \begin{equation}\label{congruence}
    \alpha_{q,1}(\mathcal{L}) + \alpha_{q,2}(\mathcal{L}) \equiv 0 \mod t_q.
\end{equation}

\begin{definition}\label{twisted trop type} A \textbf{twisted tropical type} $\tau$ is the data of a twisted pre-tropical type $\underline{\tau}$ along with the extra data of:


(5) \textbf{Rooting parameters.} An assignment of a rooting parameter to every edge $e \in E(\Gamma), t_e \in \N.$

(6) \textbf{Gerby slopes. } For each oriented edge $\vec{e}$ a gerby slope $\tilde{m}_{\vec{e}} = - \tilde{m}_{\overleftarrow{e}} \in \Z$. At each vertex $v \in V(\Gamma)$ these slopes must also satisfy the \textbf{balancing condition}:

        $$ d_v = \sum_{v \in e} (\frac{r}{t_e}) \tilde{m}_{\vec{e}} + \sum_{v \in l_i \in L(\Gamma)} c_i. $$

        Furthermore, the numerical assumptions of Section \ref{contactdata} must also hold, for both the marked leg data and the gerby edge parameters and gerby slopes. More precisely, if we denote by $$m_{\vec{e}} := (\frac{r}{t_e}) \tilde{m}_{\vec{e}} $$ the \textbf{coarse slope at} $\vec{e}$, we require $$\frac{r}{\gcd(r,m_{\vec{e}})} = t_e.$$

\end{definition}

\begin{definition}\label{orbitroptype}

A \textbf{mod $r$ tropical type} $[\tau]$ is the data of a twisted pre-tropical type $\underline{\tau}$ along with the extra data of:


(5) \textbf{Rooting parameters.} An assignment of a rooting parameter to every edge $e \in E(\Gamma), t_e \in \N.$ 

(6') \textbf{ Gerby slopes mod $r$. } For each oriented edge $\vec{e}$ a gerby slope $[\tilde{m}_{\vec{e}}] = - [\tilde{m}_{\overleftarrow{e}}] \in \Z/t_e\Z$. At each vertex $v \in V(\Gamma)$ these slopes must also satisfy the \textbf{balancing condition modulo $r$}:

        $$ d_v \equiv \sum_{v \in e} (\frac{r}{t_e}) [\tilde{m}_{\vec{e}}] + \sum_{v \in l_i \in L(\Gamma)} c_i \mod r. $$

   Furthermore, the numerical assumptions of Section \ref{contactdata} must also hold, for both the marked leg data and the gerby edge parameters and gerby slopes. More precisely, let $m_{\vec{e}} \in \Z$ be any lift of $[m_{\vec{e}}] := (\frac{r}{t_e}) [\tilde{m}_{\vec{e}}] \in \Z/r\Z$. We then require \begin{equation}\label{noderooting}
       \frac{r}{\gcd(r,m_{\vec{e}})} = t_e.
   \end{equation} Note that this is well-defined since $\gcd(r,m_{\vec{e}})$ is independent of lift. 

\end{definition}

It is then clear that a twisted tropical type $\tau$ induces a mod $r$ tropical type $[\tau]$ by taking the $[m_{\vec{e}}]$ to be the reductions of the $m_{\vec{e}}$ modulo $r$, and in general we shall use $[ \_ ]$ to denote the reduction modulo $r$.

\begin{remark}
    We often denote a tropical type in any of the senses of Definitions \ref{pretroptype}, \ref{twisted trop type}, \ref{orbitroptype} by drawing a graph $\Gamma$ mapping to the ray $\mathbb{R}_{\geq 0}$. Such data gives us the assignment of cones, however it need not be the case that any type is actually realised via such a continuous map $\Gamma \rightarrow \mathbb{R}_{\geq 0}$. This is what is referred to as \textbf{realisability} in \cite[Definition 2.44(2)]{ACGS20}.
\end{remark}

\begin{definition}\label{r-weighting}
    For $\underline{\tau}$ a twisted pre-tropical type, denote by $$W_{\underline{\tau},r} $$ the set of admissible gerby slope assignments as above. An element of $W_{\underline{\tau},r}$ we call an \textbf{$r$-weighting}. Equivalently, $W_{\underline{\tau},r}$ may be viewed as the set of mod $r$ tropical types with fixed underlying pre-tropical type $\underline{\tau}$.
\end{definition}

The notion of an $r$-weighting is a generalisation of the notion of an $r$-weighting in \cite[Section 0.4.1]{JPPZ17}.

\begin{lemma}\label{r-weightings}
    The set of $r$-weightings may be identified with $$W_{\underline{\tau},r} = (\Z/r\Z)^{b_1(\Gamma)}. $$

\end{lemma}

\begin{proof}
    Pick a maximal tree $T \subset \Gamma$. Fix any values for $[m_{\vec{e}}] \in \Z/r\Z$ for $e \in E(\Gamma) \setminus E(T)$. We may write the mod $r$ balancing equation at $v \in V$ as $$d_v' \equiv \sum_{v \in e \in E(T)} [m_{\vec{e}}] + \sum_{v \in l_i \in L(\Gamma)}c_i \mod r $$ where $d_v'$ is $d_v$ after combining with any terms $[m_{\vec{e}}]$ for $e \not \in E(T)$. Since $T$ is genus $0$, the proof of \cite[Lemma 3.4]{BNR22} shows that, the resulting equations have a unique solution for the remaining $[m_{\vec{e}}]$. Indeed, the Lemma is for the analogous genus $0$ situation by working with equations over $\Z$. Hence the $[m_{\vec{e}}]$ for $e \in E(T)$ are uniquely determined. We were free to pick the $[m_{\vec{e}}]$ for $e \in E(\Gamma) \setminus E(T)$ as we wish, and $|E(\Gamma) \setminus E(T)| = b_1(\Gamma)$. Any choice of the $[m_{\vec{e}}]$ uniquely determines the rooting parameters $t_e$ via equation \eqref{noderooting}, and the result follows.
\end{proof}

\begin{construction}\label{assocmodrtype}
    Given $\elt \in \moduli$ we may associate a mod $r$ tropical type $[\tau]$ as follows: \begin{itemize}
    \item \textbf{Conditions (1), (2) \& (5)} $\Gamma$ is the dual intersection complex of $\mathcal{C}$ with genus labelling of the vertices $g_v \in \N$ induced by the genus of the associated component $\mathcal{C}_v \subset \mathcal{C}.$ The rooting parameters are the orders of the isotropy groups at the associated special points.

    \item \textbf{Condition (3).} Let the degree of $v$ be $$d_v := r \cdot \mathrm{deg}(\mathcal{L}|_{\mathcal{C}_v}). $$

    \item \textbf{Condition (4).} For each vertex $v \in V(\Gamma)$, if $s|_{\mathcal{C}_v} \equiv 0$ then assign the cone $\mathbb{R}_{>0}$. Such vertices we call \textbf{internal vertices} and denote the set of such vertices $V_+$. Else, assign the cone $0 \in \mathbb{R}_{\geq 0}$. These vertices we call \textbf{external vertices} and denote the set of them by $V_0$.

    \item \textbf{Condition (6)'.} Let $\vec{e}$ be a directed edge. This gives us the data of two ordered components $\mathcal{C}_{v_1}, \mathcal{C}_{v_2}$, meeting at a node $q_e$ corresponding to $e$, where $v_1$ is the source and $v_2$ is the sink of the orientation. We then define $$[\tilde{m}_{\vec{e}}] := [\alpha_{q_e,1}(\mathcal{L})]$$ as in Definition \ref{alphas}.
 
\end{itemize} 

\end{construction}

In order for the above to be a well-defined tropical type, we need to check the balancing condition at each vertex $v$ holds, and anti-symmetry of the gerby slopes holds. For anti-symmetry, note that $t_{e}$ is independent of orientation, and since the actions of the isotropy are inverse on each branch of the node $q_e$ we have $\tilde{m}_{\overleftarrow{e}} = t_e - \tilde{m}_{\vec{e}}$ when the age at $q_e$ is non-zero (and $\tilde{m}_{\overleftarrow{e}} = \tilde{m}_{\vec{e}} = 0$ otherwise), so that $$[m_{\overleftarrow{e}}] = [\tfrac{r}{t_e}(t_e - \tilde{m}_{\vec{e}})] = [r - m_{\vec{e}}] = -[m_{\vec{e}}].$$

\begin{lemma}
    Under construction \ref{assocmodrtype}, the balancing condition modulo $r$ holds at each vertex.
\end{lemma}

\begin{proof}
     To see the balancing condition modulo $r$ at a vertex $v$, consider the normalisation $$\nu: \tilde{\mathcal{C}}_v \rightarrow \mathcal{C}_v.$$ For a node $q \in \mathcal{C}$ that also lives in $\mathcal{C}_v$, we have $$ \nu^{-1}(q) =\begin{cases}
			\{q_1\}, & \text{a \ singleton \ if \ } q \text{\ not \ a \ node \ of \ } \mathcal{C}_v \\
           \{q_1, q_2 \}, & \text{ \ if \ } q \text{ \ a \ node \ of \ } \mathcal{C}_v
		 \end{cases} $$

     Then, by definition of age, we must have \begin{equation}\label{balancingpullback}
         \nu^{*}\mathcal{L}|_{\mathcal{C}_v} \cong \mathcal{O}(\sum_{i \: \ p_i \in \mathcal{C}_v}  \tilde{c}_i p_i + \sum_{q \in \mathcal{C}_v \ \mathrm{not \ self-node}} \alpha_{q,\mathcal{C}_v}(\mathcal{L}) q_1 +  \sum_{q \in \mathcal{C}_v \ \mathrm{self-node}, i = 1,2} \alpha_{q,i}(\mathcal{L}) q_i) \otimes \mathcal{F}
     \end{equation} where $\mathcal{F}$ is an everywhere age $0$ line bundle. Taking the degree of \eqref{balancingpullback}, noting that $$ \mathrm{deg}(p_i) = \frac{1}{s_i}, \  \mathrm{deg}(q_e) = \frac{1}{t_e}, \  \mathrm{deg}(\mathcal{F}) \in \Z $$ and multiplying the resulting expression by $r$, using congruence \eqref{congruence} then taking the resulting equality modulo $r$ gives the result.
     
\end{proof}

\subsubsection{Distinguished loci of $\moduli$ }\label{strata}

Let $[\tau]$ be a mod $r$ tropical type. Define $$Z_{[\tau]}^{\circ} \subset \moduli $$ to be the locus of objects inducing type $[\tau]$, and denote by $Z_{[\tau]}$ its closure inside $\moduli$.

For $[\tau]$ the unique tropical type consisting of precisely $1$ vertex of genus $g$ mapping to $0$, denote the associated locus $Z_{[\tau]} = Z_{\mathrm{main}}$, the \textbf{main component}. 

\begin{example}
\begin{figure}[h!]
    \centering
    \includegraphics[scale = 0.3]{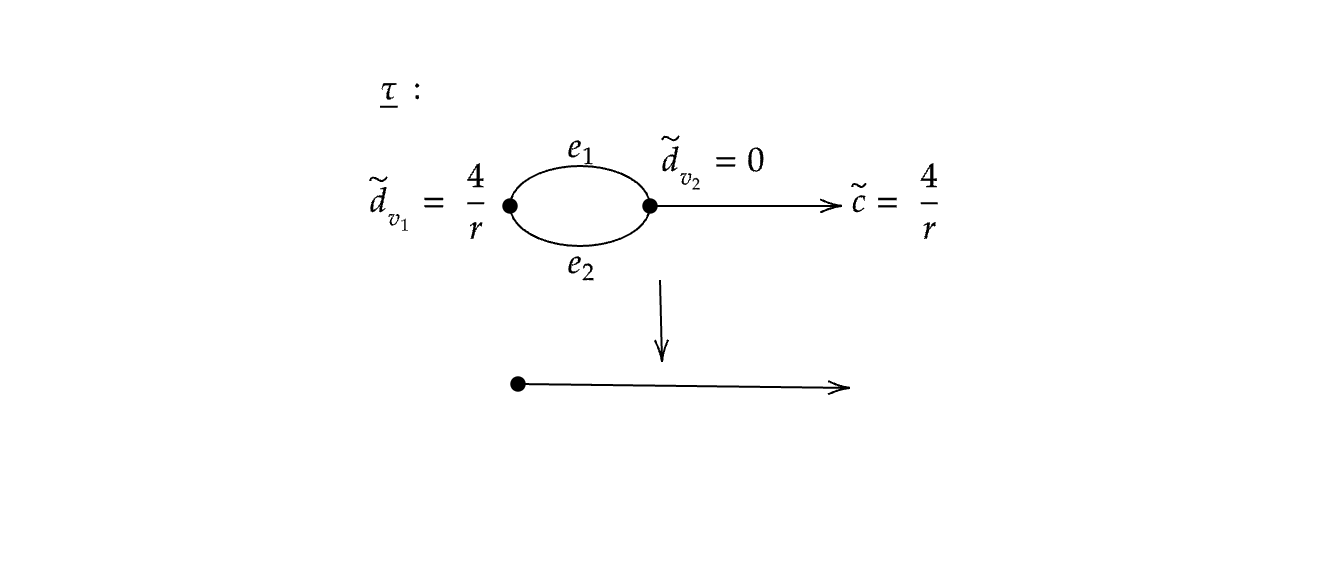}
    \caption{A genus $1$ pre-tropical type.}
    \label{fig:g1ex}
\end{figure}

A genus $1$ twisted pre-tropical type $\underline{\tau} $ is depicted in Figure \ref{fig:g1ex}. In order to make $\underline{\tau}$ into a mod $r$ tropical type, we need to assign an $r$-weighting to the two edges $e_1,e_2$. In terms of the coarse slopes, the modulo $r$ balancing condition of Definition \ref{orbitroptype} (6') at vertex $v_1$ says $$ 4 \equiv [m_{\vec{e}_1}] + [m_{\vec{e}_2}] \mod r $$ and the global degree condition tells us this is the only equation we need satisfied. There are clearly $r$ choices of solutions to this equation, in agreement with Lemma \ref{r-weightings} since $b_1(\Gamma) = 1$. After picking any solution $[m_{\vec{e}_1}], [m_{\vec{e}_2}]$, the rooting parameters $t_{e_1},t_{e_2}$ are uniquely determined by the additive orders of the $[m_{\vec{e}_i}] \in \Z/r\Z$, via equation \eqref{noderooting}.
\end{example} \qed

\subsection{Structure of $\moduli$ }

\subsubsection{Set up }

In this section we will classify the general element of $\moduli$ and construct free actions of certain semi-abelian varieties on strata of $\moduli$ built from Jacobians of curves.

We will often use the following result telling us how to push forward a line bundle on a twisted curve to the coarse space. 

\begin{lemma}[\cite{Cru24} Lemma 3.2.1]\label{pushforward}
    Let $\mathcal{C}$ be a prestable twisted curve with gerby smooth points $p_1, \cdots, p_n$ with isotropy $\mu_{s_i}$ and gerby nodes $q_1, \cdots, q_k $ with isotropy $\mu_{t_j}$. Let $\mathcal{L}$ be a line bundle on $\mathcal{C}$.  Let $\nu : \tilde{\mathcal{C}} \rightarrow \mathcal{C}$ be the normalisation of $\mathcal{C}$ and $\nu^{-1}(q_j) = \{q_{j1}, q_{j2}\}$. Suppose we have $$\nu^{*} \mathcal{L} \cong \mathcal{O}(\sum_{i = 1}^n a_i p_i + \sum_{j = 1}^k b_{j1} q_{j1} + \sum_{j = 1}^k b_{j2} q_{j2}) $$ for some integers $a_i, b_{j1}, b_{j2}$. Let $\pi : \tilde{\mathcal{C}} \rightarrow \tilde{C}$ be the coarse moduli map of the normalisation. Then we have $$\pi_* \nu^{*} \mathcal{L} \cong \mathcal{O}_{\tilde{C}}(\sum_{i = 1}^n \lfloor \frac{a_i}{s_i} \rfloor \bar{p}_i + \sum_{j = 1}^k \lfloor \frac{b_{j1}}{t_j} \rfloor \bar{q}_{j1} + \sum_{j = 1}^k \lfloor \frac{b_{j2}}{t_j} \rfloor \bar{q}_{j2}).$$
\end{lemma} \qed

Intuitively speaking, this says the pushforward of a Cartier divisor is given by subtracting off the ``fractional part'' $\langle D \rangle$, i.e. $\pi_* \mathcal{O}(D) = \mathcal{O}(\lfloor D \rfloor)$ where $\lfloor D \rfloor = D - \langle D \rangle$.

\subsubsection{Structure of the line bundles }\label{structure of the LB}

Consider $\mathcal{C}$ smooth of genus $0$ with isotropy $\mu_t$ at $q_1$ and $q_2$. Any orbifold line bundle on $\mathcal{C}$ is of the form $\mathcal{L} \cong \mathcal{O}(a q_1 + b q_2)$. The ages of $\mathcal{L}$ determine $a,b \mod t$. Additionally, if $\mathcal{L}$ is pulled back from a nodal genus $1$ curve via normalisation, we have $a \equiv - b \mod t$. We can therefore decompose $\mathcal{L}$ as $$\mathcal{O}(\langle a \rangle q_1  - \langle a \rangle q_2) \otimes \mathcal{F}$$ for $\mathcal{F}$ an age $0$ line bundle. If we enforce the degree constraints of section \ref{setup}, $\mathrm{deg}(\mathcal{F})$ must have integer degree between $-1/2$ and $1/2$, thus is degree $0$. The following lemma generalises this observation, by decomposing line bundles appearing into two parts: one part capturing the age information, and a second part coming from the coarse space. The proposition will tell us, if we decompose appropriately, we can enforce the latter part to be degree $0$.

\begin{proposition}\label{structure}
    Let $\elt \in \moduli$ and $E \subset \mathcal{C}$ an irreducible component. Denote the marked points that lie in $E$ by $p_1, \cdots, p_k$ and the nodes of $\mathcal{C}$ that lie in $E$ which are not also nodes of $E$ by $q_1,\cdots,q_N$, and $q_1', \cdots, q_M'$ the self-nodes of $E$. Denote by $\nu: \tilde{E} \rightarrow E$ the normalisation of $E$. Write $\nu^{-1}(q_{j}') = \{q_{j1}',q_{j2}'\}$ for the preimages of the self-nodes of $E$.
    
    Then \begin{enumerate}

        \item If $s$ does not vanish identically on $E$ we must have $$\nu^{*}\mathcal{L}|_E  \cong \mathcal{O}( \sum_{i = 1}^k \tilde{c}_i p_i + \sum_{j = 1}^N \alpha_{q_j,E}(\mathcal{L}) q_j + \sum_{j = 1}^M \sum_{i = 1,2} \alpha_{q_j',i}(\mathcal{L})q_{ji}') =: \mathcal{F}_0$$ Furthermore, $\nu^{*}s|_{E}$ is the unique section, up to scaling, that cuts out the Cartier divisor $\sum_{i = 1}^k \tilde{c}_i p_i + \sum_{j = 1}^N \alpha_{q_j,E}(\mathcal{L}) q_j + \sum_{j = 1}^M \sum_{i = 1,2} \alpha_{q_j',i}(\mathcal{L})q_{ji}'$.

        \item If $s$ vanishes identically on $E$ and on no other component intersecting $E$, we must have $$ \nu^{*}\mathcal{L}|_E  \cong \mathcal{F}_+ \otimes \mathcal{G} \otimes \mathcal{J}$$ where \begin{itemize}
            \item $\mathcal{F}_+ := \mathcal{O}( \sum_{i = 1}^k \tilde{c}_i p_i + \sum_{j = 1}^N (\alpha_{q_j,E}(\mathcal{L})-t_j) q_j).$
            \item $\mathcal{G} \cong \mathcal{O}( \sum_{j = 1}^M \alpha_{q_j',1}(\mathcal{L})(q_{j1}' - q_{j2}')).$
            \item $\mathcal{J} \in \mathrm{Jac}(\tilde{E})$ is an everywhere age $0$ degree $0$ line bundle.
        \end{itemize} 
    \end{enumerate}
    
\end{proposition}

\begin{remark}\label{G diff age 0}
   The line bundle $\mathcal{F}_{+}$ is canonical, up to isomorphism, but the line bundles $\mathcal{G}, \mathcal{J}$ are non-canonical. Indeed, we have chosen a particular ordering for each pair $q_{j1}',q_{j2}'$ to construct $\mathcal{G}$. If we reversed the ordering for a particular $j$, $\mathcal{G}$ differs by twisting by \begin{equation}\label{G diff}
       \mathcal{O}( \alpha_{q_j',1}(\mathcal{L})(q_{j1}' - q_{j2}') - \alpha_{q_j',2}(\mathcal{L})(q_{j2}' - q_{j1}')) \cong \mathcal{O}( (\alpha_{q_j',1}(\mathcal{L}) + \alpha_{q_j',2}(\mathcal{L}))(q_{j1}' - q_{j2}')).
   \end{equation} By equation \eqref{congruence}, this is an everywhere age $0$ line bundle and is degree $0$. It can therefore be absorbed into the $\mathcal{J}$ term.
\end{remark}

\begin{proof}

The proof can be considered a higher genus modification of \cite[Lemma 4.2.7]{ACW17}. The main difference being, where in \cite{ACW17} just the degrees of line bundles are taken into account, in our higher genus setting it will be essential to keep track of the degree $0$ cycles themselves.

In all cases, by definition of age, we can write $$\nu^{*}\mathcal{L}|_E  \cong \mathcal{O}( \sum_{i = 1}^k \tilde{c}_i p_i + \sum_{j = 1}^N \alpha_{q_j,E}(\mathcal{L}) q_j + \sum_{j = 1}^M \sum_{i = 1,2} \alpha_{q_j',i}(\mathcal{L})q_{ji}') \otimes \mathcal{F}$$ for $\mathcal{F}$ age $0$ at every point of $\tilde{E}$. In particular, if $\pi: \tilde{E} \rightarrow \tilde{\bar{E}}$ denotes the coarse map of $\tilde{E}$, then $\mathcal{F} = \pi^{*}\mathcal{F}'$ for some $\mathcal{F}'$ on $\tilde{\bar{E}}$. 

We now prove part (1). Suppose $s \not \equiv 0 $. By the projection formula and Lemma \ref{pushforward}, we see \begin{equation}\label{trivpush}
    \pi_{*} \nu^{*}\mathcal{L}|_E \cong \mathcal{F}'.
\end{equation} Furthermore, $\pi_{*} \nu^{*}s|_E \in H^0(\tilde{\bar{E}}, \mathcal{F}')$ is not the $0$ section since $s$ does not identically vanish on $E$. It follows that $$\mathrm{deg}(\mathcal{F}') \geq 0 $$ and since we are also assuming $\mathrm{deg} \mathcal{L}|_E < 1/2$ as in section \ref{setup}, we must have $$-\frac{1}{2} < \mathrm{deg}(\mathcal{L}|_E) = \mathrm{deg}(\nu^{*}\mathcal{L}|_E) = \mathrm{deg}(\mathcal{O}( \sum_{i = 1}^k \tilde{c}_i p_i + \sum_{j = 1}^N \alpha_{q_j,E}(\mathcal{L}) q_j + \sum_{j = 1}^M \sum_{i = 1,2} \alpha_{q_j',i}(\mathcal{L})q_{ji}')) + \mathrm{deg}(\mathcal{F}') < \frac{1}{2}.$$ Since all the coefficients in the first line bundle are non-negative, this implies $$\mathrm{deg}(\mathcal{F}') < \frac{1}{2}$$ and so, since $\mathcal{F}'$ is from the coarse space and therefore has integral degree, $$\mathrm{deg}(\mathcal{F}') = 0, \mathcal{F}' \in \mathrm{Jac}(\tilde{\bar{E}}).$$ The only Jacobian element that admits a non-zero section is the trivial bundle, thus $\mathcal{F}' \cong \mathcal{O}$ (via section $\pi_* \nu^{*}(s|_E) \not = 0$). Furthermore, equation \eqref{trivpush} implies the space of sections of $\nu^{*} \mathcal{L}|_E$ is one-dimensional, which must be generated by $\pi_* \nu^{*}(s|_E)$. This section cuts out the Cartier divisor as claimed in part (1). This proves part (1).

    Now for part (2), suppose that $s$ vanishes identically on $E$ but on no other component meeting $E$. Let $D$ be the union of the other components of $\mathcal{C}$ that intersect $E$, not including $E$, so $E \cap D = \{q_1, \cdots, q_N \}$. We have $$ \sum_{i = 1}^k \mathrm{age}_{p_i} (\mathcal{L}|_E) \leq \sum_{i = 1}^n \mathrm{age}_{p_i}(\mathcal{L}) = \mathrm{deg}(\mathcal{L}) < 1/2.$$ On the other hand, $$\sum_{j = 1}^N \mathrm{age}_{q_j} (\mathcal{L}|_E) = \sum_{j = 1}^N (1 - \mathrm{age}_{q_j} (\mathcal{L}|_D)) = N - \sum_{j = 1}^N \mathrm{age}_{q_j} (\mathcal{L}|_D) $$ since the actions of the isotropy groups at the nodes are inverse on each branch. Here we have used that $\mathrm{age}_{q_j}(\mathcal{L}|_D) \neq 0$: since $s$ vanishes identically on $E$, the restriction of $s$ to the component of $D$ through $q_j$ vanishes at $q_j$, so by part (1) its order of vanishing $\alpha_{q_j,D}(\mathcal{L})$ there is non-zero. Since $s$ does not vanish identically on any component of $D$ meeting $E$, applying part (1) to these components gives $$  \sum_{j = 1}^N \mathrm{age}_{q_j} (\mathcal{L}|_D) \leq \mathrm{deg} \mathcal{L}|_{D} < 1/2$$ where the latter inequality is by the degree assumption of section \ref{setup}. The above shows that \begin{equation}\label{ineq1}
        N - \frac{1}{2} < \sum_{j = 1}^N \mathrm{age}_{q_j}(\mathcal{L}|_E) \leq N
    \end{equation} and also \begin{equation}\label{ineq2}
        0 \leq \sum_{i = 1}^k \mathrm{age}_{p_i}(\mathcal{L}|_E) < \frac{1}{2}.
    \end{equation} Summing equations \eqref{ineq1} and \eqref{ineq2} gives \begin{equation}\label{ineq3}
        N - \frac{1}{2} < \sum_{i =1 }^k \mathrm{age}_{p_i}(\mathcal{L}|_E) + \sum_{j = 1}^N \mathrm{age}_{q_{j}}(\mathcal{L}|_E) < N + \frac{1}{2}.
    \end{equation}

    Write \begin{equation}\label{window2}
        \nu^{*}\mathcal{L}|_E  \cong \mathcal{O}( \sum_{i = 1}^k \tilde{c}_i p_i + \sum_{j = 1}^N (\alpha_{q_j,E}(\mathcal{L})-t_j) q_j + \sum_{j = 1}^M \alpha_{q_j',1}(\mathcal{L})(q_{j1}' - q_{j2}')) \otimes \mathcal{F}
    \end{equation} for some age $0$ line bundle $\mathcal{F}$ to be determined. Taking the degree of equation \eqref{window2} gives \begin{equation}\label{degeqn}
        \mathrm{deg}(\nu^{*} \mathcal{L}|_{E}) = \sum_{i = 1}^k \frac{\tilde{c}_i}{s_i} + \sum_{j = 1}^N (\mathrm{age}_{q_j}(\mathcal{L}|_{E}) - 1) + \mathrm{deg}(\mathcal{F})
    \end{equation} $$  =  \sum_{i = 1}^k \mathrm{age}_{p_i}(\mathcal{L}|_E) + \sum_{j = 1}^N \mathrm{age}_{q_{j}}(\nu^{*} \mathcal{L}|_E) - N + \mathrm{deg}(\mathcal{F}).$$

    Rearranging equation \eqref{degeqn} and plugging into equation \eqref{ineq3} and cancelling $N$ from each term gives \begin{equation}\label{finalineq}
        -\frac{1}{2} < \mathrm{deg}(\mathcal{L}|_E) - \mathrm{deg}(\mathcal{F}) < \frac{1}{2}.
    \end{equation}

    Now $\mathrm{deg}(\mathcal{F}) \in \Z$ and by the degree assumptions in section \ref{setup} we have $-\frac{1}{2} < \mathrm{deg}(\mathcal{L}|_E) < \frac{1}{2}$, therefore inequality \eqref{finalineq} implies $$\mathrm{deg}(\mathcal{F}) = 0 $$ as desired.

\end{proof}

\begin{remark}
    \cite[Example 3.2.6]{Cru24} gives a genus $0$ example where $\mathcal{J}$ cannot be multi-degree $0$. The key point is the degree bound of lying in $(-1/2,1/2)$ cannot be satisfied otherwise in the presence of a ``mid-age" contact order at the node.
\end{remark}

\begin{example}\label{Jacexample}
    Let $\mathcal{C}$ be the genus $2$ twisted curve with a single node and genus $1$ normalisation $\tilde{\mathcal{C}}$, as depicted below. \begin{center}
        \includegraphics[scale = 0.5]{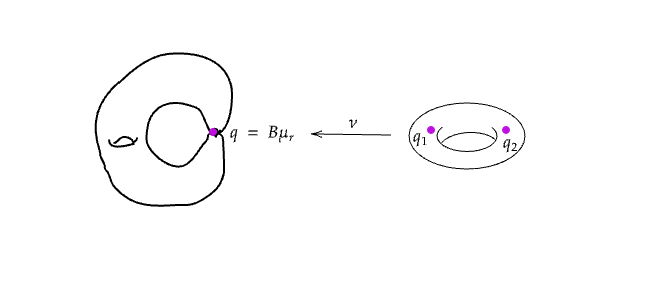}
    \end{center}

    The data of a line bundle $\mathcal{L}$ on $\mathcal{C}$ is equivalent to the data of: \begin{itemize}
        \item A line bundle on the coarse space $\mathcal{F} \in \mathrm{Jac}(C)$, an extension of $\mathrm{Jac}(\tilde{C})$ by $\mathbb{G}_m$.
        \item A line bundle on $\tilde{\mathcal{C}}$ of the form $\mathcal{O}(\lambda(q_1 - q_2))$ for $\lambda \in \{0,1,\cdots,r-1\}$.
    \end{itemize} If $\elt \in \moduli$ then $\nu^{*}\mathcal{L}$ must be age $0$ away from $q_1,q_2$ and $- \frac{1}{2} < \mathrm{deg}(\nu^{*}\mathcal{L}) = \mathrm{deg}(\mathcal{L}) < \frac{1}{2}$. Therefore we must have $$\nu^{*}\mathcal{L} \cong \mathcal{O}(\lambda(q_1 - q_2)) \otimes \mathcal{F}$$ for some $\mathcal{F} \in \mathrm{Jac}(\tilde{C})$ and so $\mathcal{L}$ is induced by $\mathcal{O}(\lambda(q_1 - q_2))$ up to twisting by an element of $\mathrm{Jac}(C)$. In particular, the possibilities for $\mathcal{L}$ form an index $r$ extension of $\mathrm{Jac}(C)$.
\end{example} \qed

\subsubsection{Semi-abelian group actions }\label{abelian group actions}

We now define a semi-abelian variety that helps classify the possible choices of line bundles appearing for a fixed curve $\mathcal{C}$ arising as a part of the data of an object of $\moduli$. 

Given $\elt \in \moduli$, define $$\mathcal{C}_{+} := \bigcup_{v \in V_+} \mathcal{C}_v \subset \mathcal{C}$$ and $$\mathcal{C}_{0} := \bigcup_{v \in V_0} \mathcal{C}_v \subset \mathcal{C}.$$ Let $\alpha: \tilde{\mathcal{C}} \rightarrow \mathcal{C}$ be the partial normalisation of $\mathcal{C}$ at the points $\mathcal{C}_+ \cap \mathcal{C}_0$, so $\tilde{\mathcal{C}} = \mathcal{C}_0 \coprod \mathcal{C}_+$. 

\begin{definition}[bipartite edges and normalised graph]\label{bipartite edges}
Let $\elt$ have associated dual graph $\Gamma$. We call an edge $e \in E$ \textbf{bipartite} if it connects together a vertex in $V_0$ to a vertex in $V_+$. Denote the set of all bipartite edges by $E^{\mathrm{b}}$.

Additionally, define $\Gamma^{\dag}$ to be the graph induced from $\Gamma$ after removing all bipartite edges $E^{\mathrm{b}}$. Equivalently, $\Gamma^{\dag}$ is the dual graph of the curve $\tilde{\mathcal{C}}$. Denote by $\Gamma_0, \Gamma_+ \subset \Gamma$ the subgraphs induced on $V_0$ and $V_+$ respectively, so that $\Gamma^{\dag} = \Gamma_0 \sqcup \Gamma_+$; in particular $b_0(\Gamma_0)$ is the number of connected components of $\mathcal{C}_0$.
\end{definition} \qed

We have a short exact sequence of sheaves $$1 \rightarrow \mathcal{O}_{\mathcal{C}}^{*} \rightarrow \alpha_* \mathcal{O}_{\tilde{\mathcal{C}}}^{*} \rightarrow \prod_{q \in \mathcal{C}_+ \cap \mathcal{C}_0} \mathbb{C}^{*}_q \rightarrow 0 $$ where $\mathbb{C}^{*}_q$ is the skyscraper sheaf supported at $q$ with sections $\mathbb{C}^{*}$. Passing to the associated long exact sequence and passing to an induced short exact sequence yields $$1 \rightarrow \mathbb{G}_m^{|E^{\mathrm{b}}| - b_0(\Gamma^{\dag}) + 1} \rightarrow \mathrm{Pic}(\mathcal{C}) \rightarrow \mathrm{Pic}(\mathcal{C}_0) \oplus \mathrm{Pic}(\mathcal{C}_+) \rightarrow 0.$$ Note that we may describe $b_0(\Gamma^{\dag})$ as the number of connected components of $\tilde{\mathcal{C}}$ and $|E^{\mathrm{b}}| = |\mathcal{C}_0 \cap \mathcal{C}_+|$. 
 
We may also pass to multi-degree $0$ line bundles, giving rise to \begin{equation}\label{SESJac}
    1 \rightarrow \mathbb{G}_m^{|E^{\mathrm{b}}| - b_0(\Gamma^{\dag}) + 1} \rightarrow \mathrm{Jac}(\mathcal{C}) \xrightarrow{\alpha^{*}} \mathrm{Jac}(\mathcal{C}_0) \oplus \mathrm{Jac}(\mathcal{C}_+) \rightarrow 0.
\end{equation} The latter non-trivial homomorphism in the above sequence is pullback of line bundles $\alpha^{*}$.

\begin{definition}\label{Jactau}

Define $$\mathrm{Jac}^{+}_{\mathcal{C}} \subset (\alpha^{*})^{-1}(\{ \mathcal{O}_{\mathcal{C}_0} \} \oplus \mathrm{Jac}(\mathcal{C}_+))$$ to be the subgroup of line bundles induced by sequence \eqref{SESJac} which are everywhere age $0$. The age $0$ condition is equivalent to doing the analogous construction but for coarse spaces only, and so $\mathrm{Jac}_{\mathcal{C}}^{+}$ depends only on the coarse curves $C, C_0, C_+$. 

For $[\tau]$ a mod $r$ tropical type, denote by $$\mathrm{Jac}_{[\tau]} \subset \mathrm{Jac}$$ the universal family over the stratum $Z_{[\tau]}^{\circ} \subset \moduli$ of these $\mathrm{Jac}_{\mathcal{C}}^{+}$, where $ \mathrm{Jac}$ is the universal Jacobian pulled back from $\mathfrak{M}_{g,n}^{\mathrm{pre}}$.
    
\end{definition}

\begin{example}
    This example highlights the global nature of elements of $\mathrm{Jac}_{\mathcal{C}}^{+}$ and in particular how Definition \ref{Jactau} accounts for the gluing between $\mathcal{C}_0$ and $\mathcal{C}_+$. Let $\mathcal{C}$ be the twisted curve with two genus $0$ irreducible components: \begin{center}

\tikzset{every picture/.style={line width=0.75pt}} 

\begin{tikzpicture}[x=0.75pt,y=0.75pt,yscale=-1,xscale=1]

\draw  [draw opacity=0] (342.18,161.4) .. controls (320.57,154.69) and (305,135.77) .. (305,113.48) .. controls (305,91.86) and (319.66,73.41) .. (340.26,66.21) -- (359.5,113.48) -- cycle ; \draw   (342.18,161.4) .. controls (320.57,154.69) and (305,135.77) .. (305,113.48) .. controls (305,91.86) and (319.66,73.41) .. (340.26,66.21) ;  
\draw  [draw opacity=0] (310.43,66.43) .. controls (332.09,72.97) and (347.82,91.76) .. (348,114.05) .. controls (348.17,135.67) and (333.66,154.24) .. (313.12,161.61) -- (293.5,114.48) -- cycle ; \draw   (310.43,66.43) .. controls (332.09,72.97) and (347.82,91.76) .. (348,114.05) .. controls (348.17,135.67) and (333.66,154.24) .. (313.12,161.61) ;  

\draw (322,180.4) node [anchor=north west][inner sep=0.75pt]    {$\mathcal{C} \ $};
\draw (271,106.4) node [anchor=north west][inner sep=0.75pt]    {$\mathcal{C}_{0} \ $};
\draw (354,108.4) node [anchor=north west][inner sep=0.75pt]    {$\mathcal{C}_{+} \ $};

\end{tikzpicture}
    \end{center} Then the short exact sequence \eqref{SESJac} is $$1 \rightarrow \mathbb{G}_m \rightarrow \mathrm{Jac}(\mathcal{C}) \rightarrow A_0 \oplus A_+ \rightarrow 0  $$ where $A_0 = \mathrm{Jac}(\mathcal{C}_0) \cong A_+ = \mathrm{Jac}(\mathcal{C}_+)$ are the finite abelian groups consisting of line bundles of the form $\mathcal{O}(u q_1 + v q_2) $ for $\frac{u}{t_1} + \frac{v}{t_2} = 0, 0 \leq |u| < t_1, 0 \leq |v| < t_2$ and $q_i \cong B\mu_{t_i}$ are the two nodes. In particular, a line bundle in $A_+$ is everywhere age $0$ if and only if it is the trivial bundle, and thus $$\mathrm{Jac}_{\mathcal{C}}^{+} \cong \mathbb{G}_m$$ corresponding to the moduli of twistings given by gluing $\mathcal{C}_0$ to $\mathcal{C}_+$.
\end{example}


\begin{proposition}\label{action}
    For $[\tau]$ any mod $r$ tropical type, there is a free action $$\mathrm{Jac}_{[\tau]} \times [(\mathbb{G}_m)^{b_0(\Gamma_0)}/ \mathbb{G}_m] \curvearrowright Z_{[\tau]}^{\circ} $$ which acts transitively on the sub-locus of data with a fixed curve $\mathcal{C}$, which we denote here by $Z_{\mathcal{C}} \subset Z_{[\tau]}^{\circ}$. Here $(\mathbb{G}_m)^{b_0(\Gamma_0)}/ \mathbb{G}_m$ is the quotient by the diagonal sub-torus, and this acts via scaling of sections.
\end{proposition}

\begin{proof}
    We need to define $\mathcal{J} \cdot \elt$ for $\mathcal{J} \in \mathrm{Jac}^{+}_{\mathcal{C}}$. We keep $\mathcal{C}$ fixed, and define $\mathcal{J}\cdot \mathcal{L} = \mathcal{J} \otimes \mathcal{L}$ the usual tensor product of line bundles on $\mathcal{C}$.  We just need to define the section $\mathcal{J}\cdot s \in H^0(\mathcal{C},\mathcal{J} \otimes \mathcal{L})$. The idea is simple: $\mathcal{J} \otimes \mathcal{L}$ differs from $\mathcal{L}$ only on the locus of $\mathcal{C}_+$ where the section $s$ identically vanishes. We can then define $\mathcal{J}\cdot s|_{\mathcal{C}_v} = s|_{\mathcal{C}_v}$ for each component. More formally, there is a short exact sequence of sheaves $$ 0 \rightarrow \mathcal{J} \otimes \mathcal{L} \rightarrow \alpha_* \alpha^{*} (\mathcal{J} \otimes\mathcal{L}) \rightarrow \bigoplus_{q \in \mathcal{C}_0 \cap \mathcal{C}_+} \mathbb{C}_q \rightarrow 0.$$ This induces the exact sequence on cohomology groups \begin{equation}\label{twistedsections}
        0 \rightarrow H^0(\mathcal{C},\mathcal{J} \otimes \mathcal{L}) \rightarrow H^0(\tilde{\mathcal{C}},\alpha^*(\mathcal{J} \otimes \mathcal{L})) = H^0(\mathcal{C}_0,\alpha^*(\mathcal{J} \otimes \mathcal{L}|_{\mathcal{C}_0})) \oplus H^0(\mathcal{C}_+,\alpha^*(\mathcal{J} \otimes \mathcal{L}|_{\mathcal{C}_+} )) \rightarrow \mathbb{C}^{|E^{\mathrm{b}}|}.
    \end{equation} In other words, sections of $\mathcal{J} \otimes \mathcal{L}$ are sections of $\alpha^* (\mathcal{J} \otimes \mathcal{L})$ that glue at the $|E^{\mathrm{b}}|$ nodes $\mathcal{C}_0 \cap \mathcal{C}_+$. However, there is a natural identification $$\mathcal{J} \otimes \mathcal{L}|_{\mathcal{C}_0} \cong \mathcal{L}|_{\mathcal{C}_0} $$ and therefore the data of $(s|_{\mathcal{C}_0},0)$ induces a well-defined element of the second term of sequence \eqref{twistedsections}. This is in the kernel of the latter morphism since all the evaluation maps of the section $s$ at the points in $\mathcal{C}_0 \cap \mathcal{C}_+$ are $0$. This therefore defines a unique element of $H^0(\mathcal{C},\mathcal{J} \otimes \mathcal{L})$ which we define as $\mathcal{J} \cdot s$. It is clear that this gives an action of $\mathrm{Jac}^{+}_{\mathcal{C}}$ on $Z_{\mathcal{C}}$ which is free, since the action on the line bundle component is free. This action is compatible in families and so gives an action of $\mathrm{Jac}_{[\tau]}$ on $Z_{[\tau]}^{\circ}$.

    Now we deal with the transitivity on a given $Z_{\mathcal{C}}$. In the case connected components of $\mathcal{C}_+$ are irreducible, transitivity on the line bundle factor follows immediately from Proposition \ref{structure} (2). 
    
    Now consider the case $\mathcal{C}_+$ is potentially reducible. Suppose that $\elt, (\mathcal{C},(\mathcal{L}',s')) \in Z_{[\tau]}^{\circ}$, and consider the restrictions of the line bundles to $ \mathcal{C}_+$, $$\nu^* \mathcal{L} \cong \mathcal{F}_+ \otimes \mathcal{G}_1 \otimes \mathcal{J}_1, \ \nu^* \mathcal{L}' \cong \mathcal{F}_+ \otimes \mathcal{G}_2 \otimes \mathcal{J}_2$$ as in the notation of Proposition \ref{structure} (2), where $\mathcal{J}_1,\mathcal{J}_2, \mathcal{G}_1, \mathcal{G}_2$ are total degree $0$ line bundles and $\nu : \tilde{\mathcal{C}} \rightarrow \mathcal{C}$ is the partial normalisation along $\mathcal{C}_0 \cap \mathcal{C}_+$. 

    We claim that $$ (\mathcal{J}_2 \otimes \mathcal{G}_2) \otimes (\mathcal{G}_1 \otimes \mathcal{J}_1)^{\vee} \in \mathrm{Jac}(\tilde{\mathcal{C}})$$ is multi-degree $0$. In other words, $\mathcal{J}_1 \otimes \mathcal{G}_1, \mathcal{J}_2 \otimes \mathcal{G}_2$ have the same multi-degrees. If not, then there must exist a vertex $v \in V_+$ such that $\mathrm{deg}((\mathcal{J}_2 \otimes \mathcal{G}_2)|_{\mathcal{C}_v} \otimes (\mathcal{J}_1 \otimes \mathcal{G}_1)|_{\mathcal{C}_v}^{\vee}) \in \mathbb{Z} \setminus 0$. The fact this degree is integral is because $\mathcal{J}_1,\mathcal{J}_2$ are age $0$ and $\mathcal{G}_2 \otimes \mathcal{G}_1^{\vee}$ is age $0$ by Remark \ref{G diff age 0}. But then the condition $$\mathrm{deg}(\mathcal{L}|_{\mathcal{C}_v}) \in (-1/2,1/2)$$ implies $$ \mathrm{deg}(\mathcal{L}'|_{\mathcal{C}_v}) = \mathrm{deg}(\mathcal{L}|_{\mathcal{C}_v} \otimes (\mathcal{J}_2 \otimes \mathcal{G}_2)|_{\mathcal{C}_v} \otimes (\mathcal{J}_1 \otimes \mathcal{G}_1)|_{\mathcal{C}_v}^{\vee}) \not \in (-1/2,1/2)$$ giving a contradiction to $(\mathcal{C},(\mathcal{L}',s'))$ defining an element of $\moduli$. It then follows the action is transitive as in the previous case: $(\mathcal{J}_2 \otimes \mathcal{G}_2) \otimes (\mathcal{G}_1 \otimes \mathcal{J}_1)^{\vee}$ induces a torsor's worth of elements of $\mathrm{Jac}_{\mathcal{C}}^{+}$ under the group of possible gluings of $\mathcal{C}_+$ to $\mathcal{C}_0$. Then there is a corresponding $\mathcal{J} \in \mathrm{Jac}_{\mathcal{C}}^{+}$ in this torsor which is the desired element we act by to get from $\mathcal{L}$ to $\mathcal{L}'$. It is easy to see such an element sends $s$ to $s'$, up to scalar, since on external components the action is trivial and the sections already agree (Proposition \ref{structure} (1)) and on internal components both $s,s'$ vanish identically.

    We see that the action above is transitive up to potential scalings of the sections. On components on which $s$ is not identically $0$, we can scale the associated section by a constant. On any neighbouring components, such a scaling induces the same scaling on these components also, giving the $\mathbb{G}_m^{b_0(\Gamma_0)}$ factor. Now, such automorphisms are distinct up to isomorphism induced by the global scaling on the line bundle which is a $\mathbb{G}_m$. This scaling therefore defines a free action of $[(\mathbb{G}_m)^{b_0(\Gamma_0)}/ \mathbb{G}_m]$ on $Z_{\mathcal{C}}$, again compatible in families. The product of the two actions is transitive by the above discussion and so the result follows.
\end{proof}


\subsection{Classification of irreducible components}\label{classification}

\subsubsection{Essential types }

\begin{definition}[Essential mod $r$ tropical types.]\label{essentialtype}

A mod $r$ tropical type $[\tau]$ is called \textbf{essential} if: 

\begin{itemize}
    \item The decomposition of the set of vertices of $\Gamma$, $$V(\Gamma) = V_+ \cup V_0 $$ as in Construction \ref{assocmodrtype} makes $\Gamma$ into a bipartite graph.

    \item For any $v \in V_+$, the genus $g_v > 0$.
\end{itemize}
\end{definition}

\begin{figure}[h!]
    \centering
    \includegraphics[scale = 0.5]{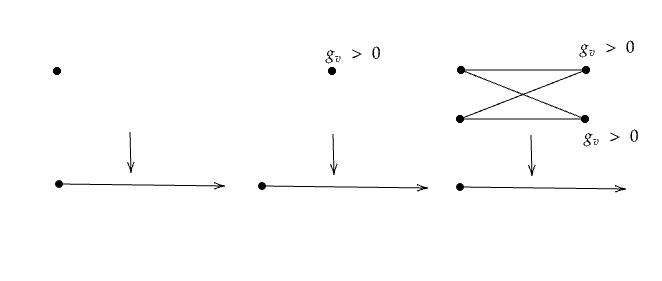}
    \caption{Three different depictions of mod $r$ tropical types}
    \label{essentialtypeexamples}
\end{figure}

\begin{remark}\label{essentialremarks}
    By definition, for $\elt$ to induce an essential mod $r$ type $[\tau]$, $\mathcal{C}_+, \mathcal{C}_0$ must both be disjoint unions of smooth twisted curves and so Proposition \ref{action} tells us the action of $\mathrm{Jac}^{+}_{\mathcal{C}} \times \mathbb{G}_m^{|V_0| - 1}$ on the fibre of $Z_{[\tau]}^{\circ} $ corresponding to $\mathcal{C}$ is transitive. Furthermore, an essential type has a canonical orientation on edges. Namely, we will often want to orient the edges $\vec{e}$ such that the source of $\vec{e}$ is in $V_0$ and the sink is in $V_+$; the bipartite condition tells us this is well-defined.
\end{remark}

\begin{definition}\label{inducibletype}
    A mod $r$ tropical type $[\tau]$ is called \textbf{inducible} if there exists an object $\elt \in \moduli$ whose associated type is $[\tau]$. Equivalently, $[\tau]$ is inducible if and only if $Z_{[\tau]} \not = \emptyset$.
\end{definition}




\begin{theorem}\label{irredcpts}
    There is a bijection $$\mathrm{Irred. \ cpts \ of \ } \moduli \leftrightarrow \mathrm{Inducible \ essential \ tropical \ types \ } [\tau] $$ given by $[\tau] \mapsto Z_{[\tau]}$.
\end{theorem}

The proof of the theorem will follow from a sequence of smoothing lemmas for line bundle-section pairs on curves.

\subsubsection{Roadmap}

We briefly outline the different components that go into proving Theorem \ref{irredcpts}.

\textbf{Step 1: Smoothing.} The goal of this section is to show, starting with object $\elt \in \moduli$ there always exists an infinitesimal deformation of $\elt$ in the moduli space whose generic fibre is of essential type (culminating in Theorem \ref{essentialtypethm}). This involves producing infinitesimal smoothings of nodes where on each branch the section identically vanishes (Lemma \ref{internaledge}), producing infinitesimal smoothings of nodes where on each branch the section does not identically vanish (Lemma \ref{contractededge}), and deforming away rational components on which the section identically vanishes (Lemma \ref{genus0}). This is achieved by constructing the deformations on simpler example curves, then gluing together deformations to get the general results, using Lemma \ref{propagation}.

\textbf{Step 2: Non-smoothing.} Step $1$ allows us to show that the strata indexed by essential types union to the whole moduli space. We then aim to show each of these strata are unions of irreducible components. This amounts to identifying explicit open loci in each stratum, which correspond to the locus where non-trivial Jacobian factors on the associated line bundles appear, following Proposition \ref{structure}. Showing these subspaces are indeed open comes down to proving there are obstructions to creating deformations that change the mod $r$ tropical type. Namely, we show there aren't infinitesimal deformations that contract any edges or move interior vertices to an exterior vertex. All of the arguments here rely on the structural results for objects in the moduli space, Proposition \ref{structure}. Once again this is argued first for the case of simpler curves, then the general obstructions to deformations are deduced from these cases.

\textbf{Step 3: Irreducibility of each essential stratum.} Section \ref{putting it all together} performs a final analysis to deduce each essential stratum is irreducible (Lemma \ref{stratairred}). This uses the description of each stratum as a torsor under a semi-abelian group given in section \ref{abelian group actions}. These abelian groups are irreducible, and the strata are realised as torsors under these groups over boundary strata of moduli of curves (up to stacky issues). The irreducibility then follows from the irreducibility of these two types of objects. The result then easily follows. 

\subsubsection{Smoothing results }\label{smoothing results}

Throughout this section we will denote by $\Delta := \mathrm{Spec}(\C[[\epsilon]])$ the infinitesimal disc. We will make frequent use of the following technical Lemma.

\begin{lemma}[\cite{Cru24} Lemma 3.3.6]\label{propagation}
Let $\mathcal{C}$ be a pre-stable twisted curve and write $\mathcal{C} = Z_1 \cup Z_2$ a union of two proper subcurves $Z_1,Z_2$ intersecting only in nodes $Z_1 \cap Z_2 = \coprod_{j \in J} q_j$. Let: \begin{itemize}
    \item $\mathcal{Z}_i \rightarrow \Delta$ be deformations of $Z_i$ that do not smooth any of the nodes $q_j, j \in J$, for $i = 1,2$. In other words, there are disjoint trivial subfamilies of nodes $\coprod_{j \in J} \tilde{q}_j \subset \mathcal{Z}_i, \ i = 1,2$ where $\tilde{q}_j = q_j \otimes \C[[\epsilon]]$.
    \item $\mathcal{C}^{\mathrm{def}} \rightarrow \Delta$ be the deformation of $\mathcal{C}$ formed by the pushout diagram \begin{center}
        \begin{tikzcd}
\mathcal{C}^{\mathrm{def}}         & \mathcal{Z}_1 \arrow[l, hook']                                   \\
\mathcal{Z}_2 \arrow[u, hook] & \coprod_{j \in J} \tilde{q}_j \arrow[l, hook] \arrow[u, hook']
\end{tikzcd}
   \end{center} 

    \item $(\mathcal{L},s)$ a line bundle-section pair on $\mathcal{C}$.
    \item $(\tilde{\mathcal{L}}_i,\tilde{s}_i)$ any two deformations of $(\mathcal{L}_{i},s_{i}):= (\mathcal{L}|_{Z_i},s|_{Z_i})$ over $\mathcal{Z}_i$ for $i = 1,2$ such that the restrictions of $\tilde{s}_i$ to $\tilde{q}_j$ are independent of $\epsilon$.
\end{itemize}

Then there is a $\mathbb{G}_a^{|J| -k_1 - k_2 + 1}$-torsor of line bundle-section pairs $(\tilde{\mathcal{L}},\tilde{s})$ inducing this data under restriction, where $k_i$ denotes the number of connected components of $Z_i$.

\end{lemma}

The key use of the above Lemma is that it allows us to run deformation theory arguments on a subset of components of a curve, and then glue these results together to deduce a global deformation theory statement. Indeed, it tells us that if we write $\mathcal{C} = \mathcal{C}_1 \cup \mathcal{C}_2$ and deform each $\mathcal{C}_i$ individually, along with cooking up a line-bundle section pair over each of these deformations, we can often glue these line bundle section pairs to a global line-bundle section pair over a deformation of $\mathcal{C}$.

\begin{lemma}\label{internaledge}
Let $\elt \in \moduli$ such that there exist two potentially non-distinct irreducible components $\mathcal{C}_1,\mathcal{C}_2 \subset \mathcal{C}$ meeting at a node $q$ with $s|_{\mathcal{C}_1\cup \mathcal{C}_2} \equiv 0$. Then there exists an infinitesimal smoothing of $q$ in the moduli space $\moduli$.    
\end{lemma}

\begin{proof}
    Suppose first that $\mathcal{C} = \mathcal{C}_1 \cup_q \mathcal{C}_2$. Then $(\mathcal{L},s) = (\mathcal{L},0)$. The universal Picard stack $$\mathrm{Pic} \rightarrow \mathfrak{M}^{\mathrm{tw}}_{g,n} $$ over the stack of pre-stable twisted genus $g$ $n$-marked curves is smooth. Therefore we can extend any smoothing of the node $\mathcal{C}^{\mathrm{def}}/\Delta$ (for $\Delta = \mathrm{Spec}(\C[[\epsilon]])$) to a smoothing of the line bundle $\mathcal{L}$. 
    
    Now for the general case of $\mathcal{C}$ having more components, take $\mathcal{C}^{\mathrm{def}}/\Delta$ to be a family of marked pre-stable curves that smooth $q$ but restricts to the trivial family on other components away from $\mathcal{C}_1 \cup \mathcal{C}_2$. Then $$ \mathcal{C}^{\mathrm{def}} = \mathcal{Z}_1 \cup \mathcal{Z}_2$$ where \begin{itemize}
        \item $\mathcal{Z}_1$ is an irreducible component of $\mathcal{C}^{\mathrm{def}}$ which is a smoothing of $(\mathcal{C}_1 \cup \mathcal{C}_2)/\Delta$.
        \item $\mathcal{Z}_2$ is a union of other irreducible components corresponding to the trivial families $\mathcal{C}_l \otimes \C[[\epsilon]]$ for $l \not = 1,2$.

    \end{itemize} 

    Observe that \begin{itemize}
    \item $Z_1 \cap Z_2 = \coprod_{i \in J} q_i$ a union of nodes distinct from $q$.

    \item $\mathcal{Z}_1 \cap \mathcal{Z}_2 = \coprod_{i \in J} \tilde{q}_i$ where $\tilde{q}_i = q_i \times \Delta$ a trivial family. Note that in particular, if $q_i = B \mu_t$ then $\tilde{q}_i \cong B \mu_t \times \Delta$ non-canonically. This follows from the vanishing of the sheaf cohomology group $H^2(\Delta, \mu_t) = 0$.
\end{itemize}

Define $(\tilde{\mathcal{L}}_1, \tilde{s}_1)$ to be the line bundle constructed from the first part of the proof with $0$ section, and $(\tilde{\mathcal{L}}_2, \tilde{s}_2) = (\mathcal{L}|_{Z_2} \otimes \C[[\epsilon]], s|_{Z_2} \otimes \C[[\epsilon]])$ the trivial deformation of the restriction. The above data then puts us in the setup of Lemma \ref{propagation} and therefore we obtain our desired $(\tilde{\mathcal{L}},\tilde{s})$ on $\mathcal{C}^{\mathrm{def}}$.
\end{proof}

\begin{lemma}\label{contractededge}
    Let $\elt \in \moduli$ such that there exist two irreducible, potentially non-distinct, components $\mathcal{C}_1,\mathcal{C}_2 \subset \mathcal{C}$ meeting at a node $q$ with neither $s|_{\mathcal{C}_1}, s|_{\mathcal{C}_2}$ identically $0$. Then there exists an infinitesimal smoothing of $q$ in the moduli space $\moduli$.
\end{lemma}

\begin{proof}
    First suppose $\mathcal{C} = \mathcal{C}_1 \cup \mathcal{C}_2$. Then take $\mathcal{C}^{\mathrm{def}}/\Delta$ to be any smoothing of $q$ as marked curves. Denote the sections of this family by $\sigma_i$. Consider the Weil divisor $$D := \sum_{i = 1}^n \tilde{c}_i \sigma_i. $$ Since $\mathcal{C}^{\mathrm{def}}$ is irreducible, in fact $D$ is Cartier (as the $\sigma_i$ live only in the smooth locus) and therefore $D$ gives rise to a line bundle-section pair $$(\mathcal{O}(D), s_D) $$ on $\mathcal{C}^{\mathrm{def}}$. Note that the age of $\mathcal{L}$ at $q$ is trivial: by Proposition \ref{structure} (1), $\mathrm{deg}(\mathcal{L}|_{\mathcal{C}_i}) \geq \alpha_{q,\mathcal{C}_i}(\mathcal{L})/t_q$ for $i = 1,2$ (with both branches contributing if $\mathcal{C}_1 = \mathcal{C}_2$), and if the age at $q$ were non-zero these contributions would sum to $1$, contradicting the degree bounds of Section \ref{setup}. Proposition \ref{structure} part (1) then implies that this deforms $(\mathcal{L},s)$.

    Now for the general case of $\mathcal{C}$ having more components, we choose $\mathcal{C}^{\mathrm{def}}/\Delta$ to smooth $q$ but trivial on the other components. As at the end of Lemma \ref{internaledge}, we can then conclude by applying Lemma \ref{propagation} by gluing the construction in the first part of the proof to the trivial deformation data of the other components, where now $D$ also includes the terms $\alpha_{q_j,\mathcal{C}_1 \cup \mathcal{C}_2}(\mathcal{L})\, \tilde{q}_j$ for the nodes $q_j$ joining $\mathcal{C}_1 \cup \mathcal{C}_2$ to the other components (these are trivial families lying in the smooth locus of the smoothing of $\mathcal{C}_1 \cup \mathcal{C}_2$).
\end{proof}

\begin{corollary}\label{bipartite}

Given any $\elt \in \moduli$ there is an infinitesimal deformation of $\elt$ to an object whose associated mod $r$ tropical type is bipartite with respect to $V = V_0 \coprod V_+$.
    
\end{corollary}

\begin{proof}
    First apply Lemma \ref{internaledge} repeatedly for each purely internal edge $e \in E(\Gamma_+)$ to infinitesimally smooth out such nodes, giving a curve $\mathcal{C}'$. Then applying Lemma \ref{contractededge} repeatedly for each purely external edge $e \in E(\Gamma_0) $ puts us in the setup as in the statement.
\end{proof}

\begin{lemma}\label{genus0}

Let $\elt \in \moduli$ whose associated mod $r$ tropical type $[\tau]$ is of the form \begin{center}

\tikzset{every picture/.style={line width=0.75pt}} 

\begin{tikzpicture}[x=0.75pt,y=0.75pt,yscale=-1,xscale=1]

\draw    (239,188) -- (403,189.95) ;
\draw [shift={(405,189.97)}, rotate = 180.68] [color={rgb, 255:red, 0; green, 0; blue, 0 }  ][line width=0.75]    (10.93,-3.29) .. controls (6.95,-1.4) and (3.31,-0.3) .. (0,0) .. controls (3.31,0.3) and (6.95,1.4) .. (10.93,3.29)   ;
\draw [shift={(239,188)}, rotate = 0.68] [color={rgb, 255:red, 0; green, 0; blue, 0 }  ][fill={rgb, 255:red, 0; green, 0; blue, 0 }  ][line width=0.75]      (0, 0) circle [x radius= 3.35, y radius= 3.35]   ;
\draw    (239,80) -- (335,105.97) ;
\draw [shift={(335,105.97)}, rotate = 15.14] [color={rgb, 255:red, 0; green, 0; blue, 0 }  ][fill={rgb, 255:red, 0; green, 0; blue, 0 }  ][line width=0.75]      (0, 0) circle [x radius= 3.35, y radius= 3.35]   ;
\draw [shift={(239,80)}, rotate = 15.14] [color={rgb, 255:red, 0; green, 0; blue, 0 }  ][fill={rgb, 255:red, 0; green, 0; blue, 0 }  ][line width=0.75]      (0, 0) circle [x radius= 3.35, y radius= 3.35]   ;
\draw    (239,100) -- (335,105.97) ;
\draw [shift={(335,105.97)}, rotate = 3.56] [color={rgb, 255:red, 0; green, 0; blue, 0 }  ][fill={rgb, 255:red, 0; green, 0; blue, 0 }  ][line width=0.75]      (0, 0) circle [x radius= 3.35, y radius= 3.35]   ;
\draw [shift={(239,100)}, rotate = 3.56] [color={rgb, 255:red, 0; green, 0; blue, 0 }  ][fill={rgb, 255:red, 0; green, 0; blue, 0 }  ][line width=0.75]      (0, 0) circle [x radius= 3.35, y radius= 3.35]   ;
\draw    (240,131) -- (335,105.97) ;
\draw [shift={(335,105.97)}, rotate = 345.24] [color={rgb, 255:red, 0; green, 0; blue, 0 }  ][fill={rgb, 255:red, 0; green, 0; blue, 0 }  ][line width=0.75]      (0, 0) circle [x radius= 3.35, y radius= 3.35]   ;
\draw [shift={(240,131)}, rotate = 345.24] [color={rgb, 255:red, 0; green, 0; blue, 0 }  ][fill={rgb, 255:red, 0; green, 0; blue, 0 }  ][line width=0.75]      (0, 0) circle [x radius= 3.35, y radius= 3.35]   ;
\draw    (310,136) -- (310.95,174.97) ;
\draw [shift={(311,176.97)}, rotate = 268.6] [color={rgb, 255:red, 0; green, 0; blue, 0 }  ][line width=0.75]    (10.93,-3.29) .. controls (6.95,-1.4) and (3.31,-0.3) .. (0,0) .. controls (3.31,0.3) and (6.95,1.4) .. (10.93,3.29)   ;
\draw    (239,80) .. controls (279,50) and (296,77.94) .. (335,105.97) ;
\draw    (240,131) .. controls (279,139.97) and (326,136.97) .. (335,105.97) ;

\draw (233,105.4) node [anchor=north west][inner sep=0.75pt]    {$\vdots $};
\draw (211,67.4) node [anchor=north west][inner sep=0.75pt]    {$v_{1}$};
\draw (210,87.4) node [anchor=north west][inner sep=0.75pt]    {$v_{2}$};
\draw (212,121.4) node [anchor=north west][inner sep=0.75pt]    {$v_{k}$};
\draw (342,96.4) node [anchor=north west][inner sep=0.75pt]    {$v_{0}$};
\draw (333,74.4) node [anchor=north west][inner sep=0.75pt]    {$g_{v_{0}} =\ 0$};
\draw (272,71.4) node [anchor=north west][inner sep=0.75pt]    {$\vdots $};
\draw (296,113.4) node [anchor=north west][inner sep=0.75pt]    {$\vdots $};

\end{tikzpicture}
\end{center} i.e. there is precisely one irreducible component $\mathcal{C}_0$ on which $s$ identically vanishes, and every other component only intersects $\mathcal{C}_0$. Then there exists an infinitesimal smoothing of all nodes $q_{i}$ of $\mathcal{C}$ \textbf{simultaneously}.
    
\end{lemma}

\begin{proof}
    Let $N$ denote the number of nodes and let $\mathcal{C}^{\mathrm{def}'} \rightarrow \Delta^N$ be the universal smoothing of $\mathcal{C}$ as a marked curve. For each node $q_i$, write $$\alpha_i := \alpha_{q_i,\mathcal{C}'_i}(\mathcal{L}),$$ where $\mathcal{C}'_i \neq \mathcal{C}_0$ is the external component containing $q_i$. By Proposition \ref{structure} (1), $\alpha_i$ is the order of vanishing of $s|_{\mathcal{C}'_i}$ at $q_i$, and $\alpha_i \geq 1$ since $s$ vanishes identically on $\mathcal{C}_0 \ni q_i$. Define $$\bar{\alpha}_i := \prod_{j \not = i} \alpha_j \in \N.$$ Let $\mathcal{C}^{\mathrm{def}}/\Delta$ be the family of marked curves pulled back from $\mathcal{C}^{\mathrm{def}'}$ via $$\Delta \rightarrow \Delta^N $$ $$\epsilon \mapsto (u_i\epsilon^{\bar{\alpha}_i})_{i=1}^N $$ for constants $u_i \in \C^*$ to be chosen below. Consider the Weil divisor \begin{equation}\label{divisor}
    D:= \sum_{i = 1}^n \tilde{c}_i \sigma_i + (\prod_{i = 1}^N \alpha_i ) \mathcal{C}_0 \subset \mathcal{C}^{\mathrm{def}}.
    \end{equation} The local model of $\mathcal{C}^{\mathrm{def}}$ about $q_i$ is $$\mathcal{C}^{\mathrm{def}} = [\{xy = u_i\epsilon^{\bar{\alpha}_{i}}\}/ \mu_{t_i}] $$ where $$\mathcal{C}'_i = (x = \epsilon = 0), \mathcal{C}_0 = (y = \epsilon = 0).$$ Since the two branches of an $A_{\bar{\alpha}_i - 1}$-singularity meet with intersection number $1/\bar{\alpha}_i$, in this local model $$\mathcal{C}'_i \cdot \mathcal{C}_0 = \frac{1}{\bar{\alpha}_i} \mathrm{deg}(q_i) =
    \frac{1}{\bar{\alpha}_i t_i}$$ and so globally, for each external component $\mathcal{C}' \neq \mathcal{C}_0$, $$(\prod_{i = 1}^N \alpha_i) \mathcal{C}_0 \cdot \mathcal{C}' = \sum_{q_l \in \mathcal{C}'} \alpha_l q_l.$$ Furthermore, note that $D$ in \eqref{divisor} is in fact Cartier. To see this, note that $\tilde{c}_i \sigma_i$ is Cartier for all $i$ since $c_i \in \N$ and $\sigma_i$ lies in the smooth locus of $\mathcal{C}^{\mathrm{def}}$. Also, near $q_i$ the divisor given by $$(y^{\alpha_i} = 0) $$ is a $(\prod_{i = 1}^N \alpha_i) $ thickening of $\mathcal{C}_0$ and therefore $(\prod_{i = 1}^N \alpha_i) \mathcal{C}_0$ is Cartier. We then have an induced line bundle section pair \begin{equation}\label{smoothedpair1}
     (\mathcal{O}(D),s_D)
    \end{equation}  on $\mathcal{C}^{\mathrm{def}}$.
    
    It follows, by Proposition \ref{structure} part (1), that for each external component $\mathcal{C}' \neq \mathcal{C}_0$ we have $$(\mathcal{O}(D),s_D)|_{\mathcal{C}'} \cong (\mathcal{L},s)|_{\mathcal{C}'}.$$ Additionally, on $\mathcal{C}_0$ we have $$D \cdot \mathcal{C}_0 = \sum_{p_i \in \mathcal{C}_0} \tilde{c}_i p_i + (\prod_{i = 1}^N \alpha_i)(- \sum_{\mathcal{C}' \neq \mathcal{C}_0} \mathcal{C}' \cdot \mathcal{C}_0) = \sum_{p_i \in \mathcal{C}_0} \tilde{c}_i p_i - \sum_{i = 1}^N \alpha_i q_i.$$ This follows from the fact that the central fibre $\mathcal{C}$ satisfies $\mathcal{C} \cdot \mathcal{C}_0 = 0$ and that the other components only intersect $\mathcal{C}_0$. Since $\alpha_{q_i,\mathcal{C}_0}(\mathcal{L}) - t_i = -\alpha_i$ and $\mathcal{C}_0$ has genus $0$, Proposition \ref{structure} (2) then shows that $$(\mathcal{O}(D),s_D)|_{\mathcal{C}_0} \cong (\mathcal{L},s)|_{\mathcal{C}_0}.$$ Therefore $(\mathcal{O}(D),s_D)$ \eqref{smoothedpair1} restricted to the central fibre agrees with $(\mathcal{L},s)$ on every irreducible component, so the two pairs can only differ in their gluing data at the nodes $q_i$, which is an element of $(\C^*)^N$. Away from the nodes, $\mathcal{O}(D)$ is generated along $\mathcal{C}_0$ by $\epsilon^{-P}$ for $P := \prod_{i = 1}^N \alpha_i$, whereas near $q_i$ it is generated by $y^{-\alpha_i}$, and in the local model $$\epsilon^{-P} = u_i^{\alpha_i} x^{-\alpha_i} y^{-\alpha_i}.$$ Hence replacing $u_i$ by $c u_i$ changes the gluing of $\mathcal{O}(D)|_{\mathcal{C}}$ at $q_i$ by the factor $c^{\alpha_i}$, without changing the restrictions to the components. We may therefore choose the $u_i$ such that $(\mathcal{O}(D),s_D)|_{\mathcal{C}} \cong (\mathcal{L},s)$, and the result follows.
\end{proof}

\begin{lemma}\label{genus0general}
    Let $\elt \in \moduli$ be of mod $r$ tropical type $[\tau]$ that satisfies the conditions for being essential aside from perhaps the existence of internal genus $0$ vertices. Then there exists an infinitesimal smoothing of $\elt \in \moduli$ smoothing all the internal genus $0$ vertices. In particular, the deformed data is of essential tropical type.
\end{lemma}

\begin{proof}
     Write $\mathcal{C} = \mathcal{C}_1 \cup \mathcal{C}'$, where $\mathcal{C}'$ is the union of the internal components of positive genus and $\mathcal{C}_1$ is the union of all other components. Then $\mathcal{C}_1 \cap \mathcal{C}'$ consists of nodes $q'_j$ lying on external components of $\mathcal{C}_1$, at which $s|_{\mathcal{C}_1}$ vanishes. Since $\Gamma$ is bipartite, every node $q_i$ of $\mathcal{C}_1$ joins an external component to a unique internal genus $0$ component $\mathcal{C}_0^{(l)}$, and the proof of Lemma \ref{genus0} applies verbatim to $(\mathcal{C}_1, (\mathcal{L},s)|_{\mathcal{C}_1})$ after the following changes: the $q'_j$ are treated as additional marked points, $\bar{\alpha}_i$ is the product of the $\alpha_j$ over the other nodes $q_j \in \mathcal{C}_0^{(l)}$, and in \eqref{divisor} each $\mathcal{C}_0^{(l)}$ appears with multiplicity $\prod_{q_i \in \mathcal{C}_0^{(l)}} \alpha_i$. This produces a smoothing $\mathcal{C}_1^{\mathrm{def}}$ of all nodes of $\mathcal{C}_1$, in which the $q'_j$ deform trivially, and a pair $(\tilde{\mathcal{L}}_1,\tilde{s}_1)$ on it deforming $(\mathcal{L},s)|_{\mathcal{C}_1}$ with $\tilde{s}_1$ vanishing along the $\tilde{q}'_j$. Taking the trivial deformation of $(\mathcal{C}', (\mathcal{L}|_{\mathcal{C}'},0))$, Lemma \ref{propagation} glues these to the desired deformation of $\elt$.

\end{proof}

\begin{theorem}\label{essentialtypethm}
    Given any object $\elt \in \moduli$ there exists a $1$-parameter infinitesimal deformation $\mathcal{C}^{\mathrm{def}} \rightarrow \Delta$ with deformed line bundle-section pair $(\tilde{\mathcal{L}},\tilde{s})$ in $\moduli$ such that, over the generic fibre, $(\mathcal{C}^{\mathrm{def}}_{\eta}, (\tilde{\mathcal{L}},\tilde{s}))$ is of essential tropical type. 
\end{theorem}

\begin{proof}
    First apply Lemmas \ref{contractededge},\ref{internaledge} and then apply Lemma \ref{genus0general}.
\end{proof}

\subsubsection{Non-smoothing results }

\begin{definition}
    Denote by $Z_{[\tau]}^+ \subset Z_{[\tau]}^{\circ}$ the locus of objects where, for each internal vertex $v \in V_+$, the associated Jacobian factor at $v$ as given by Proposition \ref{structure} is non-trivial.
\end{definition}

Note that $Z_{[\tau]}^{+} \subset Z_{[\tau]}^{\circ}$ is an open inclusion for any $[\tau]$.

\begin{lemma}\label{opens}

Let $[\tau]$ be an essential mod $r$ tropical type. Then $$ Z_{[\tau]}^{+} \subset \moduli$$ is an open inclusion.

    
\end{lemma}

\begin{proof}
   First suppose $\elt \in \moduli$ has associated mod $r$ tropical type $[\tau]$ as in the statement of Lemma \ref{genus0}, except now suppose that the genus of the internal vertex $v_0$ is $g_{v_0} > 0$. Suppose also that $\elt \in Z_{[\tau]}^{+}$, in other words the Jacobian factor at $v_0$ is non-trivial. Note that $[\tau]$ is indeed an essential type. We first show the statement of the Lemma holds for this particular essential type. More precisely, we show that any infinitesimal deformation $\mathcal{C}^{\mathrm{def}}/\Delta$ of $\elt$ in $\moduli$ has generic fibre $(\mathcal{C}^{\mathrm{def}}_{\eta}, (\tilde{\mathcal{L}}_{\eta}, \tilde{s}_{\eta}))$ whose associated mod $r$ tropical type is also $[\tau]$. Throughout we will use the notation as in the previous lemmas. Suppose first, aiming for a contradiction, that the deformation $\mathcal{C}^{\mathrm{def}}/\Delta$ smooths all the nodes of $\mathcal{C}$. Then, by Proposition \ref{structure} (1) we must have $\tilde{s}_{\eta}$ not identically $0$, cutting out the divisor $$\sum_{i = 1}^n \tilde{c}_i \sigma_i|_{\eta} $$ where the $\sigma_i $ are the marked sections. It follows that  $$\tilde{\mathcal{L}} \cong \mathcal{O}(\sum_{i = 1}^n \tilde{c}_i \sigma_i + \lambda \mathcal{C}_0) $$ for some $\lambda > 0$. Pulling back to $\mathcal{C}_i$
 for $i \not = 0$ gives $$\mathcal{L}|_{\mathcal{C}_i} \cong \mathcal{O}(\sum_{p_i \in \mathcal{C}_i} \tilde{c}_i p_i) \otimes \mathcal{O}(\sum_j \lambda (\mathcal{C}_0 \cdot \mathcal{C}_j)q_j)$$ where the latter sum is over nodes $q_j \in \mathcal{C}_i$. Comparing to the line bundle arising in Proposition \ref{structure} (1) we deduce that \begin{equation}\label{mi}
     \alpha_{q_j, \mathcal{C}_i}(\mathcal{L}) = \lambda (\mathcal{C}_0 \cdot \mathcal{C}_i).
 \end{equation} Denote by $p_1, \cdots, p_k$ the marked points of $\mathcal{C}_0$. Pulling back $\tilde{\mathcal{L}}$ to $\mathcal{C}_0$ gives \begin{equation}\label{pullback}
     \tilde{\mathcal{L}}|_{\mathcal{C}_0} \cong \mathcal{O}(\sum_{i = 1}^k \tilde{c}_i p_i + \lambda \mathcal{C}_0^2) = \mathcal{O}(\sum_{i = 1}^k \tilde{c}_i p_i) \otimes \mathcal{O}( - \lambda \sum_{j = 1}^N (\mathcal{C}_0 \cdot \mathcal{C}_j)q_j). 
 \end{equation} On the other hand, Proposition \ref{structure} (2) tells us that \begin{equation}\label{stdform}
     \tilde{\mathcal{L}}|_{\mathcal{C}_0} \cong \mathcal{L}|_{\mathcal{C}_0} \cong \mathcal{O}( \sum_{i = 1}^k \tilde{c}_i p_i) \otimes \mathcal{O}( \sum_{j = 1}^N -\alpha_{q_j,\mathcal{C}_{i_j}}(\mathcal{L})  q_j) \otimes \mathcal{J}
 \end{equation} for some Jacobian element $\mathcal{J}$, where $\mathcal{C}_{i_j}$ is the unique component containing $q_j$ that is not $\mathcal{C}_0$. Combining \eqref{mi} with \eqref{pullback} and comparing to \eqref{stdform} tells us that $$\mathcal{J} \cong \mathcal{O} $$ which is a contradiction to $\elt \in Z^{+}_{[\tau]}$.

The above shows that any infinitesimal deformation can't smooth all of the nodes at once. To deduce the general case where we assume $\mathcal{C}^{\mathrm{def}}$ smooths only a subset of nodes, observe that if we take any irreducible component $\mathcal{C}_{i,\eta}$ of the general fibre $\mathcal{C}^{\mathrm{def}}_{\eta}$ then the closure of this component in the total family $$\bar{\mathcal{C}}_{i,\eta} \subset \mathcal{C}^{\mathrm{def}}$$ is a sub-family. Indeed, the fibres of $\bar{\mathcal{C}}_{i,\eta}$ over $\Delta$ are all $1$-dimensional and so  $\bar{\mathcal{C}}_{i,\eta} \rightarrow \Delta$ is flat by miracle flatness. Since we are assuming the generic fibre has fewer irreducible components than $\mathcal{C}$, there must exist a component $\bar{\mathcal{C}}_{i,\eta}$ whose closure has central fibre with more than one irreducible component. This central fibre then has associated mod $r$ tropical type as in the first part of the proof, giving a contradiction.

Finally, a deformation which does not smooth any nodes cannot make the section non-zero on an internal component $\mathcal{C}_v$: by Proposition \ref{structure} (2) and Lemma \ref{pushforward} we have $H^0(\mathcal{C}_v,\mathcal{L}|_{\mathcal{C}_v}) = 0$, since the pushforward of $\mathcal{L}|_{\mathcal{C}_v}$ to the coarse space has negative degree if $\mathcal{C}_v$ contains nodes of $\mathcal{C}$ and is the non-trivial Jacobian factor otherwise, and this vanishing is an open condition by semicontinuity.

The above three arguments show that an infinitesimal deformation of an element $\elt \in Z_{[\tau]}^{+}$ in $\moduli$ is also inside $Z_{[\tau]}^{+}$ for any $[\tau]$ essential. Thus the result follows.
   
\end{proof}

\subsubsection{Putting it all together }\label{putting it all together}

To finish the proof of Theorem \ref{irredcpts} it suffices to show that $Z_{[\tau]}$ is irreducible.

\begin{lemma}\label{stratairred}
    Let $[\tau]$ be any mod $r$ tropical type. Then $Z_{[\tau]}$ is irreducible. 
\end{lemma}

\begin{proof}

It is sufficient to show $Z_{[\tau]}^{\circ}$ is irreducible. Consider the forgetful morphism to the interior of the associated stratum in the space of pre-stable twisted curves $$F: Z_{[\tau]}^{\circ} \rightarrow \mathfrak{M}_{[\tau]}^{\circ} $$ given by forgetting the line bundle-section data. Pick any $(\mathcal{C},(\mathcal{L}_0,s_0)) \in  Z_{[\tau]}^{\circ}$. By Proposition \ref{action}, this choice of object acts as a base point for the Jacobian fibre torsor and so we have an isomorphism $$F^{-1}(\mathcal{C}) \cong \mathrm{Jac}_{\mathcal{C}}^{+} \times [\mathbb{G}_m^{b_0(\Gamma_0)}/\mathbb{G}_m]$$ given by $$\elt \cong (\mathcal{C}, \mathcal{J} \cdot (\mathcal{L}_0, s_0)) \mapsto \mathcal{J} $$ and scaling of the section. Thus $F$ is a torsor under a smooth group scheme with connected fibres, so $F$ is smooth, hence open, with irreducible fibres. Since $ \mathfrak{M}_{[\tau]}^{\circ}$ is irreducible, the result follows.
\end{proof}



\begin{proof}[Proof of Theorem \ref{irredcpts}.]
Start off with any $\elt \in \moduli$. Theorem \ref{essentialtypethm} tells us we can infinitesimally deform into a stratum $Z_{[\tau]}^{\circ}$ for $[\tau]$ an essential type. It follows that $$\moduli = \bigcup_{[\tau] \ : \ \mathrm{essential}} Z_{[\tau]}.$$ Now let $[\tau]$ be an inducible essential type. Since $g_v > 0$ for every $v \in V_+$, twisting by $\mathrm{Jac}_{\mathcal{C}}^{+}$ as in Proposition \ref{action} produces objects with non-trivial Jacobian factors, so $Z_{[\tau]}^{+} \neq \emptyset$. By Lemma \ref{stratairred}, $Z_{[\tau]}^{+}$ is irreducible and dense in $Z_{[\tau]}$, being a non-empty open subset of the irreducible $Z_{[\tau]}^{\circ}$, and by Lemma \ref{opens} it is open in $\moduli$. By the general topology statement that if $X \subset Y$ is open and $X$ is irreducible then $\bar{X} \subset Y$ is an irreducible component of $Y$, each such $Z_{[\tau]}$ is an irreducible component of $\moduli$. Distinct essential types give distinct components, since the open sets $Z_{[\tau]}^{+}$ are pairwise disjoint. The inducible condition of Definition \ref{inducibletype} in the statement of the theorem is just so that none of the components considered above are empty. This concludes the proof.
\end{proof}


\begin{example}[Picture of $\moduli$ in genus $1$]\label{picture moduli genus 1}

We draw a picture of two irreducible components of $\moduli$ in genus $1$, showing how they intersect and what a generic point of different strata looks like. Consider the essential tropical type $[\tau]$ as in the second picture of Figure \ref{essentialtypeexamples}, consisting of a single internal vertex of genus $1$ and no edges. The following depicts $Z_{\mathrm{main}}^{\circ} \cup Z_{[\tau]}^{\circ} \subset \moduli$: \begin{center}

\tikzset{every picture/.style={line width=0.75pt}} 

\begin{tikzpicture}[x=0.75pt,y=0.75pt,yscale=-1,xscale=1]

\draw  [fill={rgb, 255:red, 126; green, 211; blue, 33 }  ,fill opacity=1 ] (193,159.98) .. controls (193,132.38) and (245.61,110) .. (310.5,110) .. controls (375.39,110) and (428,132.38) .. (428,159.98) .. controls (428,187.59) and (375.39,209.97) .. (310.5,209.97) .. controls (245.61,209.97) and (193,187.59) .. (193,159.98) -- cycle ;
\draw  [fill={rgb, 255:red, 248; green, 231; blue, 28 }  ,fill opacity=1 ] (301.45,35.97) -- (397,35.97) -- (356.05,143) -- (260.5,143) -- cycle ;
\draw    (301,172.97) ;
\draw [shift={(301,172.97)}, rotate = 0] [color={rgb, 255:red, 0; green, 0; blue, 0 }  ][fill={rgb, 255:red, 0; green, 0; blue, 0 }  ][line width=0.75]      (0, 0) circle [x radius= 3.35, y radius= 3.35]   ;
\draw    (161,180.97) .. controls (200.6,151.27) and (250,205.89) .. (289.8,177.87) ;
\draw [shift={(291,177)}, rotate = 143.13] [color={rgb, 255:red, 0; green, 0; blue, 0 }  ][line width=0.75]    (10.93,-3.29) .. controls (6.95,-1.4) and (3.31,-0.3) .. (0,0) .. controls (3.31,0.3) and (6.95,1.4) .. (10.93,3.29)   ;
\draw    (320,141.97) ;
\draw [shift={(320,141.97)}, rotate = 0] [color={rgb, 255:red, 0; green, 0; blue, 0 }  ][fill={rgb, 255:red, 0; green, 0; blue, 0 }  ][line width=0.75]      (0, 0) circle [x radius= 3.35, y radius= 3.35]   ;
\draw    (503,138.97) .. controls (430.73,110.26) and (360.42,176.62) .. (321.18,143.02) ;
\draw [shift={(320,141.97)}, rotate = 42.71] [color={rgb, 255:red, 0; green, 0; blue, 0 }  ][line width=0.75]    (10.93,-3.29) .. controls (6.95,-1.4) and (3.31,-0.3) .. (0,0) .. controls (3.31,0.3) and (6.95,1.4) .. (10.93,3.29)   ;
\draw    (339,80.97) ;
\draw [shift={(339,80.97)}, rotate = 0] [color={rgb, 255:red, 0; green, 0; blue, 0 }  ][fill={rgb, 255:red, 0; green, 0; blue, 0 }  ][line width=0.75]      (0, 0) circle [x radius= 3.35, y radius= 3.35]   ;
\draw    (211,67.97) .. controls (250.6,38.27) and (298.04,109.52) .. (337.8,81.84) ;
\draw [shift={(339,80.97)}, rotate = 143.13] [color={rgb, 255:red, 0; green, 0; blue, 0 }  ][line width=0.75]    (10.93,-3.29) .. controls (6.95,-1.4) and (3.31,-0.3) .. (0,0) .. controls (3.31,0.3) and (6.95,1.4) .. (10.93,3.29)   ;

\draw (57,180.4) node [anchor=north west][inner sep=0.75pt]    {$\left( E,\left(\mathcal{O}\left(\sum _{i\ =\ 1}^{n}\tilde{c}_{i} p_{i}\right) ,s\right)\right) \ $};
\draw (448,141.4) node [anchor=north west][inner sep=0.75pt]    {$\left( E,\left(\mathcal{O}\left(\sum _{i\ =\ 1}^{n}\tilde{c}_{i} p_{i}\right) ,0\right)\right) \ $};
\draw (32,35.4) node [anchor=north west][inner sep=0.75pt]    {$\left( E,\left(\mathcal{O}\left(\sum _{i\ =\ 1}^{n}\tilde{c}_{i} p_{i}\right) \otimes \mathcal{J} ,0\right)\right) \ $};
\draw (217,141.4) node [anchor=north west][inner sep=0.75pt]    {$Z\mathrm{_{\mathrm{m} ain}^{\circ }}$};
\draw (342,42.4) node [anchor=north west][inner sep=0.75pt]    {$Z\mathrm{_{[ \tau ]}^{\circ }}$};

\end{tikzpicture}
\end{center}

A generic point of $Z_{\mathrm{main}}$ consists of a smooth genus $1$ curve $E$, with the canonical line bundle given by Proposition \ref{structure} (1) and canonical non-zero section. We can deform this data to a generic point of $Z_{\mathrm{main}} \cap Z_{[\tau]}$ given by deforming the section to zero. Then we may bubble off the main component into $Z_{[\tau]}$ by twisting by the Jacobian of $E$. In this setting it turns out both components are of dimension $n$ with $\mathrm{dim}(Z_{\mathrm{main}} \cap Z_{[\tau]}) = n - 1.$ This will follow from the dimension calculations of strata in the next section.
    
\end{example} \qed

\subsubsection{Stratification of $\moduli$}

\

We first formalise some basic definitions pertaining to stratified spaces. 

\begin{definition}
    Let $X$ be a topological space and $X = \coprod_{i \in J} X_i$ a decomposition of $X$ into disjoint locally closed subspaces $X_i \subset X$, and $J$ a poset. \begin{itemize}
        \item We say this data forms a \textbf{weak stratification} of $X$ if, for each $j \in J$, we have $\overline{X}_j \subset \bigcup_{i \leq j}X_i$.

        \item We say this data forms a \textbf{stratification} of $X$ if, for each $j \in J$, we have $\overline{X}_j = \bigcup_{i \leq j}X_i$.
    \end{itemize}
\end{definition}

Note that, by construction, the distinguished loci of Section \ref{strata} form a weak stratification. However they do not form a stratification in the above sense. For example, in Example \ref{picture moduli genus 1} we saw that the main component $Z_{\mathrm{main}}$ is not a union of open strata $Z_{[\tau_i]}^{\circ}$ in genus $1$. More precisely, the boundary locus of $Z_{\mathrm{main}}$ consisting of smooth genus $1$ curves and the $0$-section is not a distinguished locus---the associated mod $r$ tropical type defines an irreducible component distinct from $Z_{\mathrm{main}}$.

To fix this, we introduce some more refined loci of $\moduli$ given by the extra data of a subgraph of domain graph $\Gamma$.

\begin{definition}
    Let $[\tau]$ be a mod $r$ tropical type and $\Gamma' \subset \Gamma_+$ a choice of subgraph. We then define the subspace $$Z_{[\tau],\Gamma'}^{\circ} \subset Z_{[\tau]}^{\circ}$$ to be the locus of objects $\elt \in Z_{[\tau]}^{\circ}$ for which the corresponding Jacobian factor to $\mathcal{L}|_{\mathcal{C}_{\Gamma'}}$ as given by Proposition \ref{structure} (2) $$\mathcal{J}_{\Gamma'} \cong \mathcal{O}$$ is trivial.
\end{definition}

\begin{remark}
    If $\Gamma'$ is genus $0$ then $Z_{[\tau], \Gamma'}^{\circ} = Z_{[\tau], \emptyset}^{\circ} = Z_{[\tau]}^{\circ}$. Similarly, for any choice of $\Gamma'$, if $\Gamma' \subset \Gamma'' \subset \Gamma_+$ with $g(\Gamma') = g(\Gamma'')$ then $Z_{[\tau], \Gamma'}^{\circ} = Z_{[\tau], \Gamma''}^{\circ}$. Therefore, from now on, we only ever consider choices of $\Gamma'$ which are minimal with respect to this property.
\end{remark}

\begin{example}
    Let $[\tau]$ be as in Example \ref{picture moduli genus 1} and let $\Gamma' = \Gamma$. Then $Z_{[\tau], \Gamma'}^{\circ} $ is the locus of objects $(C, (\mathcal{O}(\sum_{i=1}^n \tilde{c}_i p_i),0))$ for $C \in \mathcal{M}_{1,n}$. In particular we have $$Z_{\mathrm{main}} \cap Z_{[\tau]}^{\circ} = Z_{[\tau], \Gamma'}^{\circ}. $$
\end{example}

In forthcoming work with Sam Johnston \cite{CJ} we prove a structural result for general intersections of distinguished loci, generalising Proposition \ref{action}. For a finite set $I$ of mod $r$ tropical types, write $Z_I := \bigcap_{[\tau] \in I} Z_{[\tau]}$.

\begin{proposition}[\cite{CJ}]\label{intersection of strata structure}
   There is a dense open locus $Z_{I}^{\circ} \subset Z_I$ with the following properties: \begin{itemize}
       \item The domain curves of objects in $Z_{I}^{\circ}$ all have the same dual graph.
       \item There exists a sub-semi-abelian variety $$A_I \subset \mathrm{Jac}_{[\tau_i]} $$ for all $i \in I$ with an action $$A_I \curvearrowright Z_{I}^{\circ}.$$ Furthermore, this action acts on fibres of $Z_{I}^{\circ} \rightarrow \mathfrak{M}_{g,n}$ and is transitive on such fibres, up to the further action of scaling sections.

       \item In the case $I = \{ [\tau_1], [\tau_2] \}$ there exists a subgraph $\Gamma' \subset \Gamma$ and a mod $r$ tropical type $[\tau']$ such that the dense open locus constructed is $$Z_{[\tau'], \Gamma'}^{\circ} = Z_I^{\circ}.$$
   \end{itemize}
\end{proposition}

An immediate consequence of the above proposition is the following:

\begin{corollary}[\cite{CJ}]
    The $\{Z_{[\tau],\Gamma'}^{\circ}\}$ form a stratification of $\moduli$. 
\end{corollary}

Furthermore, Proposition \ref{intersection of strata structure} tells us the poset structure for this stratification. The intuition is that intersecting strata involves both: \begin{itemize}
    \item Passing to deeper boundary strata of $\mathrm{Log}_{\Lambda}(\mathcal{A}|\mathcal{D})$.

    \item Passing to a sub-abelian torsor over this stratum.
\end{itemize}

\begin{proposition}
    $Z_{[\tau_1], \Gamma_1'}^{\circ} \subset Z_{[\tau_2], \Gamma_2'}$ if and only if the following combinatorial conditions hold: There exists a specialisation of mod $r$ tropical types $[\tau_1] \leadsto [\tau_2]$ such that

    \begin{enumerate}
        \item if we denote $G^{+0} \subset \Gamma_{1,+}$ the largest subgraph that is specialised to a subgraph of $\Gamma_{2,0}$, then we must have $$G^{+0} \subset \Gamma_1'. $$

        \item The image of $\Gamma_1'$ in $\Gamma_2$ under the specialisation contains $\Gamma_2'$. 
    \end{enumerate}
    
\end{proposition}

\begin{proof}
    Suppose first we have a specialisation $[\tau_1] \leadsto [\tau_2]$ as in the statement of the Proposition. We need to show that for each object $\elt \in Z_{[\tau_1], \Gamma_1'}^{\circ}$ there exists an infinitesimal deformation $\defelt$ of $\elt$ with generic fibre $\defeltgen \in Z_{[\tau_2], \Gamma_2'}$. Firstly, pick a deformation $\mathcal{C}^{\mathrm{def}}$ of the curve $\mathcal{C}$ such that the generic fibre $\mathcal{C}^{\mathrm{def}}_{\eta}$ has dual graph $\Gamma_2$ and the induced specialisation $\Gamma_1 \leadsto \Gamma_2$ sends $\Gamma_1'$ to a subgraph containing $\Gamma_2'$. We then need to find $(\tilde{\mathcal{L}}, \tilde{s})$ on $\mathcal{C}^{\mathrm{def}}$ with the desired properties. To do this, we will just try to run the smoothing results of Section \ref{smoothing results}. The only potential obstruction to these smoothing results is as follows:  $\mathcal{L}$ restricts to a sub-curve given by subgraph of $G^{+0}$ with corresponding Jacobian factor non-trivial. Then the proof of Lemma \ref{opens} tells us this Jacobian factor is an obstruction. Conversely, if the corresponding Jacobian factor on the sub-curve corresponding to $G^{+0}$ is trivial the proof of Lemma \ref{genus0} works as written and thus we can find a $(\tilde{\mathcal{L}}, \tilde{s})$ extending $(\mathcal{L},s)$. Thus Condition (1) implies we can infinitesimally deform $\elt$ into $Z_{[\tau_2]}$. We now need to check the generic fibre object $\defeltgen$ lives inside $Z_{[\tau_2], \Gamma_2'}$ specifically. This is where condition (2) comes into play. Namely, condition $(2)$ is equivalent to the corresponding Jacobian factor of $\tilde{\mathcal{L}}_{\eta}|_{\mathcal{C}^{\mathrm{def}}_{\eta, \Gamma_2'}}$ being trivial. This is precisely the defining condition for an object of $Z_{[\tau_2]}$ to lie in the sub-locus $Z_{[\tau_2], \Gamma_2'}$. The above argument is completely reversible, hence the result follows.
\end{proof}

The proof of Proposition \ref{intersection of strata structure} is a slight generalisation of the above argument, yet the key observation remains the same: obstructions to deformations are precisely encoded in the Jacobians of the curves exchanged between the interior and exterior under specialisation.

\section{Polynomiality}\label{polynomiality}

We now explain the behaviour of the moduli spaces $\moduli$ as we change the rooting parameter $r$. In particular we give polynomial properties of the strata.

\subsection{Induced contact data }\label{induced contact data}

We explain how, given contact data $\Lambda$ for rooting parameter $r$, there is induced contact data $\Lambda'$ for rooting parameter $\lambda r$ for $\lambda \in \N$. This construction is given in \cite[Section 3.5]{BNR22}, but we highlight the main points again here in the language of Section \ref{setup}. Indeed, we take $\Lambda'$ to consist of the data: \begin{itemize}
    \item The source rooting parameters $s_i'$ are defined by $$s_i' := \frac{\lambda r}{\gcd(\lambda r, c_i)}, $$ which equals $\lambda s_i$ whenever $c_i | r$.
    \item The ages $a_i'$ are defined by $$a_i' := \frac{a_i}{\lambda}.$$
    \item The coarse degree remains unchanged from $\Lambda$.
    \item The genus remains unchanged from $\Lambda$.
\end{itemize}

Note that this implies that the new coarse contact orders $c_i'$ equal the old coarse contact orders $c_i$ \eqref{coursecontacts}. From now on we abuse notation and denote this lifted contact data also by $\Lambda$ since we just need to make one choice for a single parameter $r$, dropping the primes.

\begin{convention}
    In what follows, we will fix $r$ to be an auxiliary rooting parameter, while we change $\lambda$ to get a family of rooting parameters $\lambda r$.  
\end{convention}

\subsection{Comparison maps}\label{comparison maps}

Consider two different choices of target rooting parameters $\lambda r$ and $r$ for $r,\lambda \in \N$. Then there is a natural \textbf{comparison morphism} $$\pi_{\lambda r, r}: \mathrm{Orb}_{\Lambda}(\mathcal{A}_{\lambda r}) \rightarrow \mathrm{Orb}_{\Lambda}(\mathcal{A}_{r})$$ induced by taking the $\lambda$-th tensor power of line bundle-section pairs $$(\mathcal{C}, (\mathcal{L},s)) \mapsto (\mathcal{C}',(\mathcal{L}',s'))$$ where $(\mathcal{L}',s')$ is the uniquely defined line bundle-section pair up to isomorphism that pulls back to $(\mathcal{L},s)^{\otimes \lambda}$ under $\mathcal{C} \rightarrow \mathcal{C}'$. Here $\mathcal{C}'$ is the \textit{partial coarsening} of $\mathcal{C}$, which is needed in order for the resulting data to arise from a representable morphism. The partial coarsening in a similar context is discussed in the proof of \cite[Theorem 2.9]{BNR22}. More precisely, the partial coarsening is taken as follows. Let $p \in \mathcal{C}$ be a special point with isotropy $\mu_{t(\lambda \cdot r)}$. The line bundle $\mathcal{L}$ on $\mathcal{C}$ arising from a representable map $\mathcal{C} \rightarrow \mathcal{A}_{\lambda r}$ induces an injective homomorphism $$\iota: \mu_{t(\lambda \cdot r)} \hookrightarrow \mu_{\lambda r}.$$ Define a new group $\mu_{t(r)}$ as $$\mu_{t(r)} := (\otimes \lambda)(\mathrm{im}(\iota) ) \cong \mu_{t(\lambda r)}/(\ker(\otimes \lambda) \cap \mu_{t(\lambda r)}) \cong \mu_{t(\lambda r)}/\mu_{\gcd(\lambda, t(\lambda r))} \cong \mu_{t(\lambda r)/\gcd(\lambda, t(\lambda r))}.$$ We then take the partial coarsening $\mathcal{C} \rightarrow \mathcal{C}'$ such that the isotropy of the image of $p$ is $\mu_{t(r)}$, over each $p \in \mathcal{C}$. Observe that there is a commutative diagram  \begin{center}\begin{equation}\label{coarsening homs}
    \begin{tikzcd}
 \mu_{t(\lambda \cdot r)} \arrow[d, "{\otimes \gcd(\lambda, t(\lambda \cdot r))}"', two heads] \arrow[r, hook] & \mu_{\lambda \cdot r} \arrow[d, "\otimes \lambda", two heads] \\
\mu_{t(r)} \arrow[r, hook]                                                                                     & \mu_{r}                                                      
\end{tikzcd} \end{equation}
\end{center}  In particular, the bottom injection allows the descended line bundle-section data to $\mathcal{C}'$ to come from a representable morphism to $\mathcal{A}_r$.

\subsection{Degree calculations }

The goal of this subsection will be to show the following result, calculating the degree of the comparison morphisms on strata.

\begin{theorem}\label{degrees}
    Let $[\tau] = (\underline{\tau},w)$ be a mod $\lambda r$ tropical type given by pre-tropical type $\underline{\tau}$ with $\lambda r$-weighting $w \in W_{\underline{\tau}, \lambda r}$. Then \begin{enumerate}
        \item $\pi_{\lambda r, r}$ restricts to a map $$Z_{(\underline{\tau},w)} \rightarrow Z_{(\underline{\tau},\lambda \cdot w)}. $$
        \item Furthermore, the map $Z_{(\underline{\tau},w)} \rightarrow Z_{(\underline{\tau},\lambda \cdot w)}$ is representable by DM stacks, with $0$ dimensional fibres of finite degree $$ d_{[\tau],\lambda r, r} := \lambda^{|E^{\mathrm{b}}| - b_0(\Gamma_+) + b_1(\Gamma_+) + 2\sum_{v \in V_+} g_v } \cdot \prod_{e \in E} \frac{t_e(r)}{t_e(\lambda r)}, $$ where $t_e(\lambda r)$ and $t_e(r)$ are the orders of the isotropy groups at the node corresponding to $e$ for the types $(\underline{\tau},w)$ and $(\underline{\tau},\lambda \cdot w)$ respectively.
    \end{enumerate}

\end{theorem}\label{Comparisonstrata}

The theorem will follow from a sequence of results analysing the isomorphism classes and automorphisms of objects in a generic fibre. 

\begin{convention}
    Earlier we used the notation $t_e$ for the order of the isotropy group at a node corresponding to edge $e \in E$. Since now we are considering families of types, we will write $t_e(r)$ for the value of this isotropy group when working with rooting parameter $r$. When it is clear from context, i.e. there is one particular edge $e$ in question, we will drop the subscript $e$ and just write $t(r)$.
\end{convention}

The following Lemma will allow us to abuse notation and denote the comparison morphism as $\elt \mapsto (\mathcal{C}',(\mathcal{L},s)^{\otimes \lambda})$ since the relevant numerical invariants of $(\mathcal{L}',s')$ agree with those of $(\mathcal{L},s)^{\otimes \lambda}$.

 \begin{lemma}[\cite{Cru24} Lemma 4.3.2]\label{alphas pullback}
    Let $q \in \mathcal{C}$ be a node mapping to $q'$ under the partial coarsening. Then the coarse slopes of Definition \ref{orbitroptype} at $q$ and $q'$ agree modulo $r$: $$\frac{\lambda r}{t_q}\alpha_{q,1}(\mathcal{L}) \equiv \frac{r}{t_{q'}}\alpha_{q',1}(\mathcal{L}') \mod r. $$
\end{lemma}

\begin{lemma}\label{Jactau torsion size}
    The size of the $\lambda$-torsion of $\mathrm{Jac}_{\mathcal{C}}^{+}$ is $$|\mathrm{Jac}_{\mathcal{C}}^{+}[\lambda]| = \lambda^{|E^{\mathrm{b}}| - b_0(\Gamma^{\dag}) + 1 + b_1(\Gamma_+) + 2\sum_{v \in V_+} g_v }. $$
\end{lemma}

\begin{proof}
    Using the short exact sequence \eqref{SESJac} we see $$|\mathrm{Jac}_{\mathcal{C}}^{+}[\lambda]| = |\mu_{\lambda}^{|E^{\mathrm{b}}| - b_0(\Gamma^{\dag}) +1}| \cdot |\mathrm{Jac}(C_+)[\lambda]|.$$ The first term is $\lambda^{|E^{\mathrm{b}}| - b_0(\Gamma^{\dag}) +1}$ and the second term is $\lambda^{b_1(\Gamma_+) + 2\sum_{v \in V_+} g_v}$, hence the result follows.
\end{proof}

\begin{lemma}\label{part (a)}
    $\pi_{\lambda r, r}$ restricts to a map $$Z_{(\underline{\tau},w)} \rightarrow Z_{(\underline{\tau},\lambda \cdot w)}. $$
\end{lemma}

\begin{proof}
        It is clear that $\pi_{\lambda r, r}$ preserves the underlying pre-tropical type. The fact the associated $r$-weighting of $(\mathcal{C}',(\mathcal{L},s)^{\otimes \lambda})$ is $\lambda \cdot w$ follows from the fact $$\mathrm{age}(\mathcal{L}^{\otimes \lambda}) \equiv \lambda \mathrm{age}(\mathcal{L}) \mod 1$$ since it arises from the $\lambda$-tensor product of representations, along with Lemma \ref{alphas pullback}.
\end{proof}

\begin{remark}
    Lemma \ref{part (a)} is compatible with the identifications in Lemma \ref{r-weightings}. More precisely, picking a maximal tree $T \subset \Gamma$ gives a commutative square \begin{center}
        \begin{tikzcd}
{W_{\underline{\tau},\lambda r}} \arrow[r, "\text{Lemma \ref{part (a)}}", two heads] \arrow[d, "\text{Lemma \ref{r-weightings}}"', no head, Rightarrow, no head] & {W_{\underline{\tau},r}} \arrow[d, "\text{Lemma \ref{r-weightings}}", no head, Rightarrow, no head] \\
(\Z/\lambda r\Z)^{b_1(\Gamma)} \arrow[r, "\mathrm{mod}\ r", two heads]                                                        & (\Z/r\Z)^{b_1(\Gamma)}
\end{tikzcd}
    \end{center}
\end{remark}

\begin{example}
The following is an example of the comparison map for a boundary stratum in genus $0$. Consider the genus $0$ mod $r$ tropical type \begin{center}

\tikzset{every picture/.style={line width=0.75pt}} 

\begin{tikzpicture}[x=0.75pt,y=0.75pt,yscale=-1,xscale=1]

\draw    (255,106) -- (390,105.97) ;
\draw [shift={(390,105.97)}, rotate = 359.99] [color={rgb, 255:red, 0; green, 0; blue, 0 }  ][fill={rgb, 255:red, 0; green, 0; blue, 0 }  ][line width=0.75]      (0, 0) circle [x radius= 3.35, y radius= 3.35]   ;
\draw [shift={(255,106)}, rotate = 359.99] [color={rgb, 255:red, 0; green, 0; blue, 0 }  ][fill={rgb, 255:red, 0; green, 0; blue, 0 }  ][line width=0.75]      (0, 0) circle [x radius= 3.35, y radius= 3.35]   ;
\draw    (256,146) -- (388,145.97) ;
\draw [shift={(391,145.97)}, rotate = 179.99] [fill={rgb, 255:red, 0; green, 0; blue, 0 }  ][line width=0.08]  [draw opacity=0] (8.93,-4.29) -- (0,0) -- (8.93,4.29) -- cycle    ;
\draw [shift={(256,146)}, rotate = 359.99] [color={rgb, 255:red, 0; green, 0; blue, 0 }  ][fill={rgb, 255:red, 0; green, 0; blue, 0 }  ][line width=0.75]      (0, 0) circle [x radius= 3.35, y radius= 3.35]   ;
\draw    (324,113.97) -- (324,136.97) ;
\draw [shift={(324,139.97)}, rotate = 270] [fill={rgb, 255:red, 0; green, 0; blue, 0 }  ][line width=0.08]  [draw opacity=0] (8.93,-4.29) -- (0,0) -- (8.93,4.29) -- cycle    ;

\draw (242,81.4) node [anchor=north west][inner sep=0.75pt]    {$d_{1}$};
\draw (386,81.4) node [anchor=north west][inner sep=0.75pt]    {$d_{2}$};
\draw (309,81.4) node [anchor=north west][inner sep=0.75pt]    {$[ m_{e}]$};

\end{tikzpicture}
\end{center} Suppose $r > d_1$. Then the balancing condition at the external vertex is $$d_1 \equiv m_e \mod r \Leftrightarrow \frac{d_1}{r} = \mathrm{age}_{q}(\mathcal{L}'|_{\mathcal{C}_0'}) $$ where $\mathcal{L}'$ is a line bundle arising for objects in the corresponding stratum of $\moduli$. Now suppose $[\tau]$ is a mod $\lambda r$-tropical type that maps to $[\tau']$ under $\pi_{\lambda r, r}$. By Lemma \ref{part (a)} $\underline{\tau} = \underline{\tau}'$ and, again by balancing at the external vertex, we must have $$\frac{d_1}{\lambda r} = \mathrm{age}_q(\mathcal{L}|_{\mathcal{C}_0}).$$ In particular, the age of the line bundle and hence the isotropy of $\mathcal{C}$ is uniquely determined. This is an artifact of there being a unique $\lambda r$-weighting for all $\lambda$ since $b_1(\Gamma) = 0$. 

Observe that the isotropy order at the node of $\mathcal{C}$ is $$t(\lambda r) = \frac{\lambda r}{\gcd(\lambda r, m_e)}$$ by coprimality \eqref{noderooting}. If we suppose additionally that $$m_e | r$$ then $$t(\lambda r) = \frac{\lambda r}{m_e} = \lambda \cdot t(r)$$ which is a $\lambda$-scaling of the isotropy at the node of $\mathcal{C}'$. We will see later that both $Z_{[\tau]} \subset \mathrm{Orb}_{\Lambda}(\mathcal{A}_{\lambda r})$ and $ Z_{[\tau']} \subset \moduli$ are codimension $1$. The comparison map $\pi_{\lambda r, r} : \mathrm{Orb}_{\Lambda}(\mathcal{A}_{\lambda r}) \rightarrow \moduli$ is then locally a $\lambda$th root stack construction along $Z_{[\tau']}$.
\end{example} \qed


\begin{lemma}\label{skeleton}
    Consider a fibre of the comparison morphism $F$ over a $\mathrm{Spec}(\C)$ point $(\mathcal{C}',(\mathcal{F},t)) \in Z_{(\underline{\tau},\lambda \cdot w)}^{\circ}$. In other words, define $F$ by the fibre square \begin{center}
    \begin{tikzcd}
F \arrow[d] \arrow[r]               & {Z_{(\underline{\tau}, w)}^{\circ}} \arrow[d]    \\
\mathrm{Spec}(\mathbb{C}) \arrow[r] & {Z_{(\underline{\tau},\lambda \cdot w)}^{\circ}}
\end{tikzcd}
\end{center} Then the domain curve $\mathcal{C}$ of an object of $F$ is unique up to isomorphism. Let $F_{\mathcal{C}}$ be the groupoid of objects of $F$ with domain curve $\mathcal{C}$, with morphisms those isomorphisms in $F$ which are the identity on $\mathcal{C}$. Then isomorphism classes of objects in $F_{\mathcal{C}}$ are a torsor under $$\mathrm{Jac}_{\mathcal{C}}^{+}[\lambda] \times (\mu_{\lambda}^{b_0(\Gamma_0)}/\mu_{\lambda})$$ (the second factor being trivial if $V = V_+$), and each object of $F_{\mathcal{C}}$ has automorphism group $\mu_{\lambda}^{\epsilon}$, where $\epsilon = 1$ if $V = V_+$ and $\epsilon = 0$ otherwise.
\end{lemma}

\begin{proof}
    $F$ is the groupoid whose objects are the data of \begin{itemize}
    \item $\elt \in Z_{[\tau]}^{\circ}$.
    \item An isomorphism $(\mathcal{C}',(\mathcal{L},s)^{\otimes \lambda}) \cong (\mathcal{C}',(\mathcal{F},t))$. 
\end{itemize} Pick any object $\elt \in F$ (where we implicitly have a recorded isomorphism $(\mathcal{C}',(\mathcal{L},s)^{\otimes \lambda }) \cong (\mathcal{C}',(\mathcal{F},t))$). Let $\mathcal{J} \in \mathrm{Jac}_{\mathcal{C}}^{+}[\lambda]$. Then, after picking an isomorphism $(\mathcal{L}\otimes \mathcal{J})^{\otimes \lambda} \cong \mathcal{L}^{\otimes \lambda}$, we may construct a new object in $F$ denoted by $$(\mathcal{C}, \mathcal{J} \cdot (\mathcal{L} , s))$$ which is defined via the action on the sections as in Proposition \ref{action}.

 As we vary over non-isomorphic $\mathcal{J}$ we obtain non-isomorphic objects of $F_{\mathcal{C}}$, since isomorphic pairs on $\mathcal{C}$ have isomorphic line bundles.

We now show that every object of $F$ arises in this way, up to rescaling the section. To see this, we first show that the domain curves $\mathcal{C}$ are determined up to isomorphism. We know all such domains $\mathcal{C}$ have the same coarse space as $\mathcal{C}'$, thus they are determined by their isotropy groups. Write $t_{e,w}(\elt)$ for the isotropy of a given $\mathcal{C}$ at a node corresponding to edge $e \in E$. Recall the coprimality condition \eqref{noderooting} $$t_{e,w}(\mathcal{C}, (\mathcal{L},s)) = \frac{\lambda r}{\gcd(\lambda r, m_{\vec{e}})},$$ where $m_{\vec{e}}$ is any lift of the coarse slope $[m_{\vec{e}}] \in \Z/\lambda r \Z$. The $[m_{\vec{e}}]$ are a part of the prescribed data of $[\tau]$, and thus $\gcd(\lambda r, m_{\vec{e}})$ is uniquely determined by $[\tau]$. Hence the $t_e$ are determined by $[\tau]$. The same argument works for the marked points. It follows that the curve $\mathcal{C}$ arising as the domain curve of objects in $F$ is unique up to isomorphism.

After fixing the domain curve $\mathcal{C}$, the only remaining freedom is in the line bundle-section pairs. By Proposition \ref{action} the line bundles must be of the form above. As for the sections, on each connected component of $\mathcal{C}_0$ the section $s$ is determined by $s^{\otimes \lambda}$ up to scaling by $\mu_{\lambda}$, and a simultaneous scaling by $\zeta \in \mu_{\lambda}$ is induced by the automorphism $\zeta$ of $\mathcal{L}$, which is compatible with the identification of $\lambda$-th powers. This gives the factor $\mu_{\lambda}^{b_0(\Gamma_0)}/\mu_{\lambda}$. Finally, an automorphism of an object of $F_{\mathcal{C}}$ is a scalar $\zeta \in \mu_{\lambda}$ with $\zeta s = s$, which forces $\zeta = 1$ unless $s \equiv 0$, i.e. unless $V = V_+$. This concludes the Lemma. 
\end{proof}

 The fibre $F$ is obtained from $F_{\mathcal{C}}$ by taking into account the automorphisms of the curve $\mathcal{C}$.

\begin{lemma}\label{ghost quotient}
    Let $K := \ker(\mathrm{Aut}(\mathcal{C}/C) \rightarrow \mathrm{Aut}(\mathcal{C}'/C))$ be the group of ghost automorphisms of $\mathcal{C}$ which induce the identity on the partial coarsening $\mathcal{C}'$. Then $$K \cong \prod_{e \in E} \mu_{\gcd(\lambda, t_e(\lambda r))}, \quad |K| = \prod_{e \in E} \frac{t_e(\lambda r)}{t_e(r)},$$ the group $K$ acts on $F_{\mathcal{C}}$ by pullback, and $F \cong [F_{\mathcal{C}}/K]$.
\end{lemma}

\begin{proof}
    The group of ghost automorphisms of a twisted curve is the product of the isotropy groups of its nodes \cite[Section 7]{ACV03}, and by the local description of the partial coarsening the induced map to the ghost automorphisms of $\mathcal{C}'$ is the product over $e \in E$ of the left vertical maps of diagram \eqref{coarsening homs}. This gives the description of $K$. By Lemma \ref{skeleton} every object of $F$ is isomorphic to an object of $F_{\mathcal{C}}$; here we use that the isomorphism of domain curves may be chosen compatibly with the identifications of the partial coarsenings with $\mathcal{C}'$, since $\mathrm{Aut}(\mathcal{C}/C) \rightarrow \mathrm{Aut}(\mathcal{C}'/C)$ is surjective. A morphism in $F$ between two objects of $F_{\mathcal{C}}$ is an automorphism $\varphi$ of $\mathcal{C}$ together with an isomorphism from the first pair to the pullback of the second under $\varphi$, such that the induced isomorphism of the images in $Z_{(\underline{\tau},\lambda \cdot w)}^{\circ}$ is compatible with the identifications with $(\mathcal{C}',(\mathcal{F},t))$. In particular $\varphi$ induces the identity on $\mathcal{C}'$, so $\varphi \in K$, and these are precisely the morphisms of the quotient groupoid $[F_{\mathcal{C}}/K]$.
\end{proof}

\begin{remark}
    The action of $K$ on isomorphism classes of $F_{\mathcal{C}}$ is non-trivial in general. Indeed, by the local description of ghost automorphisms, pulling back along $\varphi \in K$ changes the gluing of $\mathcal{L}$ at each node by a root of unity, so that $\varphi^{*}\mathcal{L} \cong \mathcal{L} \otimes \mathcal{J}_{\varphi}$ for a torsion line bundle $\mathcal{J}_{\varphi}$ in the non-compact part of $\mathrm{Jac}(\mathcal{C})$, which can be non-trivial when $b_1(\Gamma) > 0$. For example, if $\Gamma$ is a cycle with two edges then $\mathcal{J}_{\varphi} = \zeta_1 \zeta_2^{-1} \in \mathbb{G}_m \subset \mathrm{Jac}(\mathcal{C})$, where $\zeta_i \in \mu_{t_i}$ is the scalar by which the component of $\varphi$ at the $i$-th node acts on the fibre of $\mathcal{L}$.
\end{remark}

\begin{proof}[Proof of Theorem \ref{degrees}.]

    Part (1) is given by Lemma \ref{part (a)}. To show part (2), to calculate the degree it is sufficient to consider the fibre in the open locus $Z_{(\underline{\tau},w)}^{\circ}$ over a $\mathrm{Spec}(\mathbb{C})$ point. Pick a morphism $\mathrm{Spec}(\mathbb{C}) \rightarrow Z_{(\underline{\tau},\lambda \cdot w)}^{\circ}$ corresponding to a fixed object $(\mathcal{C}',(\mathcal{F},t)) \in Z_{(\underline{\tau},\lambda \cdot w)}^{\circ}$. Lemma \ref{skeleton} tells us that $F_{\mathcal{C}}$ has $|\mathrm{Jac}_{\mathcal{C}}^{+}[\lambda]| \cdot \lambda^{b_0(\Gamma_0) - 1 + \epsilon}$ isomorphism classes of objects, each with automorphism group of order $\lambda^{\epsilon}$. By Lemma \ref{ghost quotient} we have $F \cong [F_{\mathcal{C}}/K]$, and so $$\mathrm{deg}(F) = \frac{\mathrm{deg}(F_{\mathcal{C}})}{|K|} = |\mathrm{Jac}_{\mathcal{C}}^{+}[\lambda]| \cdot \lambda^{b_0(\Gamma_0) - 1} \cdot \prod_{e \in E} \frac{t_e(r)}{t_e(\lambda r)}.$$ The result follows by Lemma \ref{Jactau torsion size}, using $b_0(\Gamma^{\dag}) = b_0(\Gamma_0) + b_0(\Gamma_+)$.
    
\end{proof}

\subsection{$\widehat{\Z}$-tropical types }

The goal of the rest of this section will be to show the degree of the comparison morphism behaves monomially as we vary the rooting parameter. The key to doing so will be controlling the isotropy orders $t_e(\lambda r)$ appearing in Theorem \ref{degrees} as we vary $\lambda$. To formalise this, we will need a convenient language to talk about compatible strata as we change the rooting parameter. To accomplish this, we introduce a generalised notion of a tropical type that allows us to generate compatible families of mod $r$ tropical types via natural reduction maps.

\begin{definition}\label{Zhattype} A $\widehat{\Z}-$\textbf{tropical type} $\hat{\tau}$ is the data of a twisted pre-tropical type $\underline{\tau}$ along with the extra data of $ \widehat{\Z}$-\textbf{slopes}: For each oriented edge $\vec{e}$ a $\widehat{\Z}$-slope $m_{\vec{e}} = - m_{\overleftarrow{e}} \in \widehat{\Z}$ where $\widehat{\Z} = \varprojlim_r \Z/r\Z$ is the profinite integers. At each vertex $v \in V(\Gamma)$ these slopes must also satisfy the balancing condition:

        $$ d_v = \sum_{v \in e} m_{\vec{e}} + \sum_{v \in l_i \in L(\Gamma)} c_i \in \hat{\Z}.$$

        Furthermore, the numerical assumptions of Section \ref{contactdata} must also hold, for both the marked leg data and the gerby edge parameters and gerby slopes.
\end{definition}

Given any $\widehat{\Z}$-tropical type $\hat{\tau}$, there is a natural reduction to a mod $r$ tropical type $[\hat{\tau}_r]$ given by replacing the $\widehat{\Z}$-slopes with their images under the natural homomorphism $\widehat{\Z} \rightarrow \Z/r\Z$. 

\begin{definition}
    Let $\hat{\tau}$ be a $\widehat{\Z}$-tropical type and $\vec{e} \in E$ a directed edge.\begin{itemize}
        \item We say that $\vec{e}$ is \textbf{small age} if there exists an $m \in \N_{\geq 0}$ such that $$m \equiv m_{\vec{e}} \mod \lambda$$ for all $\lambda \in \N$, i.e. $m_{\vec{e}}$ lies in the image of $\N \rightarrow \widehat{\Z}$.
        \item We say that $\vec{e}$ is \textbf{large age} if $\overleftarrow{e}$ is of small age.
    \end{itemize}

\end{definition}

\begin{remark}
    Note that reversing orientation flips small ages to large ages (and vice versa). One can also define mid-ages akin to \cite{You21} and these are preserved under changing orientation. 
\end{remark}

\begin{lemma}\label{isotropy scaling}
    Let $\hat{\tau}$ be a $\widehat{\Z}$-tropical type and $e \in E$ an edge of small or large age, with slope $m_{\vec{e}} \not = 0$. Suppose that $r$ is \textbf{sufficiently divisible} for $e$, meaning $m_{\vec{e}} | r$. Then $t_e(\lambda r) = \lambda \cdot t_e(r)$ for all $\lambda \in \N$.
\end{lemma}

\begin{proof}
    By the coprimality condition \eqref{noderooting} we have $$\frac{\lambda r}{t_{e}(\lambda r)} = \gcd(\lambda r, m_{\vec{e}}) = |m_{\vec{e}}| = \gcd(r, m_{\vec{e}}) = \frac{r}{t_e(r)}.$$
\end{proof}

\begin{remark}
    This is the place where the rooting parameter $r$ being sufficiently divisible comes into play. In general $$\frac{t_e(\lambda r)}{t_e(r)} = \lambda \cdot \frac{\gcd(r,m_{\vec{e}})}{\gcd(\lambda r, m_{\vec{e}})},$$ which for fixed $m_{\vec{e}}$ is \textbf{not a monomial} in $\lambda$, unless we only vary $\lambda$ over those numbers with a fixed value of $\gcd(\lambda, m_{\vec{e}}/\gcd(r,m_{\vec{e}}))$.
\end{remark}

\begin{remark}
    We point out different notation for the $r$-weightings and compare them for the reader, since it is an easy point of confusion: If $\tilde{w}$ is a $\lambda r$-weighting, then $\lambda \cdot \tilde{w}$ is an $r$-weighting. For a $\widehat{\Z}$-tropical type $\hat{\tau}$ we write $w_{\lambda r}$ for the $\lambda r$-weighting of its reduction, so that $\lambda \cdot w_{\lambda r} = w_{r}$.
\end{remark}

\subsection{Monomiality of irreducible components }

In this section we will show that the irreducible components of $\moduli$ exhibit monomial behaviour as we vary the rooting parameter. 

\subsubsection{Canonical lifting for essential types }

The first step is to show that, for $r$ sufficiently large in a sense to be made explicit soon, for $[\tau] = (\underline{\tau}, w)$ an essential mod $r$ tropical type then there exists at most one inducible essential mod $\lambda r$-tropical type $[\tau']$ that maps to $[\tau]$ under the comparison maps. The key observation will be using the following fact:

\begin{keypoint}\label{slogan}
    If $[\tau]$ is an inducible mod $r$ tropical type, then for any $v \in V_0$ we have $$\mathrm{deg}(\pi_* \mathcal{L}|_{\mathcal{C}_v}) \geq 0$$ for any $\elt \in Z_{[\tau]}$ and $\pi : \mathcal{C} \rightarrow C$ the coarse space.
\end{keypoint}

For what follows, every edge $e$ will have a source vertex in $V_0$ and a sink in $V_+$ so admits a canonical orientation $\vec{e}$.

\begin{center}

\tikzset{every picture/.style={line width=0.75pt}} 

\begin{tikzpicture}[x=0.75pt,y=0.75pt,yscale=-1,xscale=1]

\draw  [fill={rgb, 255:red, 255; green, 207; blue, 194 }  ,fill opacity=1 ][dash pattern={on 4.5pt off 4.5pt}] (389,42) -- (438,42) -- (417,82) -- (368,82) -- cycle ;
\draw  [fill={rgb, 255:red, 181; green, 255; blue, 106 }  ,fill opacity=1 ] (388,175) -- (437,175) -- (416,215) -- (367,215) -- cycle ;
\draw  [fill={rgb, 255:red, 255; green, 207; blue, 194 }  ,fill opacity=1 ][dash pattern={on 4.5pt off 4.5pt}] (385,93) -- (434,93) -- (413,133) -- (364,133) -- cycle ;
\draw  [fill={rgb, 255:red, 255; green, 207; blue, 194 }  ,fill opacity=1 ][dash pattern={on 4.5pt off 4.5pt}] (387,224) -- (436,224) -- (415,264) -- (366,264) -- cycle ;
\draw  [fill={rgb, 255:red, 80; green, 227; blue, 194 }  ,fill opacity=1 ] (392,341) -- (441,341) -- (420,381) -- (371,381) -- cycle ;
\draw    (409,283) -- (409,323.97) ;
\draw [shift={(409,325.97)}, rotate = 270] [color={rgb, 255:red, 0; green, 0; blue, 0 }  ][line width=0.75]    (10.93,-3.29) .. controls (6.95,-1.4) and (3.31,-0.3) .. (0,0) .. controls (3.31,0.3) and (6.95,1.4) .. (10.93,3.29)   ;
\draw    (313,61.97) -- (312,245.97) ;
\draw    (313,61.97) -- (336,61.97) ;
\draw    (312,245.97) -- (335,245.97) ;
\draw    (299,153.97) -- (312.5,153.97) ;
\draw    (231,169) -- (233.96,337.97) ;
\draw [shift={(234,339.97)}, rotate = 268.99] [color={rgb, 255:red, 0; green, 0; blue, 0 }  ][line width=0.75]    (10.93,-3.29) .. controls (6.95,-1.4) and (3.31,-0.3) .. (0,0) .. controls (3.31,0.3) and (6.95,1.4) .. (10.93,3.29)   ;

\draw (392,147.4) node [anchor=north west][inner sep=0.75pt]    {$\vdots $};
\draw (370,292.4) node [anchor=north west][inner sep=0.75pt]    {$\pi _{\lambda r,r}$};
\draw (397,51.4) node [anchor=north west][inner sep=0.75pt]    {$\emptyset $};
\draw (394,104.4) node [anchor=north west][inner sep=0.75pt]    {$\emptyset $};
\draw (394,233.4) node [anchor=north west][inner sep=0.75pt]    {$\emptyset $};
\draw (193,348.4) node [anchor=north west][inner sep=0.75pt]    {$\ (\mathbb{Z} /r\mathbb{Z})^{b_{1}( \Gamma )} \ \ni \ w$};
\draw (341,181.4) node [anchor=north west][inner sep=0.75pt]    {$\tilde{w} \ $};
\draw (193,135.4) node [anchor=north west][inner sep=0.75pt]    {$\ (\mathbb{Z} /\lambda r\mathbb{Z})^{b_{1}( \Gamma )}$};
\draw (172,244.4) node [anchor=north west][inner sep=0.75pt]    {$\mathrm{mod} \ r $};
\draw (478,141.4) node [anchor=north west][inner sep=0.75pt]    {$\subset \ \mathrm{Orb}_{\Lambda }(\mathcal{A}_{\lambda r})$};
\draw (481,352.4) node [anchor=north west][inner sep=0.75pt]    {$\subset \ \mathrm{Orb}_{\Lambda }(\mathcal{A}_{r})$};

\end{tikzpicture}
\end{center}

\begin{lemma}\label{lifting}
    Let $[\tau] = (\underline{\tau},w)$ be an essential mod $r$ tropical type. Suppose that $$r > \max_{v \in V_0, v \in e} (d_v - w(\vec{e}) + \sum_{v \in i \in L(\Gamma)} c_i ).$$ Then, for any $\lambda \in \N$, there is at most one $\lambda r$-weighting $\tilde{w}$ such that the associated mod $\lambda r$-tropical type is inducible and maps to $Z_{[\tau]}$ under the comparison map.

    Furthermore, if $\tilde{w}$ exists it must canonically be given by lifting the residues of $w$ to $\{0,1,\cdots, r-1\} \subset \{0,1,\cdots, \lambda r - 1\}$. Thus if this lift exists for all $\lambda$, there exists a genuine tropical type $\tau$ such that each $[(\underline{\tau}, \tilde{w})] = \tau \mod \lambda r.$ 
\end{lemma}

\begin{proof}
    Fix a vertex $v \in V_0$ and label the attached edges by $e_1,\cdots, e_k$. Take $w(\vec{e}_i)$ as the unique integer lifts in the range $\{0,\cdots,r-1\}$. 
    Now suppose $\tilde{w} \in W_{\underline{\tau},\lambda r}$ is a $\lambda r$-weighting that induces $w$, i.e. $\tilde{w} \equiv w \mod r$ which yields an inducible mod $\lambda r$-tropical type. We may write $\tilde{w} = w + w'$ where $w'$ is a function on the edges $E$ taking values in $\{0, r, \cdots, (\lambda - 1)r \}$. Our goal will be to show that $w'$ is identically zero. Indeed, for any $(\mathcal{C}',(\mathcal{L}',s')) \in Z_{(\underline{\tau},\tilde{w})}$, applying Lemma \ref{pushforward} we have $$ \mathrm{deg}(\pi'_* \mathcal{L}'|_{\mathcal{C}_v}) = $$ $$\frac{d_v}{\lambda r} - \sum_{i = 1}^{k-1}\frac{w(\vec{e}_i) + w'(\vec{e}_i) + \alpha_i \lambda r}{\lambda r} - \frac{d_v - \sum_{i = 1}^{k-1}(w(\vec{e}_i) + w'(\vec{e}_i)) + \beta \lambda r - \sum_{v \in l \in L(\Gamma)} c_l}{\lambda r} - \frac{\sum_{v \in l \in L(\Gamma)} c_l}{\lambda r} $$
\begin{equation}\label{pushforward deg lambda r}
     = - \sum_{i = 1}^{k-1} \alpha_i - \beta.
    \end{equation} where $\alpha_i,\beta \in \Z$ are the unique integers such that \begin{itemize}
        \item $w(\vec{e}_i) + w'(\vec{e}_i) + \alpha_i \lambda r \in \{0,\cdots, \lambda r - 1\}, i = 1, \cdots, k-1$.
        \item $d_v - \sum_{i = 1}^{k-1}(w(\vec{e}_i) + w'(\vec{e}_i)) + \beta \lambda r - \sum_{v \in l \in L(\Gamma)} c_l \in \{0,\cdots, \lambda r - 1\}.$
    \end{itemize} To parse this equation, recall that Lemma \ref{pushforward} tells us the divisor $D$ inducing a line bundle pushes forward to $D - \langle D \rangle$, the divisor minus its fractional part. In the above expression, $\frac{d_v}{\lambda r}$ is the degree of the original divisor $D$, and the three other terms are the degrees of the fractional part. To find the fractional part, we need to come up with a divisor of the same age and then adjust the degree to lie in $[0,1)$ by adding multiples of a divisor pulled back from the coarse space. This results in the form of the $3$ other terms above---the $\alpha_i, \beta$ represent the coarse divisor multiples. The middle term in the expression is a result of rearranging the balancing condition $\mod \lambda r$ at $v$ to express the $\vec{e}_k$ terms in terms of the other edge terms. Note also the marked point contributions cancel out as they do not affect the degree of the pushforward line bundle. By Key Point \ref{slogan} we must have \begin{equation}\label{slogan bound}
        - \sum_{i = 1}^{k-1} \alpha_i - \beta \geq 0.
    \end{equation}

    Note that, for each $i$, $r| w'(\vec{e}_i)$ and $w'(\vec{e}_i) \in \{0,\cdots, \lambda r - 1\}$ implies that $$0 \leq w(\vec{e}_i) + w'(\vec{e}_i) < r + (\lambda - 1)r = \lambda r $$ and therefore \begin{equation}\label{alpha = 0}
        \alpha_i = 0.
    \end{equation} By \eqref{slogan bound} and \eqref{alpha = 0} we have $\beta \leq 0$. By the definition of $\beta$ and the balancing condition modulo $\lambda r$, $$w(\vec{e}_k) + w'(\vec{e}_k) = d_v - \sum_{i = 1}^{k-1}(w(\vec{e}_i) + w'(\vec{e}_i)) + \beta \lambda r - \sum_{v \in l \in L(\Gamma)} c_l,$$ and so $\beta \leq 0$ gives $$\sum_{i = 1}^{k}(w(\vec{e}_i) + w'(\vec{e}_i)) \leq d_v - \sum_{v \in l \in L(\Gamma)} c_l.$$ Now suppose $w'(\vec{e}_j) \neq 0$ for some $j$. Then $w'(\vec{e}_j) \geq r$, and since every term on the left hand side is non-negative, $$w(\vec{e}_j) + r \leq d_v - \sum_{v \in l \in L(\Gamma)} c_l \leq d_v + \sum_{v \in l \in L(\Gamma)} c_l,$$ contradicting the bound on $r$ in the statement of the Lemma. Thus $w'$ vanishes on every edge at $v$. Since every edge of an essential type has its source in $V_0$, $w'$ is identically $0$.

    In particular, $\tilde{w}$ must be given canonically by $w$ i.e. the balancing conditions for $w$ modulo $r$ must in fact hold modulo $\lambda r$ if we take the canonical lift of the residues of $w$ to $\{0,\cdots,r-1\}$. Thus, if this holds for all $\lambda$, we can view $$w(\vec{e}) \in \varprojlim_{\lambda} \Z/\lambda r \Z $$ via a constant sequence, and so $$w(\vec{e}) \in \N. $$ Hence $w$ in fact defines a $\Z$-tropical type.
    
\end{proof}

\begin{lemma}\label{tree small age}

\

\begin{enumerate}
    \item Let $\hat{\tau}$ be a $\hat{\Z}$-tropical type with $b_1(\Gamma) = 0$. Then every edge is of small/large age. 
    \item Let $v$ be a vertex such that every attached edge is bipartite or purely external. Then all such edges are of small/large age.
\end{enumerate}

\end{lemma}

\begin{proof}
    To see part (1), by \cite[Lemma 3.4]{BNR22}, the balancing conditions at each vertex uniquely determine the slopes $m_{\vec{e}}$ and in particular the slopes are forced to be integers. Part (2) is a direct consequence of applying the proof of Lemma \ref{lifting} to the star sub-graph of $v$, and the same proof runs through. This yields canonical integer slope lifts for these edges, which tells us they are small/large age depending on the sign of the integers.
\end{proof}

\subsubsection{Monomiality of essential types }

\begin{definition}\label{k-tau}
    For $\hat{\tau}$ a $\hat{\Z}$-tropical type inducing essential types, define $$k_{\hat{\tau}} := 2\sum_{v \in V_+} g_v - |V_+|.$$ 
\end{definition}

\begin{lemma}\label{essential degree bound}

We have $$0 \leq k_{\hat{\tau}} \leq \max(2g - 1,0).$$    
\end{lemma}

\begin{proof}
    To see $0 \leq k_{\hat{\tau}}$, note that for an essential type we have $g_v \geq 1$ for all $v \in V_+$ and so $$\sum_{v \in V_+} g_v - |V_+| \geq 0. $$ We now show $k_{\hat{\tau}} \leq \max(2g-1,0)$. It is clear that \begin{equation}\label{genus ineq bound}
        k_{\hat{\tau}} \leq 2(b_1(\Gamma) + \sum_{v \in V} g_v) - |V_+| = 2g - |V_+|.
    \end{equation} If $|V_+| \geq 1$ this is bounded by $2g-1$, else $V_+ = \emptyset$ and $\hat{\tau}$ is the trivial type. A direct computation gives $k_{\hat{\tau}} = 0$ in this case. This concludes the Lemma.
\end{proof}

\begin{lemma}
    For $g > 0$, $k_{\hat{\tau}} = 2g - 1$ if and only if $\hat{\tau}$ is of the form \begin{center}

\tikzset{every picture/.style={line width=0.75pt}} 

\begin{tikzpicture}[x=0.75pt,y=0.75pt,yscale=-1,xscale=1]

\draw   (338,154) .. controls (338,140.19) and (349.19,129) .. (363,129) .. controls (376.81,129) and (388,140.19) .. (388,154) .. controls (388,167.81) and (376.81,179) .. (363,179) .. controls (349.19,179) and (338,167.81) .. (338,154) -- cycle ;
\draw    (280,120.96) -- (340,140.96) ;
\draw   (255,120.96) .. controls (255,114.06) and (260.6,108.46) .. (267.5,108.46) .. controls (274.4,108.46) and (280,114.06) .. (280,120.96) .. controls (280,127.86) and (274.4,133.46) .. (267.5,133.46) .. controls (260.6,133.46) and (255,127.86) .. (255,120.96) -- cycle ;
\draw   (255,200.96) .. controls (255,194.06) and (260.6,188.46) .. (267.5,188.46) .. controls (274.4,188.46) and (280,194.06) .. (280,200.96) .. controls (280,207.86) and (274.4,213.46) .. (267.5,213.46) .. controls (260.6,213.46) and (255,207.86) .. (255,200.96) -- cycle ;
\draw    (280,200.96) -- (343,171.96) ;
\draw    (271,263) -- (387,263.94) ;
\draw [shift={(389,263.96)}, rotate = 180.47] [color={rgb, 255:red, 0; green, 0; blue, 0 }  ][line width=0.75]    (10.93,-3.29) .. controls (6.95,-1.4) and (3.31,-0.3) .. (0,0) .. controls (3.31,0.3) and (6.95,1.4) .. (10.93,3.29)   ;
\draw [shift={(271,263)}, rotate = 0.47] [color={rgb, 255:red, 0; green, 0; blue, 0 }  ][fill={rgb, 255:red, 0; green, 0; blue, 0 }  ][line width=0.75]      (0, 0) circle [x radius= 3.35, y radius= 3.35]   ;
\draw    (326,210) -- (326,249.96) ;
\draw [shift={(326,251.96)}, rotate = 270] [color={rgb, 255:red, 0; green, 0; blue, 0 }  ][line width=0.75]    (10.93,-3.29) .. controls (6.95,-1.4) and (3.31,-0.3) .. (0,0) .. controls (3.31,0.3) and (6.95,1.4) .. (10.93,3.29)   ;

\draw (358,145.4) node [anchor=north west][inner sep=0.75pt]    {$g$};
\draw (262,114.4) node [anchor=north west][inner sep=0.75pt]    {$0$};
\draw (262,194.4) node [anchor=north west][inner sep=0.75pt]    {$0$};
\draw (258,148.4) node [anchor=north west][inner sep=0.75pt]    {$\vdots $};

\end{tikzpicture}
    \end{center} for some number (potentially $0$) of genus $0$ external vertices.
\label{top degree terms}
\end{lemma}

\begin{proof}
    Inequality \eqref{genus ineq bound} is an equality if and only if $b_1(\Gamma) = 0$ and $g_v = 0$ for all $v \in V_0$ (``all the genus is contained in the internal vertices"). This in turn equals $2g-1$ if and only if additionally $|V_+| = 1$.
\end{proof}

 \begin{theorem}\label{essential degree theorem}
    Let $\hat{\tau}$ be a $\hat{\Z}$-tropical type whose reduction modulo $\lambda r$ is an essential tropical type. Suppose that \begin{itemize}
        \item \textbf{$r$ is sufficiently large:} $r > \max_{v \in V_0, v \in e} (d_v - w_r(\vec{e}) + \sum_{v \in i \in L(\Gamma)} c_i )$.
        \item \textbf{$r$ is sufficiently divisible:} $m_{\vec{e}} | r $ for all $e \in E$.
        \item \textbf{Nodal contacts are non-trivial:} Suppose that $m_{\vec{e}} \not = 0$ for all $e \in E.$
    \end{itemize} Then the degree of the comparison morphism $$\mathrm{deg}(\pi_{\lambda r, r}|_{Z_{[\hat{\tau}_{\lambda r}]}}) = \lambda^{k_{\hat{\tau}}} $$ is a monomial in $\lambda$ with $$ 0 \leq k_{\hat{\tau}} \leq \max(2g-1,0).$$
\end{theorem}

\begin{proof}
    By Lemma \ref{lifting} the slopes $m_{\vec{e}}$ are integers, so Lemma \ref{isotropy scaling} gives $t_e(\lambda r) = \lambda \cdot t_e(r)$ for all $e \in E$. Theorem \ref{degrees} then shows that the degree is the monomial in $\lambda$ of degree $$|E^{\mathrm{b}}| - b_0(\Gamma_+) + b_1(\Gamma_+) + 2\sum_{v \in V_+} g_v - |E|.$$ For an essential type, observe that \begin{itemize}
        \item $|E^{\mathrm{b}}| = |E|$.
        \item $b_0(\Gamma_+) = |V_+|$.
        \item $b_1(\Gamma_+) = 0$.
    \end{itemize} It follows this degree equals $k_{\hat{\tau}}$. The fact the degree of the monomial lies between $0$ and $2g-1$ is Lemma \ref{essential degree bound} and the result follows.
\end{proof}

\begin{remark}
    The assumption that the nodal contacts are non-trivial is a minor point and geometrically motivated. Indeed, suppose $f: \mathcal{C} \rightarrow X_{D,r}$ induces a mod $r$ tropical type $[\tau]$ and $e$ is a bipartite edge corresponding to a node $q$. Observe that $$m_{e} = 0 \Leftrightarrow t_e = 1.$$ Assume that $t_e = 1$ i.e. $q$ is non-stacky. Let $v$ be the source of $e$ and suppose for simplicity that $q$ is the only point on $\mathcal{C}_v$ mapping to $\frac{1}{r} D$. Let $\pi : \mathcal{C} \rightarrow C$ be the coarse space. Then the line bundle $$(\pi \circ f)^{*} \mathcal{O}(D)|_{\mathcal{C}_v} \cong \mathcal{O}(k q) $$ must admit an $r$th root, for $k \in \N$ the contact order at the node. Since the map is transverse on component $\mathcal{C}_v$, $k$ cannot be $0$. Thus if $r > k$ no such $r$th root exists and $q$ would be forced to be stacky, and so the nodal contact assumption automatically holds.
\end{remark}

\begin{example}\label{examples of degrees}

\

    \begin{enumerate}
        \item For $[\tau]$ the trivial type, $k_{\hat{\tau}} = 0$. Therefore the comparison maps are birational morphisms when restricted to the main component. This can be seen directly as follows. Given $(\mathcal{C},(\mathcal{L}^{\otimes \lambda},s^{\otimes \lambda}))$ for $\mathcal{C}$ smooth and $s$ not identically $0$, we can determine $s$ up to non-zero scalar, and thus we can recover $\mathcal{L}$.

        \item Consider the following type for $g > 0$ \begin{center}

\tikzset{every picture/.style={line width=0.75pt}} 

\begin{tikzpicture}[x=0.75pt,y=0.75pt,yscale=-1,xscale=1]

\draw    (338.1,103.53) ;
\draw [shift={(338.1,103.53)}, rotate = 0] [color={rgb, 255:red, 0; green, 0; blue, 0 }  ][fill={rgb, 255:red, 0; green, 0; blue, 0 }  ][line width=0.75]      (0, 0) circle [x radius= 3.35, y radius= 3.35]   ;
\draw    (273,161) -- (379,160.97) ;
\draw [shift={(381,160.97)}, rotate = 179.98] [color={rgb, 255:red, 0; green, 0; blue, 0 }  ][line width=0.75]    (10.93,-3.29) .. controls (6.95,-1.4) and (3.31,-0.3) .. (0,0) .. controls (3.31,0.3) and (6.95,1.4) .. (10.93,3.29)   ;
\draw [shift={(273,161)}, rotate = 359.98] [color={rgb, 255:red, 0; green, 0; blue, 0 }  ][fill={rgb, 255:red, 0; green, 0; blue, 0 }  ][line width=0.75]      (0, 0) circle [x radius= 3.35, y radius= 3.35]   ;
\draw    (339,125) -- (339,151.97) ;
\draw [shift={(339,153.97)}, rotate = 270] [color={rgb, 255:red, 0; green, 0; blue, 0 }  ][line width=0.75]    (10.93,-3.29) .. controls (6.95,-1.4) and (3.31,-0.3) .. (0,0) .. controls (3.31,0.3) and (6.95,1.4) .. (10.93,3.29)   ;

\draw (320,83.4) node [anchor=north west][inner sep=0.75pt]    {$g$};

\end{tikzpicture} 
        \end{center} We calculate \begin{itemize}
            \item $2 \sum_{v \in V_+} g_v = 2g$.
            \item $|V_+| = 1$.
        \end{itemize} It follows that $$k_{\hat{\tau}} = 2g -1.$$ More generally, we see $k_{\hat{\tau}}$ is invariant under attaching external rational tails.

    \end{enumerate}

\end{example} \qed

\subsection{Polynomials from strata }

We now rephrase Theorem \ref{essential degree theorem} by explaining how to get polynomials from compatible collections of strata as we vary the rooting parameter. These polynomials will have coefficients valued in the Chow groups of the moduli spaces in question, and will be polynomial in the rooting parameter.

\begin{definition}\label{compatible system poly}
    Given a finite set of $\hat{\Z}$-tropical types $\hat{I} = \{\hat{\tau}_j\}$, define the Chow classes $$P_{\hat{I},r}(\lambda) := (\pi_{\lambda r,r})_{*}(\sum_{j \in \hat{I}}[Z_{[\hat{\tau}_{j,\lambda r}]}]) \in A_{*}(\moduli)$$ where $[\hat{\tau}_{j,\lambda r}]$ denotes the reduction of $\hat{\tau}_j$ to a mod $\lambda r$ tropical type.
\end{definition}

\begin{lemma}\label{constant term}
     Let $\hat{I} = \{\hat{\tau}_j\} $ be a collection of essential $\hat{\Z}$-tropical types. Then the constant coefficient of $P_{\hat{I},r}(\lambda)$ is given by the sum of the fundamental classes of trivial types in $\hat{I}$.
\end{lemma}

\begin{proof}
   We just need to show that $k_{\hat{\tau}} = 0$ if and only if $\hat{\tau}$ is the trivial type. One direction is Example \ref{examples of degrees} (1). For the other direction, since $g_v \geq 1$ for $v \in V_+$, observe $$2\sum_{v \in V_+} g_v - |V_+| \geq |V_+| \geq 0$$ with equality if and only if $V_+ = \emptyset$. In this case, as $\Gamma$ is bipartite, $\hat{\tau}$ must be the trivial type.
\end{proof}

\begin{corollary}\label{essential polynomial}
    Let $\hat{I} = \{\hat{\tau}_j\} $ be a collection of essential $\hat{\Z}$-tropical types. Denote by $\Gamma_j, V_j, E_j$ the underlying graph, vertices and edges associated to $\hat{\tau}_j$. Suppose that \begin{itemize}
        \item \textbf{$r$ is sufficiently large:} $r > \max_{v \in \cup_j V_{0,j}, v \in e} (d_v - w_r(\vec{e}) + \sum_{v \in i \in L(\Gamma_j)} c_i )$.
        \item \textbf{$r$ is sufficiently divisible:} $m_{\vec{e}} | r $ for all $e \in E_j$.
        \item \textbf{Nodal contacts are non-trivial:} Suppose that $m_{\vec{e}} \not = 0$ for all $e \in E_j$.
    \end{itemize} Then $P_{\hat{I},r}(\lambda)$ is a polynomial in $\lambda$ of degree at most $\max(2g-1,0)$ with constant term given by the contributions from the main component. The top degree terms are given by irreducible components consisting of an internal genus $g$ vertex.
\end{corollary}

\begin{proof}
    Immediate from combining Theorem \ref{essential degree theorem}, Lemma \ref{constant term} and Lemma \ref{top degree terms}.
\end{proof}

\section{Virtual classes }\label{virtual classes}

In genus $g \geq 1$ we have seen $\moduli$ is reducible with many components. As a result, in general $\moduli$ will be equipped with an interesting virtual class. The purpose of this section is to investigate this virtual structure and situate it alongside the stratification by mod $r$ tropical types. In particular, we show the virtual structure is the fundamental class in genus $1$ and conjecture a form for the virtual cycle that would imply polynomiality of orbifold Gromov--Witten invariants of any geometric smooth pair.

\subsection{Dimension calculations }\label{dimension calculations}

\begin{lemma}\label{dims}
    Let $[\tau]$ be a mod $r$ tropical type. Then $$\mathrm{dim}(Z_{[\tau]}) = 3g - 3 + n + b_0(\Gamma_0) - |E| + |E^{\mathrm{b}}| - b_0(\Gamma^{\dag})  + b_1(\Gamma_+) + \sum_{v \in V_+} g_v. $$
\end{lemma}

\begin{proof}

Proposition \ref{action} implies that \begin{equation}\label{dimexpression1}
    \mathrm{dim}(Z_{[\tau]}) = \mathrm{dim}(S) + \mathrm{dim}(\mathrm{Jac}_{\mathcal{C}}^{+}) + b_0(\Gamma_0) - 1
\end{equation} for any object $\elt \in Z_{[\tau]}^{\circ}$, where $S$ is a local section of the $\mathrm{Jac}_{[\tau]} \times [\mathbb{G}_m^{b_0(\Gamma_0)}/\mathbb{G}_m]$-torsor. By the exact sequence of \eqref{SESJac} we see that \begin{equation}\label{Jac+dim}
    \mathrm{dim}(\mathrm{Jac}_{\mathcal{C}}^{+}) = |E^{\mathrm{b}}| - b_0(\Gamma^{\dag}) + 1 + \mathrm{dim}(\mathrm{Jac}(C_+)) = |E^{\mathrm{b}}| - b_0(\Gamma^{\dag}) + 1 + b_1(\Gamma_+) + \sum_{v \in V_+} g_v.
\end{equation} Combining \eqref{Jac+dim} with \eqref{dimexpression1} it is therefore sufficient to prove that \begin{equation}\label{section dim}
    \mathrm{dim}(S) = 3g - 3 + n - |E|.
\end{equation} Consider the forgetful morphism to the associated stratum of the stack of pre-stable twisted curves $$S \rightarrow \mathfrak{M}_{[\tau]} $$ given by $$\elt \mapsto \mathcal{C}.$$ Since $S$ is a local section of the torsor, this forgetful map is a local isomorphism on to the image and hence $$\mathrm{dim}(S) = \mathrm{dim}(\mathfrak{M}_{[\tau]}) = 3g - 3 + n - |E|. $$ Note that the automorphisms of $(\mathcal{L},s)$ are already encoded in the group $[\mathbb{G}_m^{b_0(\Gamma_0)}/\mathbb{G}_m]$; either the automorphisms of the pair are trivial if $V \not = V_+$, else $V = V_+$ and the automorphisms of the pair are identified with $\mathbb{G}_m$, while at the same time $[\mathbb{G}_m^{b_0(\Gamma_0)}/\mathbb{G}_m] \cong \BGm$. This is as discussed in the proof of Theorem \ref{degrees}. The result then follows.
    
\end{proof}

\begin{remark}
    Observe that, for $[\tau]$ the trivial type, the dimension of the main component equals $3g - 3 + n$, which is also the dimension of $$\mathrm{Log}_{\Lambda}((\mathcal{A}| \BGm) ; \tau), $$ the moduli space of logarithmic maps of global type $\tau$, where $\tau$ is any enhancement of $\underline{\tau}$ to a genuine tropical type. See e.g.  \cite[Proposition 3.30]{ACGS20}. In genus $0$, this is a reflection of results in \cite{ACW17}; then $\moduli$ is irreducible consisting of just the main component and \cite[Theorem 4.2.2]{ACW17} says that $\moduli$ and $ \mathrm{Log}_{\Lambda}((\mathcal{A}| \BGm) ; \tau)$ are birational. Forthcoming work \cite{CJ} will generalise this for arbitrary strata, giving the dimension in terms of a corresponding space of log maps.
\end{remark}

\begin{corollary}\label{essentialdim}
    For $[\tau]$ an essential mod $r$ tropical type $$\mathrm{dim}(Z_{[\tau]}) = 3g - 3 + n - |V_+| + \sum_{v \in V_+} g_v \geq \mathrm{dim}(Z_{\mathrm{main}}).$$

    In particular, for genus $g = 1$, $\moduli$ is equidimensional of dimension $n$. For $g > 1$, $\moduli$ is never equidimensional.
\end{corollary}

\begin{proof}
    For $[\tau]$ essential, observe we have $b_0(\Gamma^{\dag}) = |V|, |E^{\mathrm{b}}| = |E|, b_0(\Gamma_0) = |V_0|$. Writing $|V| = |V_0| + |V_+|$ and plugging into Lemma \ref{dims} gives $$\mathrm{dim}(Z_{[\tau]}) = 3g - 3 + n - |V_+| + \sum_{v \in V_+} g_v.$$ The inequality follows from the fact that $g_v > 0$ for any $v \in V_+$. In genus $1$, any essential type has precisely $1$ internal vertex of genus $1$, else it is the trivial type. In the non-trivial case we see $$\mathrm{dim}(Z_{[\tau]}) = n - 1 + 1 = n $$ which is also the dimension of $Z_{\mathrm{main}}$. In genus $g > 1$, the type $[\tau]$ consisting of precisely one vertex which is internal satisfies $$\mathrm{dim}(Z_{[\tau]}) = 4g - 4 + n > 3g - 3 + n = \mathrm{dim}(Z_{\mathrm{main}}). \qedhere $$
\end{proof}

\subsection{Obstruction theories }\label{obstruction theories}

There is a natural morphism $$\nu: \moduli \rightarrow \mathfrak{M}_{\Lambda}(\BGmr) $$ given by forgetting the section and just remembering the data of the pre-stable curve of type $\Lambda$ with line bundle. As constructed in \cite[Section 5.2]{ACW17}, $\nu$ is equipped with a perfect obstruction theory \begin{equation}\label{obstheory}
    \mathbb{E}_{\nu} = (R \pi_{*} \tilde{\mathcal{L}})^{\vee}
\end{equation} where $\tilde{\mathcal{L}}$ is the universal line bundle on the universal curve $\tilde{\mathcal{C}}$ and $\pi:\tilde{\mathcal{C}} \rightarrow \moduli$ the natural projection to the base.

The perfect obstruction theory on $\moduli$ induces a virtual fundamental class on any finite type open substack.  The virtual classes are of pure dimension $$\mathrm{vd} = 3g - 3 + n.$$ We denote by $[\moduli]^{\mathrm{vir}}$ the data of all these virtual classes over all finite type open substacks which we think of as a single Chow class, but it does not live in the Chow ring of $\moduli$ just because $\moduli$ is of infinite type and the support of $[\moduli]^{\mathrm{vir}}$ may be of infinite type also.

\subsection{Splitting in genus $1$ }

We start by recalling a folklore result telling us when the virtual class and usual fundamental classes agree, in the general setting that we need for DM type morphisms between Artin stacks.

\begin{lemma}[\cite{Cru24} Lemma 5.3.1]\label{equidim virtual class}
    Let $\varphi: X \rightarrow Y $ be a morphism of algebraic stacks of DM type equipped with a relative perfect obstruction theory $\nu: \mathbb{E} \rightarrow \tau^{\geq -1}\mathbb{L}_{\varphi}$. Suppose additionally that $Y$ is smooth and $X$ is of finite type and equidimensional of dimension $\mathrm{dim}(X) = d$. Suppose also that the induced virtual dimension from $\nu$ is $\mathrm{vd} = d$. Then \begin{itemize}
        \item $\nu$ is a quasi-isomorphism.
        \item $\varphi$ is lci.
        \item The induced virtual class agrees with the fundamental class $[X]^{\mathrm{vir}} = [X]$.
    \end{itemize}

\end{lemma}


\begin{corollary}\label{g = 1 virtual class}
    In genus $g = 1$, the virtual class is $$[\moduli]^{\mathrm{vir}} = \sum_{[\tau] \ \mathrm{essential}} [Z_{[\tau]}]$$ which is the usual fundamental class on each finite union of irreducible components.
\end{corollary}

\begin{proof}
     $\moduli$ is equidimensional of dimension equalling the virtual dimension by Corollary \ref{essentialdim}, and the target $\mathfrak{M}_{\Lambda}(\BGmr)$ of $\nu$ is smooth. The result then follows by Lemma \ref{equidim virtual class} and the classification of irreducible components given by Theorem \ref{irredcpts}.
\end{proof}

\begin{theorem}\label{g = 1 virtual polynomial}
    For genus $g = 1$ and $r$ sufficiently large and divisible, the family of Chow classes $$Q(\lambda) := (\pi_{\lambda r, r})_{*}[\lmoduli]^{\mathrm{vir}} \in A_*(\moduli)$$ is a polynomial in $\lambda$ of degree $1 = 2g - 1$ whose constant term is given by the contribution from $Z_{\mathrm{main}}$.
\end{theorem}

\begin{proof}
 Combining Corollary \ref{g = 1 virtual class} with Theorem \ref{essential degree theorem} shows $Q(\lambda)$ is a Laurent polynomial in $\lambda$. Lemma \ref{essential degree bound} shows the degree condition in the statement. The constant term contribution is covered by Lemma \ref{constant term}.
\end{proof}

\subsection{Relation to geometric virtual classes }\label{geometric virtual classes}

Let $(X|D)$ be a smooth pair and consider the space of twisted maps to root stacks of $X$ along $D$, $\geommoduli$. The arguments at the beginning of \cite[Section 5.2]{ACW17} hold in higher genus, even though they are stated for genus $0$. The authors consider the natural morphisms \begin{equation}\label{forgetful}
    \geommoduli \xrightarrow{\mu_r} \mathrm{Orb}_{\Lambda}(\mathcal{A}_r) \xrightarrow{\nu} \mathfrak{M}_{\Lambda}(\BGmr)
\end{equation} for $$\mu_r: (f: \mathcal{C} \rightarrow X_{D,r}) \mapsto (f^{*}\mathcal{O}_{X_{D,r}}(\frac{1}{r}D), s_{\frac{D}{r}}) $$ and endow $\mu_r$ with a relative obstruction theory. The obstruction theory for $\nu$ was given in \eqref{obstheory}. In particular, we can describe the virtual class of $\geommoduli$ via virtual pullback from the universal target \begin{equation}\label{virtualpb}
    [\geommoduli]^{\mathrm{vir}} = \mu_r^{!} [\moduli]^{\mathrm{vir}}.
\end{equation}

Since $\geommoduli$ is of finite type, the image of $\mu_r$ is contained in a finite union of open strata $$\bigcup_{i \in I(X_{D,r})} Z_{[\tau_i]}^{\circ}.$$ There are natural comparison maps as in the case for the universal target $$\Pi_{\lambda r, r}: \mathrm{Orb}_{\Lambda}(X_{D,\lambda r}) \rightarrow \geommoduli$$ induced by post-composing with the natural root stack maps $$X_{D, \lambda r} \rightarrow X_{D,r}. $$ Again,  one needs to partially coarsen the domain curves to maintain representability as in Section \ref{comparison maps}.

The following Lemma implies that polynomial properties of the virtual class on the universal space induce polynomiality of $[\geommoduli]^{\mathrm{vir}}$.

\begin{lemma}[\cite{Cru24} Lemma 5.4.1]\label{cartesian}
    There is a cartesian square with compatible obstruction theories for the morphisms $\mu_{\lambda r}, \mu_r$ \begin{center}
    \begin{tikzcd}
{\mathrm{Orb}_{\Lambda}(X_{D,\lambda r})} \arrow[d, "{\Pi_{\lambda r, r}}"'] \arrow[r, "\mu_{\lambda r}"] & \mathrm{Orb}_{\Lambda}(\mathcal{A}_{\lambda r}) \arrow[d, "{\pi_{\lambda r, r}}"] \\
{\mathrm{Orb}_{\Lambda}(X_{D,r})} \arrow[r, "\mu_r"]                                                      & \mathrm{Orb}_{\Lambda}(\mathcal{A}_{r})                                          
\end{tikzcd}
\end{center}

\end{lemma}

\subsection{Conjectural splitting }

To fully deduce \cite{TY2, TY3} from the results of this paper thus far, one needs to know a general structure of the decomposition of the virtual class $[\moduli]^{\mathrm{vir}}$. For genus $g > 1$, Corollary \ref{essentialdim} tells us there is no hope for the virtual class to be supported on irreducible components given by essential types. The next best thing would be the following:

\begin{conjecture}
    There exists an expression of the form $$[\moduli]^{\mathrm{vir}} = \sum_{([\tau], \Gamma')} a_{[\tau], \Gamma'} [Z_{[\tau], \Gamma'}] \cap C_{[\tau], \Gamma'}$$ where the sum is over strata of dimension at least $\mathrm{vd}$ and $C_{[\tau], \Gamma'}$ is a Chow class which is a \textbf{Laurent polynomial} in $r$ such that $$ \mathrm{deg}_r([Z_{[\tau], \Gamma'}]) + \mathrm{deg}_r (C_{[\tau], \Gamma'}) \geq 0.$$
\end{conjecture}

Further analysis of the degrees generalising Definition \ref{k-tau}, Lemma \ref{essential degree bound} would be needed to conclude polynomiality of degree at most $2g-1$ in this general case. One would then be able to upgrade to the geometric setting with Lemma \ref{cartesian}. This question is addressed in forthcoming work.

\section{Geometric examples }\label{examples}

\subsection{The case of multiple divisor components }

In examples of the next section, and in many other natural geometric situations, a smooth pair $(X|D)$ may have $D$ reducible. The theory of the paper thus far deals well with the case we root along each component of $D$ with a single rooting parameter $r$, however one may want more flexibility and consider the multi-root stack $X_{D, \vec{r}}$ for $D = D_1 + \cdots + D_l$ and $\vec{r} \in \N^{l}$. In this case, the appropriate universal target is the Artin fan associated to $l$ stacky rays $\rho_{r_i}$ with sub-monoid of index $r_i$ in ray $\rho_{r_i}$ glued along the origin, denoted $\Sigma_{\vec{r}}$, corresponding to $\mathcal{A}_{\vec{r}} := \cup_{i = 1}^l \mathcal{A}_{r_i}$ identified over their $\mathrm{Spec}(\C)$ points. We may think of $\Sigma_{\vec{r}}$ as a stacky fan and each ray as a stacky cone in the sense of \cite{GS15}. Write $R := \mathrm{lcm}_i (r_i)$. There is a map of stacky fans $$\Sigma_{\vec{r}} \rightarrow \rho_{R},$$ sending the generator of each stacky ray $\rho_{r_i}$ to the generator of $\rho_{R}$, which induces a morphism of Artin fans $$\mathcal{A}_{\vec{r}} \rightarrow \mathcal{A}_{R}.$$ On the level of line-bundle section pairs this is given by tensoring together the line bundles and sections $$(\mathcal{L}_j,s_j)_{j = 1}^l \mapsto (\bigotimes_{j = 1}^l \mathcal{L}_j, \bigotimes_{j = 1}^l s_j).$$ Since the components $D_j$ are disjoint, at most one of the $s_j$ vanishes at any given point, so this morphism restricts to an isomorphism on each component $\mathcal{A}_{r_j}$ and in particular is representable. Geometrically, $\bigotimes_{j} \mathcal{L}_j$ is an $R$-th root of $\mathcal{O}(\sum_{j} \frac{R}{r_j} D_j)$, so the composite $X_{D,\vec{r}} \rightarrow \mathcal{A}_{R} \rightarrow \mathcal{A}$ is the map induced by the divisor $\sum_{j} \frac{R}{r_j} D_j$.

We then have the sequence of morphisms $$\mathrm{Orb}_{\Lambda}(X_{D,\vec{r}}) \rightarrow \mathrm{Orb}_{\Lambda}(\mathcal{A}_{\vec{r}}) \rightarrow \mathrm{Orb}_{\Lambda}(\mathcal{A}_{\mathrm{lcm}_i (r_i)}).  $$ The notion of mod $\vec{r}$-tropical type can be defined, generalising Definition \ref{orbitroptype}, which now has an assignment of cones coming from any cones in $\Sigma_{\vec{r}}$. Construction \ref{assocmodrtype} can be generalised to this setting, endowing $\mathrm{Orb}_{\Lambda}(\mathcal{A}_{\vec{r}})$ with a stratification by mod $\vec{r}$-type such that the morphism $\mathrm{Orb}_{\Lambda}(\mathcal{A}_{\vec{r}}) \rightarrow \mathrm{Orb}_{\Lambda}(\mathcal{A}_{\mathrm{lcm}_i (r_i)})$ sends strata to strata. Additionally, there is an appropriate notion of essential mod $\vec{r}$-tropical type and a generalisation of Theorem \ref{irredcpts} classifying the irreducible components of $ \mathrm{Orb}_{\Lambda}(\mathcal{A}_{\vec{r}})$. The map $\mathrm{Orb}_{\Lambda}(\mathcal{A}_{\vec{r}}) \rightarrow \mathrm{Orb}_{\Lambda}(\mathcal{A}_{\mathrm{lcm}_i (r_i)})$ will send distinct irreducible components of the domain onto the same irreducible component of the codomain.

In examples below where we have a reducible divisor and distinct rooting parameters, we will consider the composite morphism $\mathrm{Orb}_{\Lambda}(X_{D,\vec{r}}) \rightarrow \mathrm{Orb}_{\Lambda}(\mathcal{A}_{\mathrm{lcm}_i (r_i)}) $ and all pictures will be drawn with respect to this map.

\subsection{Pullback of strata to geometric moduli }
    Let $$[f: \mathcal{C} \rightarrow X_{D,r}] \in \geommoduli.$$ This twisted map is equivalent to the data of \begin{itemize}
        \item A map $\bar{f}: \mathcal{C} \rightarrow X.$
        \item A line bundle $\mathcal{L}$ on $\mathcal{C}$ and an isomorphism $$\mathcal{L}^{\otimes r} \cong \bar{f}^* \mathcal{O}(D). $$
    \end{itemize} In particular, the line bundle $\mathcal{L}$ is uniquely determined by $\bar{f}$ up to the action of the $r$-torsion of the associated Jacobian fibre $$\mathrm{Jac}_{\mathcal{C}}^{+}[r].$$ More generally, for $[\tau]$ a mod $r$ tropical type, define $$Z_{[\tau]}^{(X|D),\circ} := \mu_{r}^{-1}(Z_{[\tau]}^{\circ}) \subset \geommoduli. $$ If $Z_{[\tau]}^{(X|D),\circ} \not = \emptyset$, it is naturally a $\mathrm{Jac}_{\mathcal{C}}^{+}[r]$-torsor. Since $\mathrm{Jac}_{\mathcal{C}}^{+}[r]$ is now a finite group, as opposed to $\mathrm{Jac}_{\mathcal{C}}^{+}$ which is connected, the strata $Z_{[\tau]}^{(X|D),\circ}$ may be disconnected as examples in the next section will show.

\subsection{Examples}

In this section we exhibit some explicit computations of orbifold invariants of smooth pairs $(X|D)$ for $X$ a projective variety, and compare to the results about the universal moduli spaces earlier in the paper. Let's start by rephrasing the original example of Maulik \cite[Section 1.7]{ACW17}.

\begin{example}\label{Maulik}
    Let $X = (E \times \mathbb{P}^1| E_0 + E_{\infty})$ for $E$ a smooth genus $1$ curve. Let $X_{r,s}$ be the root stack along $E_0,E_{\infty}$ with parameters $r,s$ respectively and let $l := \mathrm{lcm}(r,s)$. Consider twisted stable maps to $X_{r,s}$ with no marked points $n = 0$ of fibre class $E$ and genus $g = 1$, $\mathrm{Orb}(X_{r,s})$. There is an isomorphism $$\mathrm{Orb}(X_{r,s}) \cong \mathbb{P}^1_{r,s} \sqcup \coprod_{r^2 - 1} B\mu_r \sqcup \coprod_{s^2 - 1} B \mu_s. $$ To see this, note that an object in $\mathrm{Orb}(X_{r,s})$ is a map $E \rightarrow X_{r,s}$ of fibre class $E$. The fibres are parameterised by the base $\mathbb{P}^1_{r,s}$, however if we map to the stacky fibres, we have the data of $$E \rightarrow E \times B \mu_r $$ (or to $E \times B \mu_s$) which is the identity on the first factor. Such data is then equivalent to a $\mu_r$-torsor over $E$, $ E' \rightarrow E$. Isomorphism classes of such torsors are classified by the \'etale cohomology group $H^1(E,\mu_r)$, and there is a natural isomorphism $$H^1(E,\mu_r) = \mathrm{Hom}(\pi_1(E), \mu_r).$$ The latter group is isomorphic to $E[r]^{\vee}$, after picking a base point in $E$, since $r$-torsion points are the deck transformations of the \'etale covers $[n] : E \rightarrow E$ in $\pi_1(E)$ (and these form a cofinal system over all \'etale covers). The Weil pairing on $E$ with this base point gives an identification with $E[r]$. The trivial torsor is already accounted for in the $\mathbb{P}^1_{r,s}$ component and so the remaining $r^2 - 1 + s^2 - 1$ are indexed by the non-trivial torsion points of $E$. The virtual dimension is easily calculated to be $0$. Under the natural map $\geommoduli \rightarrow \moduli$, $\mathbb{P}^1_{r,s} \rightarrow Z_{\mathrm{main}} \subset \mathrm{Orb}_{\Lambda}(\mathcal{A}_l)$ is contracted: its complement of the two stacky points maps to the point $(E,(\mathcal{O},1)) \in Z_{\mathrm{main}}^{\circ}$, and the two stacky points map to the boundary point $(E,(\mathcal{O},0)) \in Z_{\mathrm{main}} \cap Z_{[\tau]}$. Meanwhile the remaining gerbes $B\mu_r, B \mu_s \rightarrow Z_{[\tau]}$ for $[\tau]$ the genus $1$ single vertex type with $V = V_+$. In particular, we see that we do not obtain the whole Jacobian fibres as in the universal space, and the irreducible components of $\mathrm{Orb}(X_{r,s})$ do not align with those of $\mathrm{Orb}_{\Lambda}(\mathcal{A}_l)$.

    \begin{center}

\tikzset{every picture/.style={line width=0.75pt}} 

\begin{tikzpicture}[x=0.75pt,y=0.75pt,yscale=-1,xscale=1]

\draw  [fill={rgb, 255:red, 126; green, 211; blue, 33 }  ,fill opacity=1 ] (241,148.48) .. controls (241,121.71) and (294.95,100) .. (361.5,100) .. controls (428.05,100) and (482,121.71) .. (482,148.48) .. controls (482,175.26) and (428.05,196.97) .. (361.5,196.97) .. controls (294.95,196.97) and (241,175.26) .. (241,148.48) -- cycle ;
\draw  [fill={rgb, 255:red, 248; green, 231; blue, 28 }  ,fill opacity=1 ] (351.3,37.97) -- (450,37.97) -- (407.7,130.97) -- (309,130.97) -- cycle ;
\draw [color={rgb, 255:red, 208; green, 2; blue, 27 }  ,draw opacity=1 ]   (361.5,130.47) -- (362,130.97) ;
\draw [color={rgb, 255:red, 208; green, 2; blue, 27 }  ,draw opacity=1 ] [dash pattern={on 4.5pt off 4.5pt}]  (362,130.97) -- (361.5,130.47) ;
\draw    (360,129.97) -- (396,37.97) ;
\draw    (468,74) .. controls (436.32,104.66) and (420.32,57.92) .. (379.25,101.62) ;
\draw [shift={(378,102.97)}, rotate = 312.4] [color={rgb, 255:red, 0; green, 0; blue, 0 }  ][line width=0.75]    (10.93,-3.29) .. controls (6.95,-1.4) and (3.31,-0.3) .. (0,0) .. controls (3.31,0.3) and (6.95,1.4) .. (10.93,3.29)   ;
\draw [color={rgb, 255:red, 144; green, 19; blue, 254 }  ,draw opacity=1 ][fill={rgb, 255:red, 208; green, 2; blue, 27 }  ,fill opacity=1 ]   (360.5,130.47) -- (360,129.97) ;
\draw [shift={(360.5,130.47)}, rotate = 225] [color={rgb, 255:red, 144; green, 19; blue, 254 }  ,draw opacity=1 ][fill={rgb, 255:red, 144; green, 19; blue, 254 }  ,fill opacity=1 ][line width=0.75]      (0, 0) circle [x radius= 3.35, y radius= 3.35]   ;
\draw [color={rgb, 255:red, 144; green, 19; blue, 254 }  ,draw opacity=1 ]   (59,140) -- (156,139.97) ;
\draw    (164,138) -- (221,137.97) ;
\draw [shift={(223,137.97)}, rotate = 179.97] [color={rgb, 255:red, 0; green, 0; blue, 0 }  ][line width=0.75]    (10.93,-3.29) .. controls (6.95,-1.4) and (3.31,-0.3) .. (0,0) .. controls (3.31,0.3) and (6.95,1.4) .. (10.93,3.29)   ;
\draw [color={rgb, 255:red, 208; green, 2; blue, 27 }  ,draw opacity=1 ][fill={rgb, 255:red, 208; green, 2; blue, 27 }  ,fill opacity=1 ]   (77,123.97) ;
\draw [shift={(77,123.97)}, rotate = 0] [color={rgb, 255:red, 208; green, 2; blue, 27 }  ,draw opacity=1 ][fill={rgb, 255:red, 208; green, 2; blue, 27 }  ,fill opacity=1 ][line width=0.75]      (0, 0) circle [x radius= 3.35, y radius= 3.35]   ;
\draw [color={rgb, 255:red, 208; green, 2; blue, 27 }  ,draw opacity=1 ][fill={rgb, 255:red, 208; green, 2; blue, 27 }  ,fill opacity=1 ]   (79,74.97) ;
\draw [shift={(79,74.97)}, rotate = 0] [color={rgb, 255:red, 208; green, 2; blue, 27 }  ,draw opacity=1 ][fill={rgb, 255:red, 208; green, 2; blue, 27 }  ,fill opacity=1 ][line width=0.75]      (0, 0) circle [x radius= 3.35, y radius= 3.35]   ;

\draw [color={rgb, 255:red, 208; green, 2; blue, 27 }  ,draw opacity=1 ][fill={rgb, 255:red, 208; green, 2; blue, 27 }  ,fill opacity=1 ]   (78,138.97) ;
\draw [shift={(78,138.97)}, rotate = 0] [color={rgb, 255:red, 208; green, 2; blue, 27 }  ,draw opacity=1 ][fill={rgb, 255:red, 208; green, 2; blue, 27 }  ,fill opacity=1 ][line width=0.75]      (0, 0) circle [x radius= 3.35, y radius= 3.35]   ;
\draw [color={rgb, 255:red, 245; green, 166; blue, 35 }  ,draw opacity=1 ][fill={rgb, 255:red, 208; green, 2; blue, 27 }  ,fill opacity=1 ]   (141,123.97) ;
\draw [shift={(141,123.97)}, rotate = 0] [color={rgb, 255:red, 245; green, 166; blue, 35 }  ,draw opacity=1 ][fill={rgb, 255:red, 245; green, 166; blue, 35 }  ,fill opacity=1 ][line width=0.75]      (0, 0) circle [x radius= 3.35, y radius= 3.35]   ;
\draw [color={rgb, 255:red, 245; green, 166; blue, 35 }  ,draw opacity=1 ][fill={rgb, 255:red, 208; green, 2; blue, 27 }  ,fill opacity=1 ]   (143,74.97) ;
\draw [shift={(143,74.97)}, rotate = 0] [color={rgb, 255:red, 245; green, 166; blue, 35 }  ,draw opacity=1 ][fill={rgb, 255:red, 245; green, 166; blue, 35 }  ,fill opacity=1 ][line width=0.75]      (0, 0) circle [x radius= 3.35, y radius= 3.35]   ;

\draw [color={rgb, 255:red, 245; green, 166; blue, 35 }  ,draw opacity=1 ][fill={rgb, 255:red, 208; green, 2; blue, 27 }  ,fill opacity=1 ]   (141,139.97) ;
\draw [shift={(141,139.97)}, rotate = 0] [color={rgb, 255:red, 245; green, 166; blue, 35 }  ,draw opacity=1 ][fill={rgb, 255:red, 245; green, 166; blue, 35 }  ,fill opacity=1 ][line width=0.75]      (0, 0) circle [x radius= 3.35, y radius= 3.35]   ;
\draw [color={rgb, 255:red, 245; green, 166; blue, 35 }  ,draw opacity=1 ][fill={rgb, 255:red, 208; green, 2; blue, 27 }  ,fill opacity=1 ]   (367.8,116.58) ;
\draw [shift={(367.8,116.58)}, rotate = 0] [color={rgb, 255:red, 245; green, 166; blue, 35 }  ,draw opacity=1 ][fill={rgb, 255:red, 245; green, 166; blue, 35 }  ,fill opacity=1 ][line width=0.75]      (0, 0) circle [x radius= 3.35, y radius= 3.35]   ;
\draw [color={rgb, 255:red, 245; green, 166; blue, 35 }  ,draw opacity=1 ][fill={rgb, 255:red, 208; green, 2; blue, 27 }  ,fill opacity=1 ]   (384.2,70.36) ;
\draw [shift={(384.2,70.36)}, rotate = 0] [color={rgb, 255:red, 245; green, 166; blue, 35 }  ,draw opacity=1 ][fill={rgb, 255:red, 245; green, 166; blue, 35 }  ,fill opacity=1 ][line width=0.75]      (0, 0) circle [x radius= 3.35, y radius= 3.35]   ;

\draw [color={rgb, 255:red, 208; green, 2; blue, 27 }  ,draw opacity=1 ][fill={rgb, 255:red, 208; green, 2; blue, 27 }  ,fill opacity=1 ]   (372.63,105.48) ;
\draw [shift={(372.63,105.48)}, rotate = 0] [color={rgb, 255:red, 208; green, 2; blue, 27 }  ,draw opacity=1 ][fill={rgb, 255:red, 208; green, 2; blue, 27 }  ,fill opacity=1 ][line width=0.75]      (0, 0) circle [x radius= 3.35, y radius= 3.35]   ;
\draw [color={rgb, 255:red, 208; green, 2; blue, 27 }  ,draw opacity=1 ][fill={rgb, 255:red, 208; green, 2; blue, 27 }  ,fill opacity=1 ]   (395.28,44.51) ;
\draw [shift={(395.28,44.51)}, rotate = 0] [color={rgb, 255:red, 208; green, 2; blue, 27 }  ,draw opacity=1 ][fill={rgb, 255:red, 208; green, 2; blue, 27 }  ,fill opacity=1 ][line width=0.75]      (0, 0) circle [x radius= 3.35, y radius= 3.35]   ;

\draw    (471,110) .. controls (439.32,140.66) and (418.42,69.42) .. (377.25,112.62) ;
\draw [shift={(376,113.97)}, rotate = 312.4] [color={rgb, 255:red, 0; green, 0; blue, 0 }  ][line width=0.75]    (10.93,-3.29) .. controls (6.95,-1.4) and (3.31,-0.3) .. (0,0) .. controls (3.31,0.3) and (6.95,1.4) .. (10.93,3.29)   ;

\draw (477,60.4) node [anchor=north west][inner sep=0.75pt]  [color={rgb, 255:red, 208; green, 2; blue, 27 }  ,opacity=1 ]  {$( E,\ (\mathcal{J} ,0)) ,\ \mathcal{J} \ \in \ E[ r]$};
\draw (275,133.4) node [anchor=north west][inner sep=0.75pt]    {$Z_{\mathrm{main}}$};
\draw (134,92.4) node [anchor=north west][inner sep=0.75pt]    {$\vdots $};
\draw (370.45,84.35) node [anchor=north west][inner sep=0.75pt]  [rotate=-17.19]  {$\vdots $};
\draw (70,93.4) node [anchor=north west][inner sep=0.75pt]    {$\vdots $};
\draw (378.05,67.23) node [anchor=north west][inner sep=0.75pt]  [rotate=-17.9]  {$\vdots $};
\draw (479,101.4) node [anchor=north west][inner sep=0.75pt]  [color={rgb, 255:red, 245; green, 166; blue, 35 }  ,opacity=1 ]  {$( E,\ (\mathcal{J} ',0)) ,\ \mathcal{J} '\ \in \ E[ s]$};

\end{tikzpicture} 
    \end{center}

    By Corollary \ref{g = 1 virtual class}, we see the virtual class on the $B\mu_r,B \mu_s$ components is just the fundamental class. On the main component, the contribution is given by the degree of the obstruction bundle $$\mathrm{deg} \mathcal{O}( 0/r + \infty/s) = 1/r + 1/s.$$ Therefore the total degree of the virtual class equals $$\mathrm{deg}[\mathrm{Orb}(X_{r,s})]^{\mathrm{vir}} = 1/r + 1/s + (r^2 - 1)/r + (s^2 - 1)/s = r + s$$ which is monomial in the rooting parameters.
\end{example} \qed

\begin{example}
    Let $X_{r,s}$ be as in the example above, but now take the curve class $\beta = E + S$ where $S$ is a section, $n = 2$ with contact order $1$ with each of $E_0, E_{\infty}$ (i.e. take the components of the inertia stack of $X_{r,s}$ that are age $\frac{1}{r}, \frac{1}{s}$ induced by $E_0, E_{\infty}$ respectively). There are three distinguished loci $W_1, W_2, W_3 \subset \mathrm{Orb}_{\Lambda}(X_{r,s})$ whose closures union the whole space, whose generic members we depict below: \begin{center}

\tikzset{every picture/.style={line width=0.75pt}} 

\begin{tikzpicture}[x=0.75pt,y=0.75pt,yscale=-1,xscale=1]

\draw   (322.18,10.23) .. controls (338.39,10.23) and (351.55,33.22) .. (351.56,61.59) .. controls (351.58,89.96) and (338.45,112.96) .. (322.24,112.97) .. controls (306.03,112.98) and (292.87,89.99) .. (292.86,61.62) .. controls (292.84,33.25) and (305.97,10.24) .. (322.18,10.23) -- cycle ;
\draw  [draw opacity=0] (325.77,83.41) .. controls (318.18,79.86) and (312.92,72.15) .. (312.92,63.22) .. controls (312.92,54.49) and (317.94,46.94) .. (325.25,43.28) -- (335.2,63.22) -- cycle ; \draw   (325.77,83.41) .. controls (318.18,79.86) and (312.92,72.15) .. (312.92,63.22) .. controls (312.92,54.49) and (317.94,46.94) .. (325.25,43.28) ;  
\draw  [draw opacity=0] (320.43,49.01) .. controls (323.74,53) and (325.67,58.15) .. (325.54,63.73) .. controls (325.42,69.03) and (323.46,73.85) .. (320.29,77.6) -- (303.27,63.22) -- cycle ; \draw   (320.43,49.01) .. controls (323.74,53) and (325.67,58.15) .. (325.54,63.73) .. controls (325.42,69.03) and (323.46,73.85) .. (320.29,77.6) ;  
\draw    (240.14,9) -- (415.42,9.72) ;
\draw    (240.14,9) ;
\draw [shift={(240.14,9)}, rotate = 0] [color={rgb, 255:red, 0; green, 0; blue, 0 }  ][fill={rgb, 255:red, 0; green, 0; blue, 0 }  ][line width=0.75]      (0, 0) circle [x radius= 3.35, y radius= 3.35]   ;
\draw    (415.42,9.72) ;
\draw [shift={(415.42,9.72)}, rotate = 0] [color={rgb, 255:red, 0; green, 0; blue, 0 }  ][fill={rgb, 255:red, 0; green, 0; blue, 0 }  ][line width=0.75]      (0, 0) circle [x radius= 3.35, y radius= 3.35]   ;
\draw    (240.14,9) .. controls (209.42,36.04) and (276.69,38.59) .. (245.35,11) ;
\draw [shift={(243.85,9.72)}, rotate = 39.66] [color={rgb, 255:red, 0; green, 0; blue, 0 }  ][line width=0.75]    (10.93,-3.29) .. controls (6.95,-1.4) and (3.31,-0.3) .. (0,0) .. controls (3.31,0.3) and (6.95,1.4) .. (10.93,3.29)   ;
\draw    (415.42,9.72) .. controls (384.69,36.76) and (451.96,39.31) .. (420.63,11.72) ;
\draw [shift={(419.13,10.44)}, rotate = 39.66] [color={rgb, 255:red, 0; green, 0; blue, 0 }  ][line width=0.75]    (10.93,-3.29) .. controls (6.95,-1.4) and (3.31,-0.3) .. (0,0) .. controls (3.31,0.3) and (6.95,1.4) .. (10.93,3.29)   ;

\draw   (533.08,159.66) .. controls (548.15,159.65) and (560.37,181.02) .. (560.39,207.39) .. controls (560.4,233.76) and (548.2,255.14) .. (533.13,255.15) .. controls (518.06,255.16) and (505.83,233.79) .. (505.82,207.42) .. controls (505.81,181.05) and (518.01,159.67) .. (533.08,159.66) -- cycle ;
\draw  [draw opacity=0] (536.41,227.67) .. controls (529.36,224.37) and (524.47,217.21) .. (524.47,208.91) .. controls (524.47,200.8) and (529.13,193.77) .. (535.93,190.38) -- (545.18,208.91) -- cycle ; \draw   (536.41,227.67) .. controls (529.36,224.37) and (524.47,217.21) .. (524.47,208.91) .. controls (524.47,200.8) and (529.13,193.77) .. (535.93,190.38) ;  
\draw  [draw opacity=0] (531.45,195.7) .. controls (534.52,199.41) and (536.32,204.2) .. (536.2,209.38) .. controls (536.09,214.31) and (534.27,218.79) .. (531.32,222.27) -- (515.5,208.91) -- cycle ; \draw   (531.45,195.7) .. controls (534.52,199.41) and (536.32,204.2) .. (536.2,209.38) .. controls (536.09,214.31) and (534.27,218.79) .. (531.32,222.27) ;  

\draw    (370.3,158.9) -- (533.21,159.56) ;
\draw    (370.3,158.9) ;
\draw [shift={(370.3,158.9)}, rotate = 0] [color={rgb, 255:red, 0; green, 0; blue, 0 }  ][fill={rgb, 255:red, 0; green, 0; blue, 0 }  ][line width=0.75]      (0, 0) circle [x radius= 3.35, y radius= 3.35]   ;
\draw    (370.3,158.9) .. controls (352.34,179.41) and (385.16,181.82) .. (374.65,161.2) ;
\draw [shift={(373.76,159.56)}, rotate = 59.6] [color={rgb, 255:red, 0; green, 0; blue, 0 }  ][line width=0.75]    (10.93,-3.29) .. controls (6.95,-1.4) and (3.31,-0.3) .. (0,0) .. controls (3.31,0.3) and (6.95,1.4) .. (10.93,3.29)   ;
\draw    (545.18,208.91) .. controls (516.62,234.04) and (579.15,236.41) .. (550.03,210.77) ;
\draw [shift={(548.63,209.58)}, rotate = 39.66] [color={rgb, 255:red, 0; green, 0; blue, 0 }  ][line width=0.75]    (10.93,-3.29) .. controls (6.95,-1.4) and (3.31,-0.3) .. (0,0) .. controls (3.31,0.3) and (6.95,1.4) .. (10.93,3.29)   ;
\draw    (549.21,201.56) ;
\draw [shift={(549.21,201.56)}, rotate = 0] [color={rgb, 255:red, 0; green, 0; blue, 0 }  ][fill={rgb, 255:red, 0; green, 0; blue, 0 }  ][line width=0.75]      (0, 0) circle [x radius= 3.35, y radius= 3.35]   ;
\draw   (122.08,158.66) .. controls (137.15,158.65) and (149.37,180.02) .. (149.39,206.39) .. controls (149.4,232.76) and (137.2,254.14) .. (122.13,254.15) .. controls (107.06,254.16) and (94.83,232.79) .. (94.82,206.42) .. controls (94.81,180.05) and (107.01,158.67) .. (122.08,158.66) -- cycle ;
\draw  [draw opacity=0] (125.41,226.67) .. controls (118.36,223.37) and (113.47,216.21) .. (113.47,207.91) .. controls (113.47,199.8) and (118.13,192.77) .. (124.93,189.38) -- (134.18,207.91) -- cycle ; \draw   (125.41,226.67) .. controls (118.36,223.37) and (113.47,216.21) .. (113.47,207.91) .. controls (113.47,199.8) and (118.13,192.77) .. (124.93,189.38) ;  
\draw  [draw opacity=0] (120.45,194.7) .. controls (123.52,198.41) and (125.32,203.2) .. (125.2,208.38) .. controls (125.09,213.31) and (123.27,217.79) .. (120.32,221.27) -- (104.5,207.91) -- cycle ; \draw   (120.45,194.7) .. controls (123.52,198.41) and (125.32,203.2) .. (125.2,208.38) .. controls (125.09,213.31) and (123.27,217.79) .. (120.32,221.27) ;  

\draw    (118.3,158.9) -- (281.21,159.56) ;
\draw    (105.88,199.6) ;
\draw [shift={(105.88,199.6)}, rotate = 0] [color={rgb, 255:red, 0; green, 0; blue, 0 }  ][fill={rgb, 255:red, 0; green, 0; blue, 0 }  ][line width=0.75]      (0, 0) circle [x radius= 3.35, y radius= 3.35]   ;
\draw    (102.43,198.93) .. controls (84.46,219.44) and (117.28,221.86) .. (106.77,201.23) ;
\draw [shift={(105.88,199.6)}, rotate = 59.6] [color={rgb, 255:red, 0; green, 0; blue, 0 }  ][line width=0.75]    (10.93,-3.29) .. controls (6.95,-1.4) and (3.31,-0.3) .. (0,0) .. controls (3.31,0.3) and (6.95,1.4) .. (10.93,3.29)   ;
\draw    (281.21,159.56) .. controls (252.66,184.7) and (315.19,187.07) .. (286.06,161.42) ;
\draw [shift={(284.67,160.23)}, rotate = 39.66] [color={rgb, 255:red, 0; green, 0; blue, 0 }  ][line width=0.75]    (10.93,-3.29) .. controls (6.95,-1.4) and (3.31,-0.3) .. (0,0) .. controls (3.31,0.3) and (6.95,1.4) .. (10.93,3.29)   ;
\draw    (281.21,159.56) ;
\draw [shift={(281.21,159.56)}, rotate = 0] [color={rgb, 255:red, 0; green, 0; blue, 0 }  ][fill={rgb, 255:red, 0; green, 0; blue, 0 }  ][line width=0.75]      (0, 0) circle [x radius= 3.35, y radius= 3.35]   ;

\draw (231.75,27.14) node [anchor=north west][inner sep=0.75pt]    {$\mu _{r}$};
\draw (407.03,30.11) node [anchor=north west][inner sep=0.75pt]    {$\mu _{s}$};
\draw (362.48,185.18) node [anchor=north west][inner sep=0.75pt]    {$\mu _{r}$};
\draw (533.32,225.67) node [anchor=north west][inner sep=0.75pt]    {$\mu _{s}$};
\draw (99.48,215.18) node [anchor=north west][inner sep=0.75pt]    {$\mu _{r}$};
\draw (274.82,187.57) node [anchor=north west][inner sep=0.75pt]    {$\mu _{s}$};
\draw (312,123.4) node [anchor=north west][inner sep=0.75pt]    {$W_{1}$};
\draw (192,246.4) node [anchor=north west][inner sep=0.75pt]    {$W_{2}$};
\draw (430,247.4) node [anchor=north west][inner sep=0.75pt]    {$W_{3}$};

\end{tikzpicture}
    \end{center} In $W_1$ the elliptic curve maps to a generic fibre away from the stacky loci, and the other component depicts a section of the elliptic fibration. In $W_2$ the genus $1$ component maps to $\frac{1}{r} E_0$ and in $W_3$ the genus $1$ component maps to $\frac{1}{s} E_{\infty}$. Then by applying the mod $r$ balancing and coprimality conditions we see the node in picture $W_1$ is non-stacky and the nodes in $W_2,W_3$ have isotropy $\mu_r$ and $\mu_s$ respectively.

    Observe that $$W_1 \cong X_{r,s}, \  W_2 \cong \coprod_{r^2} (E \times B \mu_r) \times E, \ W_3 \cong \coprod_{s^2} (E \times B \mu_s) \times E.$$
    
    \begin{center}

\tikzset{every picture/.style={line width=0.75pt}} 

\begin{tikzpicture}[x=0.75pt,y=0.75pt,yscale=-1,xscale=1]

\draw  [fill={rgb, 255:red, 181; green, 255; blue, 106 }  ,fill opacity=1 ] (281.56,131.72) -- (317.86,106.38) -- (308.01,163.81) -- (271.7,189.14) -- cycle ;
\draw  [fill={rgb, 255:red, 181; green, 255; blue, 106 }  ,fill opacity=1 ] (307.56,142.72) -- (343.86,117.38) -- (334.01,174.81) -- (297.7,200.14) -- cycle ;
\draw  [fill={rgb, 255:red, 80; green, 227; blue, 194 }  ,fill opacity=1 ] (335.08,174.76) -- (378.3,197.86) -- (340.92,223.24) -- (297.7,200.14) -- cycle ;
\draw  [fill={rgb, 255:red, 248; green, 231; blue, 28 }  ,fill opacity=1 ] (351.85,165.76) -- (388.15,140.43) -- (378.3,197.86) -- (341.99,223.19) -- cycle ;
\draw  [fill={rgb, 255:red, 248; green, 231; blue, 28 }  ,fill opacity=1 ] (378.85,180.76) -- (415.15,155.43) -- (405.3,212.86) -- (368.99,238.19) -- cycle ;
\draw  [fill={rgb, 255:red, 248; green, 231; blue, 28 }  ,fill opacity=1 ] (421.85,212.76) -- (458.15,187.43) -- (448.3,244.86) -- (411.99,270.19) -- cycle ;
\draw  [fill={rgb, 255:red, 181; green, 255; blue, 106 }  ,fill opacity=1 ] (236.56,106.72) -- (272.86,81.38) -- (263.01,138.81) -- (226.7,164.14) -- cycle ;
\draw    (194,175) -- (262,223.97) ;
\draw    (194,175) -- (203,165.97) ;
\draw    (262,223.97) -- (271,214.94) ;
\draw    (335,242) -- (403,290.97) ;
\draw    (335,242) -- (344,232.97) ;
\draw    (403,290.97) -- (412,281.94) ;
\draw    (360,275.52) -- (369,266.48) ;
\draw    (219,208.52) -- (228,199.48) ;
\draw    (242.78,131.79) .. controls (224.07,116) and (277.15,105.09) .. (253.9,130.57) ;
\draw [shift={(252.78,131.76)}, rotate = 313.92] [color={rgb, 255:red, 0; green, 0; blue, 0 }  ][line width=0.75]    (10.93,-3.29) .. controls (6.95,-1.4) and (3.31,-0.3) .. (0,0) .. controls (3.31,0.3) and (6.95,1.4) .. (10.93,3.29)   ;
\draw    (312.78,169.79) .. controls (294.07,154) and (347.15,143.09) .. (323.9,168.57) ;
\draw [shift={(322.78,169.76)}, rotate = 313.92] [color={rgb, 255:red, 0; green, 0; blue, 0 }  ][line width=0.75]    (10.93,-3.29) .. controls (6.95,-1.4) and (3.31,-0.3) .. (0,0) .. controls (3.31,0.3) and (6.95,1.4) .. (10.93,3.29)   ;
\draw    (355.78,186.79) .. controls (337.07,171) and (390.15,160.09) .. (366.9,185.57) ;
\draw [shift={(365.78,186.76)}, rotate = 313.92] [color={rgb, 255:red, 0; green, 0; blue, 0 }  ][line width=0.75]    (10.93,-3.29) .. controls (6.95,-1.4) and (3.31,-0.3) .. (0,0) .. controls (3.31,0.3) and (6.95,1.4) .. (10.93,3.29)   ;
\draw    (426.78,236.79) .. controls (408.07,221) and (461.15,210.09) .. (437.9,235.57) ;
\draw [shift={(436.78,236.76)}, rotate = 313.92] [color={rgb, 255:red, 0; green, 0; blue, 0 }  ][line width=0.75]    (10.93,-3.29) .. controls (6.95,-1.4) and (3.31,-0.3) .. (0,0) .. controls (3.31,0.3) and (6.95,1.4) .. (10.93,3.29)   ;

\draw (395,222.4) node [anchor=north west][inner sep=0.75pt]    {$\ddots $};
\draw (248,144.4) node [anchor=north west][inner sep=0.75pt]    {$\ddots $};
\draw (319,187.4) node [anchor=north west][inner sep=0.75pt]    {$W_{1}$};
\draw (198,211.4) node [anchor=north west][inner sep=0.75pt]    {$W_{2}$};
\draw (344,276.4) node [anchor=north west][inner sep=0.75pt]    {$W_{3}$};
\draw (248,95.4) node [anchor=north west][inner sep=0.75pt]    {$\mu _{r}$};
\draw (321,133.4) node [anchor=north west][inner sep=0.75pt]    {$\mu _{r}$};
\draw (365,150.4) node [anchor=north west][inner sep=0.75pt]    {$\mu _{s}$};
\draw (435,199.4) node [anchor=north west][inner sep=0.75pt]    {$\mu _{s}$};
\draw (309,326.4) node [anchor=north west][inner sep=0.75pt]    {$\mathrm{Orb}_{\Lambda }( X_{r,s})$};

\end{tikzpicture}
    \end{center}

    To see this, a map of type $W_1$ is determined precisely by the choice of elliptic fibre and the choice of section. When the fibre maps to $\frac{1}{r} E_0$ or $\frac{1}{s} E_{\infty}$ there are $\mu_r, \mu_s$ automorphisms yielding $X_{r,s}$. As for $W_2$, the moduli we have involves the choice of marking on $E$, the choice of principal $\mu_r$-bundle $E \rightarrow B \mu_r$ and the choice of section. The choice of marking on the stacky curve is parameterised by the first factor $E \times B \mu_r$ (as the stacky divisor is the trivial gerbe). The choice of section is parameterised by the second factor, $E$ which is encoding the intersection point with one of the vertical divisors. This factor is un-stacky as the intersection is transverse. A similar description exists for $W_3$, just now the marked point lies on the other stacky divisor so we have replaced $r$ with $s$. Note that the trivial torsors are no longer contained in the first irreducible component, unlike in Example \ref{Maulik}; the intersections of $W_1$ with $W_2, W_3$ correspond to maps from curves with $3$ irreducible components of the following form, with a contracted rational component: \begin{center}

\tikzset{every picture/.style={line width=0.75pt}} 

\begin{tikzpicture}[x=0.75pt,y=0.75pt,yscale=-1,xscale=1]

\draw   (314.33,187.04) .. controls (325.36,186.51) and (335.05,201.74) .. (335.98,221.04) .. controls (336.9,240.35) and (328.71,256.43) .. (317.67,256.96) .. controls (306.64,257.49) and (296.95,242.26) .. (296.02,222.96) .. controls (295.1,203.65) and (303.29,187.57) .. (314.33,187.04) -- cycle ;
\draw  [draw opacity=0] (318.05,241.39) .. controls (313.64,236.17) and (310.98,229.42) .. (310.98,222.04) .. controls (310.98,214.92) and (313.46,208.38) .. (317.6,203.24) -- (340.98,222.04) -- cycle ; \draw   (318.05,241.39) .. controls (313.64,236.17) and (310.98,229.42) .. (310.98,222.04) .. controls (310.98,214.92) and (313.46,208.38) .. (317.6,203.24) ;  
\draw  [draw opacity=0] (313.6,214.04) .. controls (314.53,216.98) and (315.01,220.12) .. (314.98,223.37) .. controls (314.94,226.19) and (314.53,228.91) .. (313.77,231.49) -- (284.98,223.04) -- cycle ; \draw   (313.6,214.04) .. controls (314.53,216.98) and (315.01,220.12) .. (314.98,223.37) .. controls (314.94,226.19) and (314.53,228.91) .. (313.77,231.49) ;  
\draw    (297,85.97) -- (295,256.97) ;
\draw    (281,104) -- (391,103.47) ;
\draw    (296,134.47) ;
\draw [shift={(296,134.47)}, rotate = 0] [color={rgb, 255:red, 0; green, 0; blue, 0 }  ][fill={rgb, 255:red, 0; green, 0; blue, 0 }  ][line width=0.75]      (0, 0) circle [x radius= 3.35, y radius= 3.35]   ;
\draw    (363,103.47) ;
\draw [shift={(363,103.47)}, rotate = 0] [color={rgb, 255:red, 0; green, 0; blue, 0 }  ][fill={rgb, 255:red, 0; green, 0; blue, 0 }  ][line width=0.75]      (0, 0) circle [x radius= 3.35, y radius= 3.35]   ;

\draw (222,146.4) node [anchor=north west][inner sep=0.75pt]    {$\mathrm{deg} \ =\ 0$};

\end{tikzpicture}
    \end{center} These intersection loci are isomorphic to $E \times B \mu_r, E \times B \mu_s$ respectively, parameterising the section's height.
    
     In particular $\mathrm{Orb}_{\Lambda}(X_{r,s})$ is equidimensional of dimension $2$ with precisely $r^2 + s^2 + 1$ irreducible components and one easily computes this is also the virtual dimension. Under the map $\mathrm{Orb}_{\Lambda}(X_{r,s}) \rightarrow \mathrm{Orb}_{\Lambda}(\mathcal{A}_l)$ we draw the mod $l$-tropical types that $W_1,W_2,W_3$ map to, where $W_1$ maps to $[\tau_1]$ and $W_2,W_3$ map to the two types of the form $[\tau_2]$, which differ in which of the two markings lies on the internal vertex: \begin{center}

\tikzset{every picture/.style={line width=0.75pt}} 

\begin{tikzpicture}[x=0.75pt,y=0.75pt,yscale=-1,xscale=1]

\draw    (201,73.97) -- (200,144.97) ;
\draw [shift={(200,144.97)}, rotate = 90.81] [color={rgb, 255:red, 0; green, 0; blue, 0 }  ][fill={rgb, 255:red, 0; green, 0; blue, 0 }  ][line width=0.75]      (0, 0) circle [x radius= 3.35, y radius= 3.35]   ;
\draw [shift={(201,73.97)}, rotate = 90.81] [color={rgb, 255:red, 0; green, 0; blue, 0 }  ][fill={rgb, 255:red, 0; green, 0; blue, 0 }  ][line width=0.75]      (0, 0) circle [x radius= 3.35, y radius= 3.35]   ;
\draw    (200,210) -- (341,211.94) ;
\draw [shift={(343,211.97)}, rotate = 180.79] [color={rgb, 255:red, 0; green, 0; blue, 0 }  ][line width=0.75]    (10.93,-3.29) .. controls (6.95,-1.4) and (3.31,-0.3) .. (0,0) .. controls (3.31,0.3) and (6.95,1.4) .. (10.93,3.29)   ;
\draw [shift={(200,210)}, rotate = 0.79] [color={rgb, 255:red, 0; green, 0; blue, 0 }  ][fill={rgb, 255:red, 0; green, 0; blue, 0 }  ][line width=0.75]      (0, 0) circle [x radius= 3.35, y radius= 3.35]   ;
\draw    (200,162.97) -- (199.06,194.97) ;
\draw [shift={(199,196.97)}, rotate = 271.68] [color={rgb, 255:red, 0; green, 0; blue, 0 }  ][line width=0.75]    (10.93,-3.29) .. controls (6.95,-1.4) and (3.31,-0.3) .. (0,0) .. controls (3.31,0.3) and (6.95,1.4) .. (10.93,3.29)   ;
\draw    (382,135.97) -- (452,136.97) ;
\draw [shift={(452,136.97)}, rotate = 0.82] [color={rgb, 255:red, 0; green, 0; blue, 0 }  ][fill={rgb, 255:red, 0; green, 0; blue, 0 }  ][line width=0.75]      (0, 0) circle [x radius= 3.35, y radius= 3.35]   ;
\draw [shift={(382,135.97)}, rotate = 0.82] [color={rgb, 255:red, 0; green, 0; blue, 0 }  ][fill={rgb, 255:red, 0; green, 0; blue, 0 }  ][line width=0.75]      (0, 0) circle [x radius= 3.35, y radius= 3.35]   ;
\draw    (382,211) -- (523,212.94) ;
\draw [shift={(525,212.97)}, rotate = 180.79] [color={rgb, 255:red, 0; green, 0; blue, 0 }  ][line width=0.75]    (10.93,-3.29) .. controls (6.95,-1.4) and (3.31,-0.3) .. (0,0) .. controls (3.31,0.3) and (6.95,1.4) .. (10.93,3.29)   ;
\draw [shift={(382,211)}, rotate = 0.79] [color={rgb, 255:red, 0; green, 0; blue, 0 }  ][fill={rgb, 255:red, 0; green, 0; blue, 0 }  ][line width=0.75]      (0, 0) circle [x radius= 3.35, y radius= 3.35]   ;
\draw    (382,163.97) -- (381.06,195.97) ;
\draw [shift={(381,197.97)}, rotate = 271.68] [color={rgb, 255:red, 0; green, 0; blue, 0 }  ][line width=0.75]    (10.93,-3.29) .. controls (6.95,-1.4) and (3.31,-0.3) .. (0,0) .. controls (3.31,0.3) and (6.95,1.4) .. (10.93,3.29)   ;
\draw    (201,73.97) -- (274.03,87.6) ;
\draw [shift={(276,87.97)}, rotate = 190.57] [color={rgb, 255:red, 0; green, 0; blue, 0 }  ][line width=0.75]    (10.93,-3.29) .. controls (6.95,-1.4) and (3.31,-0.3) .. (0,0) .. controls (3.31,0.3) and (6.95,1.4) .. (10.93,3.29)   ;
\draw    (201,73.97) -- (275.07,54.48) ;
\draw [shift={(277,53.97)}, rotate = 165.26] [color={rgb, 255:red, 0; green, 0; blue, 0 }  ][line width=0.75]    (10.93,-3.29) .. controls (6.95,-1.4) and (3.31,-0.3) .. (0,0) .. controls (3.31,0.3) and (6.95,1.4) .. (10.93,3.29)   ;
\draw    (452,136.97) -- (524,136.97) ;
\draw [shift={(526,136.97)}, rotate = 180] [color={rgb, 255:red, 0; green, 0; blue, 0 }  ][line width=0.75]    (10.93,-3.29) .. controls (6.95,-1.4) and (3.31,-0.3) .. (0,0) .. controls (3.31,0.3) and (6.95,1.4) .. (10.93,3.29)   ;
\draw    (382,135.97) -- (433.2,110.85) ;
\draw [shift={(435,109.97)}, rotate = 153.87] [color={rgb, 255:red, 0; green, 0; blue, 0 }  ][line width=0.75]    (10.93,-3.29) .. controls (6.95,-1.4) and (3.31,-0.3) .. (0,0) .. controls (3.31,0.3) and (6.95,1.4) .. (10.93,3.29)   ;

\draw (135,64.4) node [anchor=north west][inner sep=0.75pt]    {$g_{v_{1}} =\ 0$};
\draw (134,136.4) node [anchor=north west][inner sep=0.75pt]    {$g_{v_{2}} =\ 1$};
\draw (449,146.4) node [anchor=north west][inner sep=0.75pt]    {$g_{v_{2}} =\ 1$};
\draw (331,114.4) node [anchor=north west][inner sep=0.75pt]    {$g_{v_{1}} =\ 0$};
\draw (230,236.4) node [anchor=north west][inner sep=0.75pt]    {$[ \tau _{1}]$};
\draw (431,231.4) node [anchor=north west][inner sep=0.75pt]    {$[ \tau _{2}]$};

\end{tikzpicture}
    \end{center} Stratum $Z_{[\tau_1]}$ is in the boundary of the main component $Z_{\mathrm{main}}$ and $Z_{[\tau_2]}$ is another irreducible component.

The virtual class agrees with the fundamental class as before. Let $\mathrm{st}: \mathrm{Orb}_{\Lambda}(X_{r,s}) \rightarrow \overline{M}_{1,2}$ denote the stabilisation map, and note that $$\mathrm{st}_* [\mathrm{Orb}_{\Lambda}(X_{r,s})]^{\mathrm{vir}} = 0$$ since $\mathrm{dim}(\overline{M}_{1,2}) = 2$ and on each irreducible component of $\mathrm{Orb}_{\Lambda}(X_{r,s})$ the complex structure of the elliptic curve is fixed, so $\mathrm{st}$ must lower dimension on all components. To get an interesting polynomial, one needs to consider appropriate insertions. Indeed, consider $\mathrm{ev}_1^{*}[\mathrm{pt}]$, the pullback of an age $\frac{1}{r}$ point class on $\frac{1}{r} E_0$. On each irreducible component of $W_2$, the intersection with this locus is isomorphic to $E \times B\mu_r$ corresponding to the moduli of the joining node. The intersection of this locus with $W_3$ fixes the section, and hence get loci isomorphic to $E \times B \mu_s$ corresponding to the moduli of the second marking on $E$. The intersection of this locus with $W_1$ fixes the section and gives $\mathbb{P}^1_{r,s}$ corresponding to choice of elliptic fibre. The map $\mathrm{st}$ contracts the unstable rational tail in each of these loci and we see $$\mathrm{st}_* (\mathrm{ev}_1^{*}[\mathrm{pt}] \cap [\mathrm{Orb}_{\Lambda}(X_{r,s})]^{\mathrm{vir}}) = (r + s)[(E,p)]$$ where $(E,p) \subset \overline{M}_{1,2} $ corresponds to the elliptic fibre over $\overline{M}_{1,1}$ corresponding to $E$ and a choice of $p \in E$.

\end{example} \qed


\newpage

\bibliography{ArXivV1main}
\bibliographystyle{plain}

\end{document}